\documentclass[a4paper, 11pt, reqno]{article}

\usepackage[T1]{fontenc}
\usepackage[utf8]{inputenc}
\usepackage{lmodern}
\usepackage{lmodern}

\usepackage[scale=0.83,centering]{geometry}
\usepackage{amsmath,amsfonts, mathabx,amssymb,amsthm,graphicx,bbm}
\usepackage{color,mathrsfs,tikz,hyperref}
\usepackage[font={footnotesize}]{caption}

\numberwithin{equation}{section}

\definecolor{darkorange}{rgb}{1.0, 0.55, 0.0}
\definecolor{brinkpink}{rgb}{0.98, 0.38, 0.5}
\definecolor{verde}{RGB}{0,140,69}

\newcommand{\bbD}{{\ensuremath{\mathbbm D}} }
\newcommand{\bbE}{{\ensuremath{\mathbbm E}} }

\newcommand{\bbN}{{\ensuremath{\mathbbm N}} }

\newcommand{\bbP}{{\ensuremath{\mathbbm P}} }
\newcommand{\bbQ}{{\ensuremath{\mathbbm Q}} }
\newcommand{\bbR}{{\ensuremath{\mathbbm R}} }

\newcommand{\bbZ}{{\ensuremath{\mathbbm Z}} }

\newcommand{\cA}{{\ensuremath{\mathcal A}} }
\newcommand{\cB}{{\ensuremath{\mathcal B}} }
\newcommand{\cC}{{\ensuremath{\mathcal C}} }

\newcommand{\cI}{{\ensuremath{\mathcal I}} }

\newcommand{\cR}{{\ensuremath{\mathcal R}} }

\newcommand{\cU}{{\ensuremath{\mathcal U}} }
\newcommand{\cV}{{\ensuremath{\mathcal V}} }
\newcommand{\cW}{{\ensuremath{\mathcal W}} }

\newcommand{\bE}{{\ensuremath{\mathbf E}} }

\newcommand{\bP}{{\ensuremath{\mathbf P}} }

\newcommand{\bT}{{\ensuremath{\mathbf T}} }

\newcommand{\bV}{{\ensuremath{\mathbf V}} }

\newcommand{\T}{\mathsf{T}}

\newcommand{\ga}{\alpha}
\newcommand{\gb}{\beta}
\newcommand{\gd}{\delta}
\newcommand{\gep}{\varepsilon}       

\newcommand{\gz}{\zeta}

\newcommand{\go}{\omega}

\newcommand{\suptwo}[2]{\sup\limits_{\substack{#1 \\ #2}}}
\newcommand{\inftwo}[2]{\sup\limits_{\substack{#1 \\ #2}}}
\renewcommand{\tilde}{\widetilde}
\newcommand{\what}{\widehat}
\newcommand{\wcheck}{\widecheck}
\newcommand{\ind}{\mathbbm{1}}
\newcommand{\dd}{{\ensuremath{\mathrm d}} }

\DeclareMathOperator*{\argmax}{arg\,max}

\newtheorem{assump}{Assumption}
\newtheorem{notation}{Notation}
\newtheorem{theorem}{Theorem}[section]
\newtheorem{definition}{Definition}[section]
\newtheorem{lemma}[theorem]{Lemma}

\newtheorem{proposition}[theorem]{Proposition}
\newtheorem{remark}{Remark}[section]

\renewcommand{\epsilon}{\varepsilon}

\title{One-dimensional polymers in random environments:\\
stretching vs. folding}

\author{Quentin~Berger
\footnote{Sorbonne Universit\'e, LPSM,
Campus Pierre et Marie Curie, case 158,
4 place Jussieu, 75252 Paris Cedex~5, France,
\emph{quentin.berger@sorbonne-universite.fr}}
\and
Chien-Hao~Huang
\footnote{Department of Mathematical Sciences, National Chengchi University, Taipei 16302, Taiwan,
\emph{haohuang@nccu.edu.tw}}
\and
Niccol\`o~Torri
\footnote{Universit\'e Paris-Nanterre, Laboratoire MODAL'X, UMR CNRS 9023 and FP2M, CNRS FR 2036, France,
\emph{niccolo.torri@parisnanterre.fr}}
\and
Ran~Wei
\footnote{Department of Mathematics, Nanjing University, 22 Hankou Road, Nanjing 210093, China,
\emph{weiran@nju.edu.cn}}
}

\date{}

\begin{document}

\maketitle

\begin{abstract}
In this article we study a \emph{non-directed polymer model} on $\mathbb Z$, that is a one-dimensional simple random walk placed in a random environment.
More precisely, 
the law of the random walk is
modified by the exponential of the sum of
potentials $\beta \omega_x -h$ sitting on the range of the random walk, 
where $(\omega_x)_{x\in \mathbb Z}$
are i.i.d.\ random variables (the disorder) and  $\beta\geq 0$ (disorder strength) and $h\in \mathbb{R}$ (external field) are two parameters.
When $\beta=0,h>0$, this corresponds to 
a random walk penalized by its range;
when $\beta>0, h=0$, this corresponds to the  ``standard'' polymer model in random environment, except that it is non-directed. 
In this work, we allow the parameters 
$\beta,h$ to vary according to the length of the random walk and we study in detail
the competition between the \emph{stretching effect} of the disorder, the \emph{folding effect} of the external field (if $h\ge 0$) and the \emph{entropy cost} of atypical trajectories.
We prove a complete description of the (rich) phase diagram
and we identify scaling limits of the model in the different phases.
In particular, in the case $\beta>0, h=0$
of the non-directed polymer, if $\go_x$ has a finite second moment we find a range size fluctuation exponent $\xi=2/3$.

\medskip



\end{abstract}

\section{Introduction}
 We consider here a simple symmetric random walk on
 $\mathbb{Z}^d$, $d\geq 1$,
placed in a time-independent random environment, see \cite{H19}. The interaction with the environment 
occurs on the range of the random walk,
\textit{i.e.}\ on the sites visited by the walk.
This model may be seen as a disordered version of random walks penalized by their range (in the spirit of \cite{Bolt94,DV79}). One closely related model is the celebrated directed polymer in random environment model (see~\cite{C17} for a review),
which has attracted interest from both the mathematical and physics communities over the last forty years,
and can be used to describe a polymer chain placed in a solvent with impurities.

\subsection{The model}  
Let $S:=(S_{n})_{n\geq0}$ be a simple symmetric random walk on $\mathbbm{Z}^{d}$, $d\geq 1$, starting from $0$, whose trajectory represents a (non-directed) polymer.
Let $\bP$ denote its law.
The \textit{random environment}, or \emph{disorder}, is modeled by a field $\omega:=(\omega_{x})_{x\in\mathbbm{Z}^{d}}$ 
of i.i.d.\ random variables. 
We let 
$\bbP$ denote the law of $\omega$, and
$\mathbbm{E}$ the expectation
with respect to $\bbP$ (assumptions on the law of $\go$ are detailed in Section~\ref{sec:setting} below). 

For $\beta\geq 0$ (the disorder strength, or  inverse temperature) and $h \in \bbR$ (an external field),
we define for all $N\in \bbN$ the following Gibbs transformation of the law $\bP$, called the \emph{polymer measure}:
\begin{equation}\label{E101}
\frac{\dd\mathbf{P}_{N,\beta,h}^{\omega}}{\dd\mathbf{P}}(S):=\frac{1}{Z_{N,\beta,h}^{\omega}}\exp\bigg(\sum\limits_{x\in\mathbbm{Z}^{d}}(\beta\omega_{x}-h)\mathbbm{1}_{\{x\in\mathcal{R}_{N}\}}\bigg),
\end{equation}
where  $\mathcal{R}_{N}=\{S_{0}, S_1, \ldots, S_N\}$
is the range of the random walk up to time $N$, and
\begin{equation*}\label{E103}
Z_{N,\beta,h}^{\omega}:=\mathbf{E}\bigg[\exp\bigg(\sum\limits_{x\in\mathbbm{Z}^{d}}(\beta\omega_{x}-h)\mathbbm{1}_{\{x\in\mathcal{R}_{N}\}}\bigg)\bigg] = \mathbf{E}\bigg[\exp\bigg(\beta\sum\limits_{x\in\cR_N}\omega_{x}  -  h |\cR_N|\bigg)\bigg]
\end{equation*}
is the partition function of the model, defined so that $\mathbf{P}_{N,\beta,h}^{\omega}$ is a probability measure.

Let us stress the main differences with the standard directed polymer model: 
(i)~here, the random walk does not have a preferred direction;
(ii)~there is an additional external field $h\in \bbR$;
(iii)~the random walk can only pick up one weight $\beta\omega_{x}-h$ at a site $x\in \bbZ^d$, so returning to an already visited site
does not bring any reward or penalty
(in the directed polymer model, the environment is renewed each time).

\smallskip

We now wish to understand the typical behavior of polymer trajectories
$(S_0, \ldots, S_N)$ under the polymer measure $\bP_{N,\beta,h}^{\omega}$.
Two important quantities we are interested in
are\footnote{We use the standard notation $a_n\asymp b_n$ if $\limsup_{n\to+\infty} \frac{a_n}{b_n} <+\infty$ and $\limsup_{n\to+\infty} \frac{b_n}{a_n} <+\infty$.}
\begin{itemize}
\item  the \emph{range size} exponent $\xi$, loosely defined as $\bbE \bE_{N,\gb,h}^{\omega} |\cR_N|  \asymp N^{\xi}$;
\item the  \emph{fluctuation} exponent $\chi$, loosely
defined as $\mathbb{V}\mathrm{ar}(\log Z_{N,\gb,h}^{\omega} )  \asymp N^{2 \chi}$.
\end{itemize}

In view of \eqref{E101}, there are three quantities that may influence the behavior of the polymer:
the energy collected from the \emph{random environment} $\omega$;
the penalty $h$ (or reward depending on its sign) for having a large range;
the entropy cost of the exploration of the random walk $S$. 
If $\beta=0$ and $h>0$, then we recover a random walk penalized by its
range. This model is by now quite well understood: the random walk
folds itself in a ball of radius $\rho_d N^{1/(d+2)}$, for some specific constant $\rho_d>0$, see \cite{DV79,Szni91,Bolt94,BC18,DFSX18} (these works mostly focus on the case of dimension $d\geq 2$).
If $\beta=0$ and $h<0$, then 
we get a random walk rewarded by its range: the 
random walk ``stretches'' to obtain a range of order $N$.
If $\gb>0$ and $h=0$, then we obtain a model for a non-directed polymer
in random environment, the environment being seen only once by the random walk (in the same spirit as the excited random walk \cite{BW03} or more generally the cookie random walk~\cite{Zer05}).
In general, disorder should have a ''stretching'' effect because
the random walk is trying to reach more favorable regions in the environment.
We will see that it is indeed the case in dimension $d=1$, where we find 
that the random walk stretches up to a distance $N^{2/3}$ ($\xi=\frac23$).

\subsection{Setting of the paper}
\label{sec:setting}

In this article, we focus on the case of the dimension $d=1$:
the behavior of the model is already very rich 
and we are able to obtain sharp results.
Let us mention that in dimension $d\geq 2$ some aspects of the model are considered in~\cite{BTW21}, but many questions remain open.

Our main  assumption on the environment
is that  $\omega_x$  is in the domain
of attraction of some $\alpha$-stable law, with $\ga \in (0,2]$,
$\ga\neq 1$. More precisely, we assume the following\footnote{We use the standard notation $a(t)\sim b(t)$ if $\lim_{t\to+\infty} \frac{a(t)}{b(t)} =1$ and $a(t) =o(b(t))$
if $\lim_{t\to+\infty} \frac{a(t)}{b(t)} =0$.}.

\begin{assump}
\label{assump1}
 If $\ga = 2$ we assume that $\bbE[\go_0]=0$ and $\bbE[\go_0^2]=1$. 
  If $\ga \in (0,1) \cup (1,2)$ we assume that 
$\bbP(\go_0>t) \sim p\, t^{-\ga}$ and $\bbP(\go_0<-t) \sim q\, t^{-\ga}$  as $t\to\infty$,  
with  $p+q=1$ (and $p>0$);  if $q=0$, we interpret it as  $\bbP(\go_0<-t) =o (t^{-\ga})$.
 Moreover, if $\ga \in (1,2)$, we also assume that $\bbE[\go_0]=0$.
\end{assump}

Let us stress that Assumption~\ref{assump1} ensures that:
\begin{itemize}
\item if $\ga =2$, then $\go_i$ is in the normal
domain of attraction, so that $(\frac{1}{\sqrt{n}} \sum_{i = un}^{vn} \go_i )_{u\leq 0 \leq v}$ converges to a two-sided (standard) Brownian Motion.
\item if $\ga \in (0,1)\cup (1,2)$, then $\go_i$ is in the domain
of attraction of some non-Gaussian stable law and $(\frac{1}{n^{1/\alpha}} \sum_{i = un}^{vn} \go_i )_{u\leq 0 \leq v}$ converges to a two-sided $\alpha$-stable  L\'evy process.
\end{itemize}
We leave the case $\ga=1$ aside mostly for simplicity: indeed, to obtain a process convergence as above, a non-zero centering term is in general needed (even in the symmetric case $p=q$, see~\cite[IX.8]{Feller2}, or~\cite{Ber19}); however most of our analysis applies in that case.
We also focus on pure power tails when $\alpha \in (0,2)$, simply to lighten notation and simplify the statements: our results could easily be adapted to the case of regularly varying tails.

Henceforth we refer to $(X_t)_{t\in \bbR}$ as the two-sided Brownian motion if $\alpha =2$ and as the two-sided L\'evy process defined below if $\ga \in (0,1) \cup (1,2)$.
 We refer to Chapter 1 of \cite{Levyproc}  for an overview on L\'evy processes.

\begin{notation}\label{defL}
Let $\ga\in (0,2]$. 
We let $X^{(1)}=(X_v^{(1)})_{v\ge 0}$ and $X^{(2)}=(X_u^{(2)})_{u\ge 0}$
 be two i.i.d.\ standard Brownian motions if $\alpha=2$ and  two i.i.d.\ ($\alpha$-stable) L\'evy processes with no Brownian component, no drift, and L\'evy measure $\nu(\dd x) = \ga(p\ind_{\{x>0\}}~+~q\ind_{\{x<0\}}) |x|^{-1-\alpha}~\dd x$
 if  $\ga \in (0,2)$.
\end{notation}

We now define a coupling between the discrete environment $(\omega_x)_{x\in \mathbb Z}$ and the processes $X^{(1)}$ and $X^{(2)}$,
using the construction proposed in \cite{MR775792}. Let us consider the space $\bbD=\bbD(\bbR_+,\bbR)$ of all c\`adl\`ag real functions equipped with the Skorokhod metric $d$. 
We let 
\begin{equation}\label{eq:DefSigma}
\Sigma_j^{+} = \Sigma_j^{+}(\go) := \sum_{x=0}^{j} \go_x \, , \qquad \Sigma_{j}^{-}  = \Sigma_j^{-} (\go):= \sum_{x=-j}^{-1} \go_x  \quad \text{ for } j \geq 0\, ,
\end{equation}
and for $u,v \ge 0$
\begin{equation}
X_{N,v}^{(1)} = X_{N,v}^{(1)}(\go) :=N^{-\frac{1}{\alpha}}\Sigma_{\lfloor Nv \rfloor}^+, \qquad 
X_{N,u}^{(2)} = X_{N,u}^{(2)}(\go)=N^{-\frac{1}{\alpha}}\Sigma_{\lfloor Nu \rfloor}^- \,.
\end{equation}
For $0<\alpha<2$ the two (independent) processes $X_{N}^{(1)}=(X_{N,v}^{(1)})_{v\ge 0}$ and $X_N^{(2)}=(X_{N,u}^{(2)})_{u\ge 0}$ are c\`adl\`ag (they are $\bbD$-valued random variables) and they converge in distribution to $X^{(1)}=(X_v^{(1)})_{v\ge 0}$ and $X^{(2)}=(X_u^{(2)})_{u\ge 0}$ as in Notation~\ref{defL}, which are $\bbD$-valued random variables.
Then, as done in~\cite[Section 3]{MR775792}, we can build on the same probability space a sequence of random fields $\go^{(N)} = (\go_x^{(N)})_{x\in \bbZ}$ parametrized by $N$,
such that $\go^{(N)}$ has the same law as the original environment $\go$ for every $N$
and for which the processes  
$X_N^{(1)} (\go^{(N)})$ and $X_N^{(2)} (\go^{(N)})$ converge a.s.\ in the Skorokhod metric on $\bbD$ to $X^{(1)}$ and $X^{(2)}$ respectively---we refer to chapter VI in \cite{JS03} for a characterisation of the convergence of sequences in $\bbD$.
A coupling can also be realized in the case $\alpha=2$, see e.g.~\cite[Ch.~2]{CR14}.
We denote by $\hat\bbP$ the law of the coupling
and in order to lighten notation we will denote~$\hat\go$ instead of~$\go^{(N)}$ in such a coupling,
letting the dependence on $N$ be implicit.

\subsection{Presentation of a  first result}

In the present paper,
we allow $\gb$ and $h$ to vary with the size of the system,
giving rise to a large diversity of possible behaviors.
Before we go into these details, let us 
already state how our results translate in the case of fixed parameters $\gb,h$.

We  define 
\begin{equation*}
M_N^+ :=\max_{0 \leq n\leq N} S_n \geq 0 \qquad \text{and} \qquad  M_N^{-} :=  \min_{0\leq n\leq N} S_n \leq 0
\end{equation*}
the right-most and left-most points of the random walk after $N$ steps. 
In particular, the size of the range is $M_N^+-M_N^-$.

\begin{theorem}
\label{thm:fixed}
Consider the coupling $\hat\bbP$ defined above.

\begin{enumerate}
\item Case $\ga\in (1,2]$.

\begin{enumerate}
\item If $\gb\geq 0$ and $h>0$.
Then, for any $\gep>0$, we have that 
\[
\lim_{N\to\infty} \frac{1}{N^{\frac{1}{3}} } \log Z_{N,\gb,h}^{\go }=-\frac{3}{2}(h \pi)^{2/3} \qquad \text{$\bbP$-a.s.}
\]
 and
\[
\lim_{N\to\infty} \bP_{N,\gb,h}^{\go } \Big( \Big| \frac{1}{N^{1/3}} (M_N^+-M_N^-) - \pi^{\frac23} h^{-\frac13} \Big| >\gep \Big)   = 0 \qquad \text{$\bbP$-a.s.}
\]
(the coupling is not needed here).
\vspace{0.2cm}

\item \label{thm:fixedb}If $\gb >0$ and $h=0$.
Then, letting 
\[
Y_{u,v}^{ (2)} = Y_{u,v}^{\beta,  (2)} :=  \gb (X_{v}^{(1)}  + X_{u}^{(2)}) - \frac{1}{2} ( u\wedge v +u+v)^2
\] 
for $u,v\geq 0$,
we have 
\[
\lim_{N\to\infty} \frac{1}{N^{\frac{1}{2\ga-1}} } \log Z_{N,\gb,h}^{\hat\go } =  \sup\limits_{u,v\geq 0 }  \Big\{ Y_{u,v}^{(2)} \Big\} \in (0,+\infty) \,, \qquad  \text{ in $\hat\bbP$-probability.} 
\]
Additionally, the location of the maximizer of $\sup_{u,v\geq 0 }  \big\{ Y_{u,v}^{ (2)} \big\}$ is unique, that is $(\cU^{(2)},\cV^{(2)}):=$ $ \argmax_{u,v\geq 0}  \big\{ Y_{u,v}^{(2)} \big\}$ is well-defined, and for any $\gep>0$, 
\[
\lim_{N\to\infty} \bP_{N,\gb,h}^{\hat\go }  \Big( \Big|  \frac{1}{N^{\frac{\ga}{2\ga-1}}} (M_N^- , M_N^+ ) - (-\cU^{(2)},\cV^{(2)}) \Big|>\epsilon \Big) = 0 \,, \qquad  \text{ in $\hat\bbP$-probability.} 
\]

\item If $\gb\geq 0$ and $h<0$.
Then for any $\gep>0$, we have that 
\[
\lim_{N\to\infty} \bP_{N,\gb,h}^{\go } \Big( \Big| \frac{1}{N} |S_N| -   \tanh |h| \Big| >\gep \Big)    = 0 \,, \qquad  \text{$\bbP$-a.s.} 
\]
\end{enumerate}

\item \label{thm:fixedc} Case $\ga\in (0,1)$. Let $\gb>0$ and $h\in \bbR$.
Then letting
\[
Y_{u,v}^{(3)}= Y_{u,v}^{\gb, (3)} := \gb (X_{v}^{(1)}  + X_{u}^{(2)})  - \infty \ind_{\{u\wedge v +u+v > 1\} }
\]
for $u,v\geq 0$,
we have
\[
\lim_{N\to\infty} \frac{1}{N^{\frac{1}{\ga}}} \log Z_{N,\gb,h}^{\hat\go } =  \sup_{u,v\geq  0 } \Big\{Y_{u,v}^{(3)} \Big\} \in (0,+\infty)\,, \qquad  \text{ $\hat\bbP$-a.s.} 
\]
Additionally,  $(\cU^{(3)},\cV^{(3)}):=\argmax_{u,v\geq 0}  \big\{ Y_{u,v}^{ (3)} \big\}$ is well-defined and for any $\gep>0$, 
\[
\lim_{N\to\infty} \bP_{N,\gb,h}^{\hat\go }  \Big( \Big|  \frac{1}{N} (M_N^- , M_N^+ ) - (-\cU^{(3)},\cV^{(3)}) \Big|>\epsilon \Big) = 0 \,, \qquad  \text{ $\hat\bbP$-a.s.} 
\]
\end{enumerate}
\end{theorem}

Let us stress that, when $\ga \in (1,2]$,
the range size (see Definition~\ref{def:endtoend} below for a proper definition) of the polymer is:
\begin{itemize}
\item[(a)] of order $N^{1/3}$ if $h>0$ --- folded phase, this is included in Theorem~\ref{C104};
\item[(b)] of order $N^{\ga/(2\ga-1)}$ if $h=0$, $\gb>0$ --- extended phase, this is included in Theorem~\ref{C102};
\item[(c)] of order $N$ if $h<0$ --- extended phase, this is included in Theorem~\ref{R4tilde}.
\end{itemize}
On the other hand, in the case $\ga\in (0,1)$,
the range size is always of order $N$, for whatever value of $h\in \bbR$ --- extended phase, this is included in Theorem~\ref{C103} below.

\begin{remark}\label{rem:thmfixed}
The main advantage of using the coupling $\hat\bbP$ defined above is that we are able to discuss the size of the range of the random walk in relation to the environment. More precisely, with high $\hat \bbP$-probability (or $\hat \bbP$-a.s., depending on the case), the end-points of the range under the polymer measure converge to the maximizer of the variational problem. This is what we present in Theorem \ref{thm:fixed} and later in Section~\ref{sec:MainResults} (Regions~2, 3 and 4, see Figure~\ref{F1}). 
 Let us mention that there are some convergences in $\hat \bbP$-probability that we are not able to upgrade to $\hat \bbP$-a.s.\ convergences: this is due to a lack of control of the convergence of $(\Sigma_j^{\pm})_{j\geq 1}$ on different scales in Lemmas~\ref{lem:R2step2} and~\ref{R4LM2}, see Remark~\ref{rem:almostsure}. Note that this problem disappears we have a good control on the coupling (see again Remark~\ref{rem:almostsure}) or when $\xi=1$.
On the other hand, in the case when the limiting behavior is not affected by the environment (for example in Regions~1, 5, 6 discussed in Section~\ref{sec:MainResults}), the coupling is not needed.
\end{remark}

\subsection{Varying the parameters $\gb$ and $h$}

In order to observe a transition between a folded phase ($h>0,\gb=0$)
and an unfolded phase ($h=0,\gb>0$,  or $h< 0$),
a natural idea is to consider parameters $\beta$ and $h$ that depend on the size of the system, \textit{i.e.}\ $\beta:=\beta_{N}$ and $h:=h_{N}$. 
There are then some sophisticated balances between the energy gain, the range penalty and the entropy cost as we tune $\beta_{N}$ and $h_{N}$. 
Our main results identify the different regimes 
for the behavior of  the random walk:
we provide a complete (and rich) phase diagram (see Figures~\ref{F1}-\ref{F2}-\ref{F3} below),
and describe each phase precisely (range size and fluctuation exponents, limit of the rescaled log-partition function).

In the rest of the paper, we therefore consider the following setting:
\begin{equation}
\label{def:betah}
\beta_{N}:=\hat{\beta}\, N^{-\gamma} \quad
\text{ and } \quad 
h_{N}:= \hat{h}\, N^{-\zeta} \, ,
\end{equation}
where $\gamma, \zeta\in \bbR$ describe the asymptotic behavior of $\gb_N, h_N$, and $\hat{\beta}>0,\hat{h} \in \bbR$ are two fixed parameters.
We could consider a slightly more general setting,
adding some slowly varying function in the asymptotic behavior
of $\beta_N$ or $h_N$:
we chose to stick to the simpler strictly power-law case,
to avoid lengthy notation and more technical calculations.


\section{Some heuristics: presentation of the phase diagrams} 
\label{sec:heuristics}

In this section, we focus on the case $h\geq 0$; the case $h<0$ is considered in Section~\ref{rem:h<0} (it has a less rich behavior and is somehow simpler, 
see Remark~\ref{rem:h} below).
In analogy with the directed polymer model in a heavy-tailed random
environment \cite{BT19a,BT19b}, the presence of heavy-tails
(Assumption~\ref{assump1}) strongly impacts the behavior of the model:
the phase diagrams are different according to whether
$\ga\in (1,2]$, $\ga\in (\frac12,1)$ or $\ga \in (0,\frac12)$.

\smallskip
Let us denote $\xi$ the typical range size exponent of the random walk under the polymer measure $\mathbf{P}_{N,\beta_{N},h_{N}}^{\omega}$
(see Definition~\ref{def:endtoend} below for a proper definition),
and let us develop some heuristics to determine $\xi\in [0,1]$.

First of all, thanks to Lemmas~\ref{lem:superdiffusive}-\ref{lem:subdiffusive} in Appendix, we have for $0<a<b$,
\begin{equation}\label{E106} 
\log\mathbf{P}\big( |\mathcal{R}_{N}|  \in (a N^{\xi} , bN^{\xi}) \big) \asymp \log\mathbf{P}\Big(\max_{0\leq n\leq N}|S_{n}|  \in (a N^{\xi} , bN^{\xi}) \Big)\asymp\begin{cases}
-N^{2\xi-1},& \text{if}~\xi\geq\frac{1}{2},\\
-N^{1-2\xi},& \text{if}~\xi\leq\frac{1}{2}.
\end{cases}
\end{equation}
If $\xi >1/2$, this corresponds to a ``stretching'' of the random walk, whereas when $\xi<1/2$, this corresponds to a ``folding'' of the random walk. We  refer to~\eqref{E106}
as the \emph{entropic cost} of having range size $N^{\xi}$.

Then, if the range size is of order $N^{\xi}$,
then under Assumption~\ref{assump1} and in view of \eqref{def:betah}, we get that
\begin{equation}\label{E107-1}
\gb_N \sum_{x\in \cR_N} \go_x \asymp  \hat \gb N^{\frac{\xi}{\ga} -\gamma} \, , \qquad h_N |\cR_N| \asymp \hat h  N^{\xi-\zeta} \, .
\end{equation}
We refer to the first term as the ``disorder'' term, and to the second one as the ``range'' term (recall we focus for now on the case $\hat h> 0$ so the ``range'' term is always with a minus sign).
All together, if the range size is of order $N^{\xi}$, then  $\log Z_{N,\beta_{N},h_{N}}^{\omega}$ should get a contribution from three terms:
\begin{equation}\label{E107}
\text{disorder} \asymp \hat \gb N^{\frac{\xi}{\alpha}-\gamma} \,,
\quad 
\text{range} \asymp  - \hat h N^{\xi-\zeta}  \,,
\quad
\text{entropy} \asymp -
\begin{cases}
N^{1- 2\xi} \quad \text{if } \xi\leq 1/2 \, , \\
N^{2\xi-1} \quad \text{if } \xi \geq 1/2\, .
\end{cases}
\end{equation}

In \eqref{E107}, there is therefore a competition
between the ``disorder'' (first term), the ``range'' (second term), and the ``entropy'' (last term). 
We now discuss how a balance can be achieved between these terms depending on $\gamma$ and $\zeta$ (and how they determine $\xi$).
There are three main possibilities:
\begin{enumerate}
\item[(i)] there is a ``disorder''-``entropy'' balance (and the ``range'' term is negligible);
\item[(ii)] there is a ``range''-``entropy'' balance (and the ``disorder'' term is negligible);
\item[(iii)] there is a ``range''-``disorder'' balance (and the ``entropy'' term is negligible).
\end{enumerate}
To summarize, all three regimes can occur (depending on $\gamma,\zeta$) if $\alpha\in (1,2]$;
on the other hand, regime (iii) disappears if $\alpha\in (0,1)$,
and regime (i) disappears if $\alpha\in (0,\frac12)$.
We now determine for which values of $\gamma,\zeta$ one can observe the different regimes above:
we  consider the three subcases $\ga\in (1,2]$, $\ga\in (\frac12,1)$ and $\ga\in (0,\frac12)$ separately.

\subsection{Phase diagram for $\ga\in (1,2]$}
Instead of looking for ``disorder''-``entropy'', ``range''-``entropy'' or ``range''-``disorder'' balance, we will
find conditions to have the  ``disorder" term  much larger, much smaller, or of the order of the ``range" term.

\smallskip
\noindent
\textit{Case I (``disorder"$\,\gg\,$``range").} This corresponds to having $\xi/\alpha-\gamma>\xi-\zeta$.
In that case, the random walk should not feel the penalty for having a large range, so we should have $\xi\geq 1/2$.
The competition occurs only between energy and entropy: one could achieve a balance if  $\xi/\alpha-\gamma=2\xi-1$, that is if
\begin{equation}\label{E108}
\xi=\frac{\ga}{2\ga-1} (1-\gamma)  \quad\mbox{when}\quad\gamma<\frac{(2\ga-1)\zeta-(\ga-1)}{\ga} \,,
\end{equation}
where the condition on $\gamma$ derives from the fact that  $\xi/\alpha-\gamma>\xi-\zeta$, \textit{i.e.}\ $\gamma<\zeta -\xi\,\frac{\alpha-1}{\alpha} $, in the regime considered here.
However, since $\xi\leq1$, we should have $\xi=1$ when $\gamma$ is too small, more precisely when $\gamma\leq- \frac{\alpha-1}{\alpha}$. Thus, we should have have
\begin{equation}\label{E109}
\xi=1\quad\mbox{when}\quad\gamma\leq-\frac{\alpha-1}{\alpha}~\mbox{and}~\gamma<\zeta-\frac{\ga-1}{\ga}.
\end{equation}
Also, since $\xi\geq1/2$, we should have $\xi=1/2$ if $\gamma$ is too large, more precisely when
 $\gamma\geq \frac{1}{2\ga}$. Thus, we should have
\begin{equation}\label{E110}
\xi=\frac{1}{2}\quad\mbox{when}\quad\gamma\geq\frac{1}{2\ga}~\mbox{and}~\gamma<\zeta-\frac{\ga-1}{2\ga}.
\end{equation}

\smallskip
\noindent
\textit{Case II (``disorder"$\,\ll\,$``range").} This corresponds to having $\xi/\alpha-\gamma<\xi-\zeta$. In that case, the random walk 
feels the penalty for having a large range, so we should have $\xi\le 1/2$. 
The competition occurs only between range and entropy: one could achieve a balance if  $\xi-\zeta=1-2\xi$, that is if
\begin{equation}\label{E111}
\xi=\frac{1+\zeta}{3}\quad\mbox{when}\quad\gamma>\frac{(2\ga+1)\zeta-(\ga-1)}{3\ga}\, , 
\end{equation}
where the condition on $\gamma$ derives from the fact that  $\xi/\alpha-\gamma>\xi-\zeta$, \textit{i.e.}\ $\gamma>\zeta -\xi\,\frac{\alpha-1}{\alpha} $, in the regime considered here.
Since $\xi\in[0,1/2]$, similarly to \eqref{E109}-\eqref{E110} we should have that 
\begin{equation}\label{E112}
\xi=0\quad\mbox{when}\quad\zeta\leq-1~\mbox{and}~\gamma>\zeta,
\end{equation}
and
\begin{equation}\label{E113}
\xi=\frac{1}{2}\quad\mbox{when}\quad\zeta\geq\frac{1}{2}~\mbox{and}~\gamma>\zeta-\frac{\ga-1}{2\ga}.
\end{equation}

\smallskip
\noindent
\textit{Case III (``disorder"$\,\asymp\,$``range"$\,\gg\,$``entropy'').} This corresponds to having $\xi/\ga-\gamma=\xi-\zeta$, that is
\begin{equation}\label{E114a}
\xi=\frac{\ga}{\ga-1} (\zeta-\gamma)\, .
\end{equation}
In this regime, the entropy cost should be negligible compared to the  disorder gain, and we should therefore have that 
$\xi/\alpha -\gamma >1-2\xi$ if 
$\xi\leq1/2$ and $\xi/\alpha -\gamma >2\xi-1$ for $\xi\geq1/2$: after some calculation
(and using \eqref{E114a}), we find the following condition on $\gamma$
\begin{equation}\label{E114}
\frac{(2\ga-1)\zeta-(\ga-1)}{\ga} < \gamma <\frac{(2\ga+1)\zeta-(\ga-1)}{3\ga}.
\end{equation}
Moreover, since $\xi\in[0,1]$, we must have
\begin{equation}\label{E115}
\zeta-\frac{\ga-1}{\ga}\leq\gamma\leq\zeta.
\end{equation}

\smallskip
To summarize, for $\ga \in (1,2]$, we have identified six different regimes according to the value of $\gamma,\zeta$:
they are described as follows.
A representation of these regions in the $(\zeta,\gamma)$-diagram is given in Figure~\ref{F1} below.
\begin{description}
\item[Region 1] ``disorder'', ``range'' $\ll$ ``entropy'' 
(Case I-II degenerate)
\[
R_{1}=\left\{\xi=\tfrac{1}{2},~\gamma>\tfrac{1}{2\ga},~\zeta>\tfrac{1}{2}\right\} \,;
\]
\item[Region 2] ``range'' $\ll$ ``disorder'' $\asymp$ ``entropy'' (Case I)
\[
R_{2}=\left\{\xi=\tfrac{\ga}{2\ga-1}(1-\gamma),~\tfrac{1-\ga}{\ga}<\gamma<\tfrac{(2\ga-1)\zeta-(\ga - 1)}{\ga} \wedge \tfrac{1}{2\ga}\right\} \, ;
\]
\item[Region 3] ``range'', ``entropy'' $\ll$ ``disorder'' (Case I degenerate)
\[
R_{3}=\left\{\xi=1,~\gamma<-\tfrac{\ga-1}{\ga},~\gamma<\zeta-\tfrac{\ga-1}{\ga}\right\} \,;
\] 
\item[Region 4] ``entropy'' $\ll$ ``range'' $\asymp$ ``disorder'' (Case III)
\[
R_{4}=\left\{\xi=\frac{\alpha}{\alpha-1}(\zeta-\gamma),~\tfrac{(2\ga-1)\zeta-(\ga - 1)}{\ga}\vee(\zeta-\tfrac{\ga-1}{\ga})<\gamma<\tfrac{(2\ga+1)\zeta-(\ga - 1)}{3\ga}\wedge\zeta\right\}\,;
\]
\item[Region 5]  ``disorder'' $\ll$ ``range'' $\asymp$ ``entropy'' (Case II)
\[
R_{5}=\left\{\xi=\tfrac{1+\zeta}{3},~\gamma>\tfrac{(2\ga+1)\zeta-(\ga - 1)}{3\ga},~-1<\zeta<\tfrac{1}{2}\right\}\,;
\]
\item[Region 6]  ``disorder'', ``entropy'' $\ll$ ``range''  (Case II degenerate)
\[
R_{6}=\left\{\xi=0,~\gamma>\zeta,~\zeta<-1\right\} \, .
\]
\end{description}

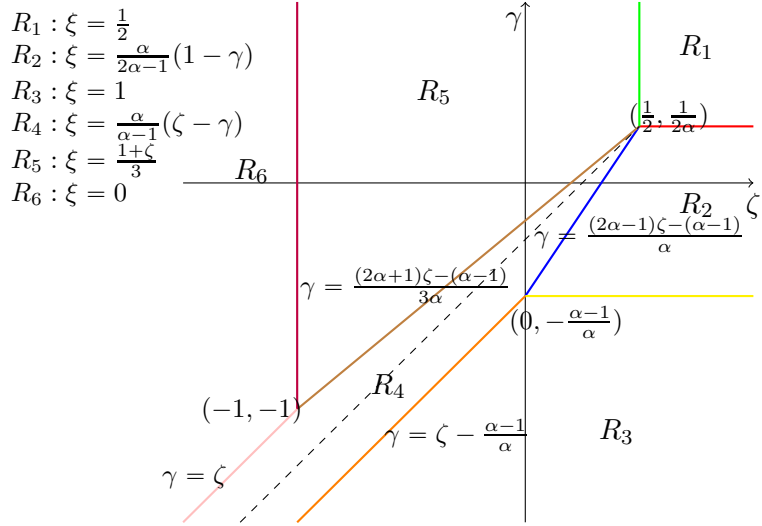
\begin{figure}[ht]
\centering
\begin{tikzpicture}[scale=3]
\draw[->](-1.5,0)--(1,0) node[anchor=north] {$\zeta$};
\draw[->](0,-1.5)--(0,0.8);
\draw (-0.05,0.8) node[anchor=north] {$\gamma$};
\draw[red,thick](0.5,0.25)--(1,0.25);
\draw[green,thick](0.5,0.25)--(0.5,0.8);
\draw (0.63,0.41) node[anchor=north] {\small $(\frac{1}{2},\frac{1}{2\ga})$};
\draw (0.75,0.7) node[anchor=north] {$R_{1}$};
\draw[blue,thick](0,-0.5)--(0.5,0.25);
\draw (0.5,-0.1) node[anchor=north] {\small $\gamma=\frac{(2\ga-1)\zeta-(\ga-1)}{\ga}$};
\draw (0.75,0) node[anchor=north] {$R_{2}$};
\draw[yellow,thick](0,-0.5)--(1,-0.5);
\draw (0.18,-0.5) node[anchor=north] {\small $(0,-\frac{\ga-1}{\ga})$};
\draw (0.4,-1) node[anchor=north] {$R_{3}$};
\draw[brown,thick](-1,-1)--(0.5,0.25);
\draw[purple,thick](-1,-1)--(-1,0.8);
\draw (-1.2,-0.9) node[anchor=north] {\small $(-1,-1)$};
\draw (-0.525,-0.325) node[anchor=north] {\small $\gamma=\frac{(2\ga+1)\zeta-(\ga-1)}{3\ga}$};
\draw (-0.4,0.5) node[anchor=north] {$R_{5}$};
\draw[pink,thick](-1,-1)--(-1.5,-1.5);
\draw (-1.2,0.15) node[anchor=north] {$R_{6}$};
\draw[orange,thick](-1,-1.5)--(0,-0.5);
\draw (-0.6,-0.8) node[anchor=north] {$R_{4}$};
\draw (-1.45,-1.2) node[anchor=north] {\small $\gamma=\zeta$};
\draw (-0.3,-1) node[anchor=north] {\small $\gamma=\zeta-\frac{\alpha-1}{\alpha}$};
\draw (-2.3,0.7) node[anchor=west] {\small $R_{1}:\xi=\frac{1}{2}$};
\draw (-2.3,0.55) node[anchor=west] {\small $R_{2}:\xi=\frac{\ga}{2\ga-1}(1-\gamma)$};
\draw (-2.3,0.4) node[anchor=west] {\small $R_{3}:\xi=1$};
\draw (-2.3,0.25) node[anchor=west] {\small $R_{4}:\xi =\frac{\alpha}{\alpha-1}(\zeta-\gamma)$};
\draw (-2.3,0.1) node[anchor=west] {\small $R_{5}:\xi=\frac{1+\zeta}{3}$};
\draw (-2.3,-.05) node[anchor=west] {\small $R_{6}:\xi=0$};
\draw[dashed] (-1.25,-1.5)--(0.5,0.25);
\end{tikzpicture}
\caption{Phase diagram in the case $\ga\in(1,2]$. The region $R_{1}$ and the dashed line $\gamma=\zeta-\frac{\alpha-1}{2\alpha}$ are the thresholds that split the regions of super-diffusivity and sub-diffusivity. 
Note that when $\ga=1$, the four lines $\gamma= \frac{(2\ga-1)\zeta-(\ga-1)}{\ga}$, $\gamma= \frac{(2\ga+1)\zeta-(\ga-1)}{3\ga}$, and $\gamma= \zeta$, $\gamma=\zeta-\frac{\alpha-1}{\alpha}$ all merge to the line $\gamma=\zeta$.
}\label{F1}
\end{figure}

\subsection{Phase diagram for $\ga\in (0,1)$}
Let us highlight the main differences with the case $\alpha \in (1, 2]$:  the region $R_4$ no longer exists when $\ga<1$, and the region $R_2$ also disappears when $\ga<1/2$.
Indeed, region~$R_4$ corresponds to the case ``disorder"$\,\asymp\,$``range", in which we have $\xi=\frac{\ga}{1-\ga} (\gamma-\zeta)$:  it is easy to check that  for $\alpha\in(0,1)$ there is no $\gamma$ that can satisfy \eqref{E115}, which suggests that there is no ``disorder''-``range'' balance possible.
In the same manner, when $\ga \in (0,\frac12)$ there is no $\gamma$ that satisfy $\frac{1-\alpha}{\alpha} < \gamma <\frac{1}{2\alpha}$ (see the definition of~$R_2$ above),
which suggests that there is no ``disorder''-``entropy'' balance possible: region~$R_2$ no longer exists.
We also refer to Section~\ref{sec:comments} (Comment 2) for further comments on why regions~$R_4$ and $R_2$ disappear for $\ga<1$ and $\ga<1/2$ respectively.

All together, for $\ga \in (\frac12,1)$ we obtain the 
$(\zeta,\gamma)$-diagram presented in Figure~\ref{F2} below: the different regions are described as follows:
\begin{itemize}
\item[] $R_{1}=\left\{\xi=\tfrac{1}{2},~\gamma>\tfrac{1}{2\ga},~\zeta>\tfrac{1}{2}\right\},$
\item[]  $R_{2}=\left\{\xi=\tfrac{\ga}{2\ga-1}(1-\gamma),~\tfrac{1-\ga}{\ga}<\gamma<\tfrac{(2\ga-1)\zeta-(\ga - 1)}{\ga} \wedge \frac{1}{2\ga}\right\},$
\item[]  $R_{3}=\left\{\xi=1,~\gamma<\tfrac{1-\ga}{\ga},~\gamma<\zeta-\tfrac{\ga-1}{\ga}\right\},$
\item[]  $R_{5}=\left\{\xi=\tfrac{1+\zeta}{3},~\big(\tfrac{(2\ga-1)\zeta-(\ga - 1)}{\ga}\big)\wedge \big(\zeta-\tfrac{\ga-1}{\ga}\big)<\gamma,~-1<\zeta<\tfrac{1}{2}\right\},$
\item[]  $R_{6}=\left\{\xi=0,~\gamma>\zeta-\tfrac{\ga-1}{\ga},~\zeta<\frac{1}{\ga}-2\right\}.$
\end{itemize}

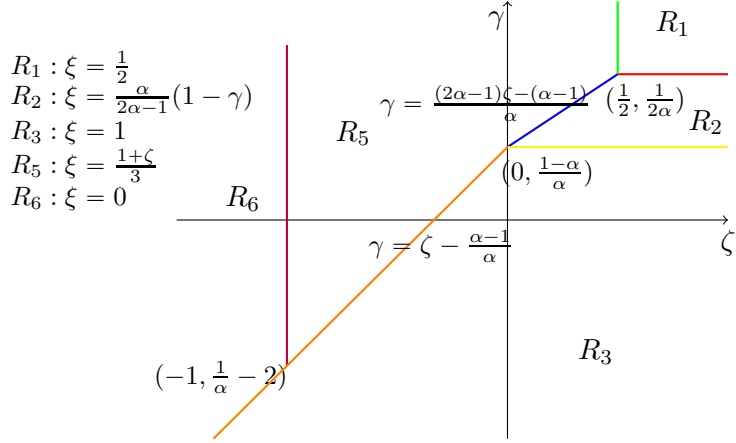
\begin{figure}[ht]
\centering
\begin{tikzpicture}[scale=2.9]
\draw[->](-1.5,0)--(1,0) node[anchor=north] {$\zeta$};
\draw[->](0,-1)--(0,1);
\draw (-0.05,1) node[anchor=north] {$\gamma$};
\draw[red,thick](0.5,2/3)--(1,2/3);
\draw[green,thick](0.5,2/3)--(0.5,1);
\draw (0.625,2/3) node[anchor=north] {\small $(\frac{1}{2},\frac{1}{2\ga})$};
\draw (0.75,1) node[anchor=north] {$R_{1}$};
\draw[blue,thick](0,1/3)--(0.5,2/3);
\draw (-0.1,2/3) node[anchor=north] {\small $\gamma=\frac{(2\ga-1)\zeta-(\ga-1)}{\ga}$};
\draw (0.9,0.55) node[anchor=north] {$R_{2}$};
\draw[yellow,thick](0,1/3)--(1,1/3);
\draw (0.18,1/3) node[anchor=north] {\small $(0,\frac{1-\ga}{\ga})$};
\draw (0.4,-1/2) node[anchor=north] {$R_{3}$};
\draw[purple,thick](-1,-2/3)--(-1,0.8);
\draw (-1.3,-0.6) node[anchor=north] {\small $(-1,\frac{1}{\ga}-2)$};
\draw (-0.7,0.5) node[anchor=north] {$R_{5}$};
\draw (-1.2,0.2) node[anchor=north] {$R_{6}$};
\draw[orange,thick](-4/3,-1)--(0,1/3);
\draw (-0.3,0) node[anchor=north] {\small $\gamma=\zeta-\frac{\alpha-1}{\alpha}$};
\draw (-2.3,0.7) node[anchor=west] {\small $R_{1}:\xi=\frac{1}{2}$};
\draw (-2.3,0.55) node[anchor=west] {\small $R_{2}:\xi=\frac{\ga}{2\ga-1}(1-\gamma)$};
\draw (-2.3,0.4) node[anchor=west] {\small $R_{3}:\xi=1$};
\draw 
(-2.3,0.25) node[anchor=west] {\small $R_{5}:\xi=\frac{1+\zeta}{3}$};
\draw 
(-2.3,0.1) node[anchor=west] {\small $R_{6}:\xi=0$};
\end{tikzpicture}
\caption{Phase diagram in the case $\ga\in(1/2,1)$. Compared to Figure~\ref{F1}, the region $R_4$ no longer exists.
}\label{F2}
\end{figure}

Finally, for $\ga \in (0, \frac12)$ we obtain the 
$(\zeta,\gamma)$-diagram presented in Figure~\ref{F3} below: the different regions are described as follows:
\nopagebreak
\begin{itemize}
\item[] $R_{1}=\left\{\xi=\tfrac{1}{2},~\gamma>\tfrac{1-\ga}{\ga},~\zeta>\tfrac{1}{2}\right\},$
\item[]  $R_{3}=\left\{\xi=1,~\gamma<\tfrac{1-\ga}{\ga},~\gamma<\zeta-\tfrac{\ga-1}{\ga}\right\},$
\item[] $R_{5}=\left\{\xi=\tfrac{1+\zeta}{3},~\tfrac{1-\ga}{\ga}\wedge \big(\zeta-\tfrac{\ga-1}{\ga}\big)<\gamma,~-1<\zeta<\tfrac{1}{2}\right\},$
\item[] $R_{6}=\left\{\xi=0,~\gamma>\zeta-\tfrac{\ga-1}{\ga},~\zeta<\frac{1}{\ga}-2\right\}.$
\end{itemize}

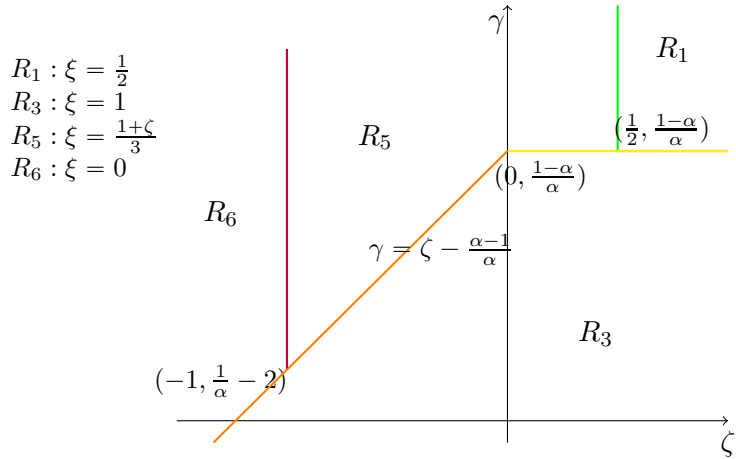
\begin{figure}[ht]
\centering
\begin{tikzpicture}[scale=2.9]
\draw[->](-1.5,-0.9)--(1,-0.9) node[anchor=north] {$\zeta$};
\draw[->](0,-1)--(0,1);
\draw (-0.05,1) node[anchor=north] {$\gamma$};
\draw[red,thick](0.5,1/3)--(1,1/3);
\draw[green,thick](0.5,1/3)--(0.5,1);
\draw (0.7,0.55) node[anchor=north] {\small $(\frac{1}{2},\frac{1-\ga}{\ga})$};
\draw (0.75,0.9) node[anchor=north] {$R_{1}$};
\draw[yellow,thick](0,1/3)--(1,1/3);
\draw (0.15,1/3) node[anchor=north] {\small $(0,\frac{1-\ga}{\ga})$};
\draw (0.4,-0.4) node[anchor=north] {$R_{3}$};
\draw[purple,thick](-1,-2/3)--(-1,0.8);
\draw (-1.3,-0.6) node[anchor=north] {\small $(-1,\frac{1}{\ga}-2)$};
\draw (-0.6,0.5) node[anchor=north] {$R_{5}$};
\draw (-1.3,0.15) node[anchor=north] {$R_{6}$};
\draw[orange,thick](-4/3,-1)--(0,1/3);
\draw (-0.3,0) node[anchor=north] {\small $\gamma=\zeta-\frac{\alpha-1}{\alpha}$};
\draw (-2.3,0.7) node[anchor=west] {\small $R_{1}:\xi=\frac{1}{2}$};
\draw 
(-2.3,0.55)node[anchor=west] {\small $R_{3}:\xi=1$};
\draw 
(-2.3,0.4) node[anchor=west] {\small $R_{5}:\xi=\frac{1+\zeta}{3}$};
\draw 
(-2.3,.25)node[anchor=west] {\small $R_{6}:\xi=0$};
\end{tikzpicture}
\caption{Phase diagram in the case $\ga\in(0,1/2)$. Compared to Figure~\ref{F2}, the region $R_2$ no longer exists.
}\label{F3}
\end{figure}

\begin{remark}
\label{rem:h}
In the case $\hat h<0$, 
one can conduct similar computation as in \eqref{E108}---\eqref{E115} and obtain a different phase diagram than those of Figures~\ref{F1}-\ref{F2}-\ref{F3}, see Figures~\ref{F4} and \ref{F5} below (note that regions $R_1,R_2,R_3$ are unchanged,
since the range term is negligible in these regions).
Let us stress that when $\hat h<0$, the ``disorder" and  ``range" terms both play in the same direction and encourage exploration, resulting in a much simpler diagram: only range size exponents $\xi\geq 1/2$ are possible, see Section~\ref{rem:h<0} below.
\end{remark}

\section{Main results}
\label{sec:MainResults}

Our main results consist in proving the phase diagrams of Figures~\ref{F1}-\ref{F2}-\ref{F3},
with a precise description of the behavior of the polymer in each region.
In order to state our results, let us introduce some definition.

\begin{definition}
\label{def:endtoend}
If $(t_N)_{N\geq 0}$ is a sequence of positive real numbers, 
we say that \emph{$(S_n)_{0\leq n \leq N}$ has range size of order $t_N$ under $\bP_{N,\gb_N,h_N}^{\go}$}
if 
\[
\lim_{\eta \downarrow 0} \limsup_{N\to+\infty} \bbE \Big[ \bP_{N,\gb_N,h_N}^{\go} \Big(  \max_{1\leq n \leq N} |S_n|  \in [\eta,\tfrac1\eta]\,  t_N  \Big) \Big] = 1\, .
\]
\end{definition}

\noindent
If $(S_n)_{0\leq n \leq N}$ has range size of order $N^{\xi}$ under $\bP_{N,\gb_N,h_N}^{\go}$, then
we say that the 
range size exponent is $\xi$.

In our results, we encounter doubly indexed
processes
\begin{equation}
\label{def:Y}
Y_{u,v} := X_v^{(1)}+X_u^{(2)}-f(u,v)  \quad \text{for } u,v\geq 0\,,
\end{equation}
where $X_v^{(1)}$ and $X_u^{(2)}$ are the L\'evy processes of Notation~\ref{defL} and~$f$ is some deterministic function (for instance, a large deviation rate function),
such that $(u,v)\mapsto f(u,v)$ is continuous
on the set where $f(u,v)<+\infty$, with $f(0,0)=0$.
We then have the following result about the maximizer of the variational problem $\sup_{u,v\geq 0} Y_{u,v}$, that we prove in Appendix~\ref{sec:unique}: it ensures the well-posedness of 
$\argmax_{u,v\geq 0} Y_{u,v}$.
\begin{proposition}\label{uniquemaximizer}
Suppose that $\bbP$-a.s.\ the variational problem $\sup_{u,v\geq 0} Y_{u,v}$
is positive with $Y$ defined in~\eqref{def:Y}, and that $Y_{u,v}\to -\infty$ as $\max(u, v) \to+\infty$. Then
\begin{equation*}
\bbP\Big(\argmax\limits_{u,v\geq 0}Y_{u,v}~\text{is a singleton}~\{(\cU,\cV)\}~\text{with}~\cU\neq\cV\Big)=1.
\end{equation*}
\end{proposition}

\noindent
Let us stress that in the case $\alpha=2$ of a Brownian motion, \cite[Lem.~2.6]{KP90}  (or \cite{Pim14}) proves the uniqueness of the maximizer for one-indexed processes, but not doubly-indexed ones.

\subsection{Statement of the results}
\label{sec:results}

We now prove six different theorems, corresponding to the six possible regions in the phase diagram presented in Figure~\ref{F1}.
We mention that we will use $\hat \bbP$ only when the coupling is needed; otherwise we will keep the notation $\bbP$ (in particular this will be the case when the limit of the rescaled log-partition function is non-random).

In this section, we again focus on the case $\hat h >0$,
but several results hold for a general $\hat h\in \bbR$: we will highlight when the results are specific to the case $\hat h >0$. The case $\hat h<0$ will be discussed separately in Section~\ref{rem:h<0}.
Note that the case $h=0$ or $\gb=0$
can be recovered by taking $\zeta =+\infty$ or $\gamma =+\infty$
respectively, while the case of constant $h$ or $\gb$ can be recovered by taking $\zeta =0$ or $\gamma =0$
respectively. One can then recover Theorem \ref{thm:fixed} from  Theorems~\ref{C102}-\ref{C103}-\ref{C104} and \ref{R4tilde} below.

\begin{theorem}[\textbf{Region 1}]\label{C101}
Assume that \eqref{def:betah} holds  with $\hat\gb\geq 0$, $\hat h \in \bbR$
and
$$
\begin{cases} \gamma>\frac{1}{2\ga} \ \ \text{and} \ \ \zeta>\frac12, 
& \quad \text{if}\ \  \alpha \in  [\frac{1}{2},1)\cup(1,2], \\
\gamma>\frac{1-\alpha}{\ga} \ \ \text{and} \ \ \zeta>\frac12, 
&\quad \text{if}\ \  \alpha \in (0,\frac{1}{2}).
\end{cases}
$$
Then  $(S_n)_{0\leq n\leq N}$ has range size of order $\sqrt{N}$ under $\bP_{N,\gb_N,h_N}^{\go}$ (\textit{i.e.}\ $\xi=\frac12$) and 
we have the following convergence
\begin{equation}\label{eq:th:C101}
\lim_{N\to+\infty} Z_{N,\beta_{N},h_{N}}^{\omega}  = 1  \qquad 
\text{$\bbP$-a.s.}\, .
\end{equation}
Moreover, we have  $\lim_{N\to\infty}\|\bP_{N,\gb_N,h_N}^{\go}-\bP\|_{\bT\bV} =0$ $\bbP$-a.s., where $\|\cdot\|_{\bT\bV}$ is the total variation distance.
\end{theorem}

Note that the total variation convergence stated in Theorem \ref{C101} implies that $S_N/\sqrt N$ converges in distribution to a Brownian motion. This convergence holds $\bbP$-almost surely.

\begin{theorem}[\textbf{Region 2}]\label{C102}
Assume that \eqref{def:betah} holds with $\hat\gb>0$, $\hat h \in \bbR$
and 
\[\tfrac{1-\ga}{\ga}<\gamma<\tfrac{(2\ga-1)\zeta-(\ga - 1)}{\ga} \wedge \tfrac{1}{2\ga} \quad \text{and} \quad \ga \in (\tfrac12,1)\cup(1,2].\]
Then $(S_n)_{0\leq n\leq N}$ has range size of order $N^{\xi}$
under $\bP_{N,\gb_N,h_N}^{\hat{\go}}$ with $\xi=\tfrac{\ga}{2\ga-1}(1-\gamma) \in (\frac12,1)$,  and  we have the following convergence  
\begin{equation}\label{E119}
\lim_{N\to\infty} \frac{1}{N^{\frac{\xi}{\ga}-\gamma}} \log Z_{N,\beta_{N},h_{N}}^{\hat\go} =  \cW_{2}
:=\sup\limits_{u,v\geq 0 }\left\{Y_{u,v}^{(2)}  \right\} 
\in(0,+\infty),  \qquad \text{ in $\hat\bbP$-probability} \, ,
\end{equation}
where $Y_{u,v}^{(2)}  =Y_{u,v}^{\hat \gb,  (2)} = \hat \gb (X_{v}^{(1)}  + X_{u}^{(2)} ) - \frac{1}{2} ( u\wedge v +u+v)^2$ is as defined in Theorem \ref{thm:fixed}-\eqref{thm:fixedb}.
Additionally, for any $\gep >0$  we have
\[ \lim_{N\to\infty} \bP_{N,\gb_N,h_N}^{\hat\go}\Big( \Big|\frac{1}{N^{\xi}} (M_N^-, M_N^+ )   - (-\cU^{(2)}, \cV^{(2)})\Big|>\epsilon  \Big) =  0 \,, \qquad
\text{in $\hat\bbP$-probability},
\]
where $(\cU^{(2)}, \cV^{(2)}) := \argmax_{u,v\geq 0}  \big\{ Y_{u,v}^{ (2)} \big\}$ is well-defined
thanks to Proposition~\ref{uniquemaximizer}.
\end{theorem}

Let us stress that the case $\ga = 2$, $\gb=\gb_N \equiv \beta>0$ and $h\equiv 0$
corresponds to the case $\gamma=0$ and $\zeta=+\infty$:
we find in that case that the range size exponent is $\xi=\frac23$.

\begin{theorem}[\textbf{Region 3}]\label{C103}
Assume that \eqref{def:betah} holds  with $\hat\gb>0$, $\hat h \in \bbR$ and
\[
\gamma<\big(\zeta-\tfrac{\ga-1}{\ga}\big) \wedge \big(\tfrac{1-\ga}{\ga}\big) \quad \text{and} \quad \ga \in (0,1)\cup(1,2].\]
Then $(S_n)_{0\leq n\leq N}$ has range size of order $N$ under $\bP_{N,\gb_N,h_N}^{\hat{\go}}$ (\textit{i.e.}\ $\xi=1$), and we have the following convergence  
\begin{equation}\label{E120}
\lim_{N\to\infty} \frac{1}{N^{\frac{1}{\ga}-\gamma}} \log Z_{N,\beta_{N},h_{N}}^{\hat\omega} 
= \cW_{3} 
:=\sup\limits_{u,v\geq 0 }\left\{Y_{u,v}^{(3)}  \right\} 
\in(0,+\infty) \,, \qquad
\text{ $\hat\bbP$-a.s.} \,,
\end{equation}
where $Y_{u,v}^{ (3)}  = Y_{u,v}^{\hat \gb, (3)}  = \hat \gb (X_{v}^{(1)}  + X_{u}^{(2)})  - \infty \ind_{\{u\wedge v +u+v > 1\} }$ is as defined in Theorem \ref{thm:fixed}-\eqref{thm:fixedc}.
Additionally, 
for any $\gep>0$ we have 
\[
\lim_{N\to\infty} \bP_{N,\gb_N,h_N}^{\hat\go}  \Big(\Big| \frac{1}{N} (M_N^-, M_N^+ )  -(-\cU^{(3)}, \cV^{(3)}) \Big|>\gep \Big)   =  0 \,,  \qquad
\text{$\hat\bbP$-a.s.}
\]
where $(\cU^{(3)},\cV^{(3)}):=\argmax_{u,v\geq 0}  \big\{ Y_{u,v}^{ (3)} \big\}$ 
is well-defined
thanks to Proposition~\ref{uniquemaximizer}.
\end{theorem}

\begin{theorem}[\textbf{Region 4}]\label{C106}
Assume that \eqref{def:betah} holds with $\hat\gb>0$, $\hat h >0$ and
\[
\big(\tfrac{(2\ga-1)\zeta-(\ga - 1)}{\ga} \big)\vee \big(\zeta-\tfrac{\ga-1}{\ga}\big)<\gamma<\big(\tfrac{(2\ga+1)\zeta-(\ga - 1)}{3\ga}\big) \wedge \zeta \quad \text{and} \quad \ga \in (1,2].\]
Then $(S_n)_{0\leq n\leq N}$ has range size of order $N^{\xi}$
under $\bP_{N,\gb_N,h_N}^{\hat{\go}}$ with  $\xi =\frac{\alpha}{\alpha-1}(\zeta-\gamma) \in (0,1)$, and we have the following convergence
\begin{equation}\label{E123}
\lim_{N\to\infty} \frac{1}{N^{\xi-\zeta}} \log Z_{N,\beta_{N},h_{N}}^{\hat\omega}
=\cW_{4} := \sup\limits_{u,v\geq 0}\left\{Y_{u,v}^{(4)}\right\}\in (0, +\infty)\,,  \qquad  \text{in $\hat\bbP$-probability}\,,
\end{equation}
where $Y_{u,v}^{ (4)} =Y_{u,v}^{\hat\gb, \hat h, (4)} :=\hat{\beta}(X_{v}^{(1)}  + X_{u}^{(2)})-\hat{h}(u+v)$.
Additionally, 
for any $\gep>0$ we have 
\[
\lim_{N\to\infty}  \bP_{N,\gb_N,h_N}^{\hat\go}  \Big( \Big| \frac{1}{N^{\xi}} (M_N^-, M_N^+ )   -(-\cU^{(4)}, \cV^{(4)})\Big|>\epsilon \Big) 
=  0   \,,  \qquad \text{in $\hat\bbP$-probability}
\]
 where $(\cU^{(4)}, \cV^{(4)}) :=\argmax_{u,v\geq0}\big\{Y_{u,v}^{ (4)}\big\}$ is well-defined thanks to Proposition~\ref{uniquemaximizer}.
\end{theorem}

\begin{theorem}[\textbf{Region 5}]\label{C104}
Assume that \eqref{def:betah} holds  with $\hat\gb>0$, $\hat h>0$ and
\[
\begin{cases} \gamma>\tfrac{(2\ga+1)\zeta-(\ga - 1)}{3\ga}\ \ \text{and}\ \ -1<\zeta<\tfrac{1}{2},  &\quad  \text{if}\ \ \alpha \in (1, 2]\,, \\
\gamma > \big(\tfrac{(2\ga-1)\zeta-(\ga - 1)}{\ga}\big)\wedge \big(\zeta-\tfrac{\ga-1}{\ga}\big) \ \ \text{and}\ \  -1<\zeta<\tfrac{1}{2},
& \quad \text{if}\ \  \alpha \in (\frac{1}{2},1) \,,\\
\gamma >\big(\tfrac{1-\ga}{\ga} \big)\wedge \big(\zeta-\tfrac{\ga-1}{\ga}\big)  \ \ \text{and}\ \  -1<\zeta<\tfrac{1}{2}, 
& \quad \text{if}\ \  \alpha \in (0,\frac{1}{2}) \, .
\end{cases}
\]
Then $(S_n)_{0\leq n\leq N}$ has range size of order $N^{\xi}$
under $\bP_{N,\gb_N,h_N}^{\go}$ with  $\xi = \frac{1+\zeta}{3} \in (0,\frac12)$ , and we have the following convergence 
\begin{equation}\label{E121}
 \lim_{N\to+\infty}  \frac{1}{N^{\xi-\zeta}} \log Z_{N,\beta_{N},h_{N}}^{\omega} =  -\frac{3}{2} (\hat h \pi)^{2/3}  =\sup_{r\geq 0}\Big\{-\hat{h}r-\frac{\pi^2}{2 r^2}\Big\}
  \qquad \text{$\bbP$-a.s.}
\end{equation}
Additionally, for every $\gep>0$, 
we have
\[
\lim_{N\to+\infty} \bP_{N,\gb_N,h_N}^{\go} \Big( \Big| \frac{1}{N^{\xi}} ( M_N^+ - M_N^-) - \pi^{\frac{2}{3}} \hat h^{-\frac{1}{3}} \Big| >\gep \Big) = 0  \qquad \text{$\bbP$-a.s.}
\]
\end{theorem}

\begin{theorem}[\textbf{Region 6}]\label{C105}
Assume that \eqref{def:betah} holds  with $\hat\gb>0$, $\hat h >0$ and
\[
\begin{cases} \gamma>\zeta\ \ \text{and}\ \ \zeta<-1,
& \quad \text{if}\ \ \alpha \in (1, 2], \\
\gamma>\zeta- \frac{\alpha-1}{\alpha}\ \ \text{and}\ \ \zeta<-1, 
& \quad \text{if}\ \ \alpha \in (0,1).
\end{cases}
\]
Then we have the following convergences 
\begin{equation}\label{E122}
\lim_{N\to+\infty} \bP_{N,\gb_N,h_N}^{\go} ( |\cR_n| =2 ) =  1,
\quad   \lim_{N\to+\infty} N^{\zeta} \log Z_{N,\beta_{N},h_{N}}^{\omega}= -2\hat{h} 
\qquad \text{$\bbP$-a.s.}
\end{equation}
\end{theorem}

\medskip

Let us conclude this section with a result that complements 
Theorems~\ref{C102}-\ref{C103}
and Theorem~\ref{C106} in the case $\xi>\frac12$.
It shows that under $\bP_{N,\gb_N,h_N}^{\hat\go}$ trajectories travel ballistically to 
the closest point between $-\cU N^{\xi}$ and $\cV N^{\xi}$
and then to the other one.
Let us introduce some notation to be able to state the result.
For $u,v\geq 0$ with $u\neq v$,  let $\sigma_{u,v}= -1$ if $u< v$ and $\sigma_{u,v} =+1$ otherwise, let $c_{u,v} = u\wedge v + u+v$,
and define the function
\begin{equation}
\label{eq:ballisticfunction}
b_{u,v} (t) = 
\begin{cases}
\sigma_{u,v} \,c_{u,v}\, t &\quad \text{ for } 0\leq t \leq \frac{u\wedge v}{c_{u,v}} \,,\\
-\sigma_{u,v} \,c_{u,v} \,  t + 2 \sigma_{u,v} \,(u\wedge v) & \quad \text{ for }   \frac{u\wedge v}{c_{u,v}}\leq t \leq 1 \,,
\end{cases}
\end{equation}
that goes
with constant speed from $0$ to the closest point between $-u$ and $v$ and then to the other one.
Now, for $\gep>0$, let us define the event
\begin{equation}
\label{eq:ball}
\cB_N^{\gep}(u,v) := \Big\{ \sup_{t\in [0,1]} \Big| \frac{1}{N^{\xi}} S_{\lfloor tN \rfloor}  -b_{u,v}(t) \Big| \leq \gep \Big\} \,.
\end{equation}
We then have the following result.

\begin{proposition}
\label{prop:ballistic}
Assume that, for some  $\xi\in(\frac12,1]$, for any $\gd>0$ we have
\begin{equation}
\label{eq:convprobab}
\lim_{N\to\infty}  \bP_{N,\gb_N,h_N}^{\hat\go}  \Big( \Big| \frac{1}{N^{\xi}} (M_N^-, M_N^+ )   -(-\cU, \cV)\Big|>\gd \Big) 
=  0   \,,  \quad \text{in $\hat\bbP$-probability (resp.\ $\hat \bbP$-a.s.)} \,,
\end{equation}
with $\cU,\cV \geq 0$ two random variables such that $\cU\neq \cV$ a.s.
Then, for any $\gep>0$ we have
\begin{equation*}
\lim_{N\to+\infty} \bP_{N,\gb_N,h_N}^{\hat\go} \big( \cB_N^{\gep} (\cU,\cV) \big) =1 \,,  \quad \text{in $\hat\bbP$-probability  (resp.\ $\hat \bbP$-a.s.)} \,.
\end{equation*}
\end{proposition}

\begin{remark}
An analogous result should also hold in the case $\xi \in (0,\frac12)$. Assume that~\eqref{eq:convprobab} holds
with $\xi\in (0,\frac12)$. Then we expect that
$\frac{N^{-\xi}S_N +\cU}{\cU+\cV}$ converges in distribution (under $\bP_{N,\gb_N,h_N}^{\hat\go}$)
towards a random variable $\mathcal X$ with density $\frac{\pi}{2} \sin(\pi x) \ind_{[0,1]}(x)$.
This result is easy to obtain for a random walk 
conditioned to remain inside an interval $[-a N^{\xi} ,bN^{\xi}]$,
but becomes trickier when the range is conditioned
to be exactly $[-a N^{\xi} ,bN^{\xi}]$.
We are not aware of any such result for random 
walks conditioned on their range, but let us mention~\cite{Bouchot22} where a closely related question is considered.
We therefore chose not to develop this in the present paper to avoid lengthening it.
\end{remark}

\subsection{Some comments on the results (case $\hat h>0$)}
\label{sec:comments}
Let us now make some observations on our results.

\smallskip
{\bf Comment 1.}
Our results describe a transition from \emph{folded} trajectories ($\xi<1/2$) to \emph{stretched} trajectories ($\xi>1/2$),
which is induced by the presence of disorder.
Let us illustrate this in the case $\ga\in (1,2]$ for simplicity;
we refer to the phase diagram of Figure~\ref{F1}.
If $\gb_N  = \hat \gb >0$ and $h_N = \hat h>0$,
that is if $\gamma = \zeta =0$, we find that trajectories are folded, with range size exponent $\xi = 1/3$.
Now, if we keep $h_N =\hat h >0$  fixed (\textit{i.e.}\ $\zeta=0$)  and increase the strength of disorder, that is if we decrease $\gamma$ (taking $\gamma<0$),
we realize that we have transitions between the following regimes: 
\begin{enumerate}
\item[(i)] if $\gamma> \frac{1-\ga}{3\ga}$, the random walk is folded with range size exponent $\xi=\frac13$ (disorder is not strong enough);
\item[(ii)] if $ \frac{1-\ga}{3\ga}> \gamma >\frac{1-\ga}{2\ga}$, then the random walk
is still folded, with range size exponent
$ \frac13 <\xi=\frac{\gamma \ga}{1-\ga}< \frac12$ (disorder makes the random walk less folded);
\item[(iii)]  if $ \frac{1-\ga}{2\ga}> \gamma >\frac{1-\ga}{\ga}$, then
 the random walk is stretched, with range size exponent
$\frac12<\xi=\frac{\gamma \ga}{1-\ga}<1$
(disorder is strong enough to stretch the random walk);
\item[(iv)] if $\gamma < \frac{1-\ga}{\ga}$, then the random walk is completely unfolded and has range size exponent $\xi=1$.
\end{enumerate}

Analogously, if  we keep $\gb_N =\hat \gb >0$ fixed (\textit{i.e.}\ $\gamma=0$) and decrease the penalty for the range, that is if we increase $\zeta$ (taking $\zeta>0$),
we have transitions between the following regimes:
\begin{enumerate}
\item[(i)] if $ 0<\zeta < \frac{\ga-1}{2\ga+1} $, then the random walk
is still folded with range size exponent
$ \frac12<\xi=  \frac{1+\zeta}{3}< \frac{\ga}{2\ga+1} <\frac12$ (disorder plays no role);
\item[(ii)] if $\frac{\ga-1}{2\ga+1}<\zeta < \frac{\ga-1}{2\ga}$, then the random walk
is still folded with range size exponent
$\frac{\ga}{2\ga+1}<\xi= \frac{\zeta\ga}{\ga-1}< \frac12$ (disorder plays a role);
\item[(iii)]   if $ \frac{\ga-1}{2\ga}< \zeta < \frac{\ga-1}{2\ga-1}$, then
 the random walk is stretched, with range size exponent
$1/2<\xi= \frac{\gamma \ga}{1-\ga}< \frac{\ga}{2\ga-1}<1$ (disorder stretches the random walk);
\item[(iv)] if $\zeta >\frac{\ga-1}{2\ga-1}$, then the random walk is stretched and has range size exponent $\frac23 \leq \xi = \frac{\ga}{2\ga-1} <1$ (and the penalty for the range is not felt).
\end{enumerate}

\smallskip
{\bf Comment 2.}
Let us now discuss the limiting distributions for the log-partition function in regions $R_2$, $R_3$, $R_4$. For simplicity,
we will restrict ourselves to the case where $u=0$ in the variational
problems~\eqref{E119}-\eqref{E120}-\eqref{E123} 
(which corresponds to considering the case of a random walk constrained to stay non-negative): the variational problems become, respectively
\begin{equation*}
\tilde \cW_{2} := \sup_{v\geq 0} \left\{ \hat \gb X_v - \tfrac12 v^2 \right\}  , \quad  \tilde \cW_{3}:= \hat \gb \sup_{v\in [0,1]} \left\{  X_v \right\} ,  \quad 
 \tilde \cW_{4} := \sup_{v\geq 0} \left\{  \hat \gb X_v -  \hat h v \right\} \,,
\end{equation*}
with $(X_v)_{v\geq 0} = (X_{v}^{(1)})_{v\geq 0}$.

{\bf a)} The variational problem $\tilde \cW_{3}$ is clearly always finite.
In the case $\ga=2$, $(X_t)_{t\geq 0}$ is a Brownian motion
and it is standard to get that $\tilde \cW_{3}$ has the distribution of $\hat \gb |Z|$, with $Z\sim \mathcal{N}(0,1)$.
In the case $\ga\in (0,2)$, $(X_t)_{t\geq 0}$ is a stable L\'evy process and we get that $\tilde \cW_{3}$ is a postitive $\alpha$-stable random variable (see~\cite[Ch.~VIII]{B96}, and also~\cite{Kuz11}).

{\bf b)} The variational problem $\tilde \cW_{4}$ is finite only when $\ga>1$: when $\ga\in (0,1)$, then $X_v$ grows typically as $v^{1/\ga} \gg v$ as $v\to\infty$ and we therefore have $\cW_{4} =+\infty$. This explains in particular why there is no energy-range balance possible if $\ga\in (0,1)$ and why region~$R_4$ no longer exists in that case.
If $\ga=2$,  $(X_t)_{t\geq 0}$ is a Brownian motion
and it is standard to get that $\tilde \cW_{4}$ is an exponential random variable (here with parameter $2\hat h /\gb^2$).
If $\ga \in (1,2)$, $(X_t)_{t\geq 0}$ is a stable L\'evy process
and $(\hat \gb X_t - \hat h t)_{t\ge 0}$ is also a L\'evy process:
 the distribution of its supremum $\tilde \cW_{4}$ has been studied extensively, going back to \cite{BD57}, but the exact distribution does not appear to be known (we refer to the recent
 papers~\cite{Cha13,KMR13}).

{\bf c)} The variational problem $\tilde \cW_{2}$ is finite only when $\ga>\frac12$: when $\ga\in (0,\frac12)$, then $X_v$ grows typically as $v^{1/\ga} \gg v^2$ as $v\to\infty$ and we therefore have $\cW_{2} =+\infty$. This explains in particular why there is no energy-entropy balance possible if $\ga\in (0,\frac12)$,
and why region $R_2$ no longer exists in that case.
In the case $\ga=2$, that is when $(X_t)_{t\ge 0}$ is a standard Brownian motion, then $\tilde \cW_{4}$  has appeared in various contexts and its density is known (its Fourier transform is expressed in terms of Airy function, see for instance \cite{DS85,Gro89}).
In the case $\ga\in (\tfrac12,2)$, 
exact asymptotics on the tail of the distribution of $\tilde \cW_{4}$
have been derived in~\cite{PS19};
we are not aware 
whether the distribution of $\tilde \cW_{4}$ has been studied in more detail.

\smallskip
{\bf  Comment 3.}
To keep the paper lighter,
we have chosen not to treat the cases
of the boundaries between different regions of the phase diagrams.
These boundary regions do not really hide anything deep:
features of both regions should appear in the limit, and ``disorder'', ``range'' and ``entropy'' may all compete at the same
(exponential) scale.
Let us state the limiting variational problems that one should find in some the most interesting boundary cases, in the case $\alpha\in(1,2]$  for simplicity (we refer to the phase diagram of Figure~\ref{F1}):

\smallskip
\textbullet\ Line between regions $R_2$ and $R_4$: $\gamma = \frac{(2\alpha-1)\zeta-(\alpha-1)}{\alpha}$ and $\zeta\in(0,\frac{1}{2})$. Then one should have $\xi=\frac{\alpha(1-\gamma)}{2\alpha-1}$ and
\[ 
 \lim_{N\to+\infty} \frac{1}{N^{2\xi-1}}\log Z_{N,\beta_{N},h_{N}}^{\hat\omega}
  =\sup_{u,v\geq 0} \Big\{  \hat \gb  (X_v^{(1)}+X_u^{(2)}) - \hat h (u+v) -\tfrac{1}{2} ( u\wedge v +u+v)^2 \Big\} 
\]
{in $\hat\bbP$-probability.}

\textbullet\ Line between regions $R_4$ and $R_5$: $\gamma = \frac{(2\alpha+1)\zeta-(\alpha-1)}{3\alpha}$ and $\zeta\in(-1,\frac{1}{2})$. Then one should have $\xi=\frac{1+\zeta}{3}$ and
\[ 
\lim_{N\to+\infty} \frac{1}{N^{1-2\xi}}\log Z_{N,\beta_{N},h_{N}}^{\hat\omega}
= \sup_{u,v\geq 0} \Big\{  \hat \gb  (X_v^{(1)}+ X_u^{(2)})  - \hat h (u+v) - \frac{\pi}{2(u+v)^2} \Big\} 
\]
{in $\hat\bbP$-probability,}
where the last term inside the supremum comes from the entropic cost of ``folding'' the random walk in the interval $[uN^\xi, v N^{\xi}]$ (see Lemma~\ref{lem:subdiffusive}).

\smallskip
\textbullet\ Line between regions $R_2$ and $R_3$: $\gamma = -\frac{\ga-1}{\ga}$ and $\zeta>0$. Then one should have $\xi=1$ and
\[
\lim_{N\to+\infty}
\frac{1}{N^{2\xi-1}}\log Z_{N,\beta_{N},h_{N}}^{\hat\omega}
= \sup_{u,v\geq 0} \Big\{  \hat \gb  (X_v^{(1)}+ X_u^{(2)}) -  \kappa(u\wedge v+u+v)   \Big\} 
\]
{in $\hat\bbP$-probability,}
where 
\begin{equation}\label{def:k}
\kappa(t):= \begin{cases}\frac12 (1+t) \log(1+t) + \frac12 (1-t) \log(1-t), \quad \text{for}\quad t\in [0,1]\\
+\infty \quad \text{for}\quad t>1,\end{cases}
\end{equation}  
 is the rate function for the large deviations of the simple random walk, see Lemma~\ref{lem:superdiffusive2}.

\smallskip
\textbullet\ Line between regions $R_3$ and $R_4$: $\gamma =\zeta -\frac{\ga-1}{\ga}$ and $\zeta<0$. Then one should have $\xi=1$ and
\[
\lim_{N\to+\infty} \frac{1}{N^{2\xi-1}}\log Z_{N,\beta_{N},h_{N}}^{\hat\omega}
=  \sup_{u,v\geq 0} \Big\{  \hat \gb  (X_v^{(1)}+ X_u^{(2)}) - \hat h (u+v) - \kappa(u\wedge v+u+v)   \Big\}
\]
{in $\hat\bbP$-probability.}

\smallskip
{\bf Comment 4.}
In region $R_5$, the disorder term does not appear in the variational formula. 
In the case $\gb=0$ and $h>0$ (\textit{i.e.}\ $\gamma=\infty$, $\zeta=0$), corresponding to the random walk penalized by its range in a homogeneous way, the behavior of the random walk is well understood:
it is confined in a segment of length $(\pi^{\frac23}\hat h^{-\frac13}) N^{1/3}$ with a random center,
see \cite{Schmock90}
for the continuum limit of the process.
In our model, we have shown that  trajectories are still confined in a segment of length $(\pi^{\frac23}\hat h^{-\frac13}) N^{1/3}$. However, disorder should appear in the fluctuations of the log-partition function
and in particular we believe that, depending on the strength $\gb_N$ of the disorder interaction,
the center of this segment should be determined by the environment so as to maximize the
amount of potentials in that segment; in particular, it should not be random anymore (under $\bP_{N,\gb_N,h_N}^{\go}$, for typical realizations of~$\go$).
This picture should hold in region $R_5$
as long as the effect of disorder is sufficiently strong.
More precisely,  using the terminology of Section~\ref{sec:heuristics},
the ``disorder'' term is $\hat \gb N^{\frac{\xi}{\ga} -\gamma}$,
with $\xi=\frac13 (1+\zeta)$: its effect does not vanish as long as $\gamma < \xi/\ga$, that is as long as $\gamma < \frac{1}{3\ga} (1+\zeta)$.
In other words, there should be another phase transition inside region $R_5$: the random walk is confined in a segment 
of length $(\pi^{\frac23}\hat h^{-\frac13}) N^{\xi}$ with $\xi = \frac13 (1+\zeta)$, but  under $\bP_{N,\gb_N,h_N}^{\go}$ the location of this segment should be
non-random (\textit{i.e.}~determined by  the realization of $\go$) when $\gamma < \frac{1}{3\ga} (1+\zeta)$,
and random when $\gamma > \frac{1}{3\ga} (1+\zeta)$ (which includes the case $\gb=0$).
We leave this as an open problem.

\section{Results in the case $\hat h <0$}
\label{rem:h<0}

\subsection{The phase diagram}

In the case $\hat h <0$, the same type of
``energy'' vs.\ ``range'' vs.\ ``entropy'' heuristics as in
Section~\ref{sec:heuristics} can be
carried out. The main difference is that the ``range'' term is now a reward rather than a penalty and thus plays in the same direction as the ``disorder'' term and encourages stretching:
the range size exponent will always verify $\xi\ge 1/2$.
Recall that for a polymer with typical range size $N^{\xi}$, the ``range'' term is of order $N^{\xi -\zeta}$,
the ``disorder'' term is of order $N^{\xi/\ga-\gamma}$
and the entropy term is $N^{2\xi-1}$ (since $\xi\geq 1/2$), as in \eqref{E107-1} and \eqref{E107}.
In a similar fashion than in Section~\ref{sec:heuristics},
we find that  two cases need to be considered.

\textit{Case I (``disorder"$\,\gg\,$``range").} 
As mentioned in Remark~\ref{rem:h}, regions $R_1,R_2,R_3$ are unchanged when $h<0$: we refer to~\eqref{E108}-\eqref{E109}-\eqref{E110} for the determination of $\xi$ in these three regions.

\textit{Case II (``disorder"$\,\ll\,$``range").} 
The balance between range and entropy is achieved if
 $\xi-\zeta =2\xi-1$ (with $\xi \in [\frac12,1]$), which gives
$\xi=1-\zeta$  when $\gamma>\frac{(2\ga-1)\zeta-(\ga-1)}{\ga}$.
Also, we have
$\xi=1$ when $\zeta\leq 0$ and $\gamma>\zeta - \frac{\ga-1}{\ga}$,
and we have
$\xi=1/2$ when $\zeta\geq 1/2$ and $\gamma>\frac{1}{2\ga}$.

\smallskip
To summarize, we can identify several regimes, according to the values of $\gamma,\zeta$:
there are five regimes when  $\ga \in (\frac12,2]$,
see Figure~\ref{F4} below;
there are four regimes when  $\ga \in (0,\frac12)$,
see Figure~\ref{F5} below.

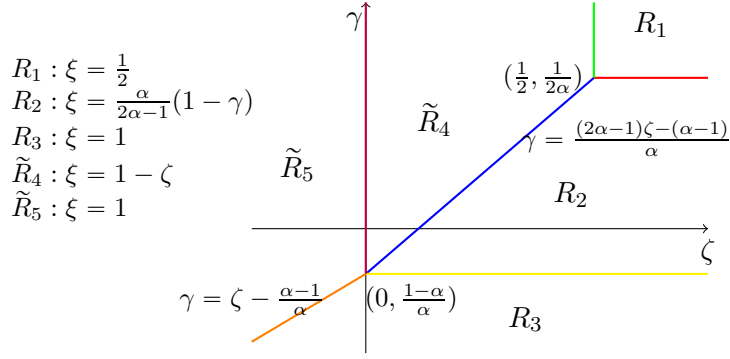
\begin{figure}[ht]
\centering
\begin{tikzpicture}[scale=3]
\draw[->](-0.5,0)--(1.5,0) node[anchor=north] {$\zeta$};
\draw[->](0,-0.55)--(0,1);
\draw (-0.05,1) node[anchor=north] {$\gamma$};
\draw[red,thick](1,2/3)--(1.5,2/3);
\draw[green,thick](1,2/3)--(1,1);
\draw (1,2/3) node[anchor=east] {\small $(\frac{1}{2},\frac{1}{2\ga})$};
\draw (1.25,1) node[anchor=north] {$R_{1}$};
\draw[blue,thick](0,-0.2)--(1,2/3);
\draw (1.65,0.4) node[anchor=east] {\small $\gamma=\frac{(2\ga-1)\zeta-(\ga-1)}{\ga}$};
\draw (0.9,0.25) node[anchor=north] {$R_{2}$};
\draw[yellow,thick](0,-0.2)--(1.5,-0.2);
\draw (0.2,-0.2) node[anchor=north] {\small $(0,\frac{1-\ga}{\ga})$};
\draw (0.7,-0.3) node[anchor=north] {$R_{3}$};
\draw[purple,thick](0,-0.2)--(0,1);
\draw (0.3,0.6) node[anchor=north] {$\tilde R_{4}$};
\draw (-0.3,0.4) node[anchor=north] {$\tilde R_{5}$};
\draw[orange,thick](-0.5,-0.5)--(0,-0.2);
\draw (-0.5,-0.2) node[anchor=north] {\small $\gamma=\zeta-\frac{\alpha-1}{\alpha}$};
\draw (-1.6,0.7) node[anchor=west] {\small $R_{1}:\xi=\frac{1}{2}$};
\draw (-1.6,0.55) node[anchor=west] {\small $R_{2}:\xi=\frac{\ga}{2\ga-1}(1-\gamma)$};
\draw (-1.6,0.4) node[anchor=west] {\small $R_{3}:\xi=1$};
\draw (-1.6,0.25) node[anchor=west] {\small $\tilde R_{4}:\xi=1-\zeta$};
\draw (-1.6,0.1) node[anchor=west] {\small $\tilde R_{5}:\xi=1$};
\end{tikzpicture}
\caption{Phase diagram for $\hat h<0$, in the case $\ga\in(1/2,2]$.
}\label{F4}
\end{figure}

\begin{figure}[ht]
\centering
\begin{tikzpicture}[scale=3]
\draw[->](-1,0)--(1.5,0) node[anchor=north] {$\zeta$};
\draw[->](0,-0.2)--(0,1);
\draw (-0.05,1) node[anchor=north] {$\gamma$};
\draw[red,thick](1,2/3)--(1.5,2/3);
\draw[green,thick](1,2/3)--(1,1);
\draw (1.1,0.65) node[anchor=north] {\small $(\frac{1}{2},\frac{1-\ga}{\ga})$};
\draw (1.2,0.95) node[anchor=north] {$R_{1}$};
\draw[yellow,thick](0,2/3)--(1,2/3);
\draw (0.2,2/3) node[anchor=north] {\small $(0,\frac{1-\ga}{\ga})$};
\draw (0.7,0.3) node[anchor=north] {$R_{3}$};
\draw[purple,thick](0,2/3)--(0,1);
\draw (0.5,0.95) node[anchor=north] {$\tilde R_{4}$};
\draw (-0.5,0.6) node[anchor=north] {$\tilde R_{5}$};
\draw[orange,thick](-1,-0.2)--(0,2/3);
\draw (-0.45,0.25) node[anchor=north] {\small $\gamma=\zeta-\frac{\alpha-1}{\alpha}$};
\draw (-1.7,0.7) node[anchor=west] {\small $R_{1}:\xi=\frac{1}{2}$};
\draw (-1.7,0.55)node[anchor=west] {\small $R_{3}:\xi=1$};
\draw (-1.7,0.4) node[anchor=west] {\small $\tilde R_{4}:\xi=1-\zeta$};
\draw (-1.7,.25)node[anchor=west] {\small $\tilde R_{5}:\xi=1$};
\end{tikzpicture}
\caption{Phase diagram for $\hat h<0$, in the case
 $\ga\in(0,1/2)$.
}\label{F5}
\end{figure}
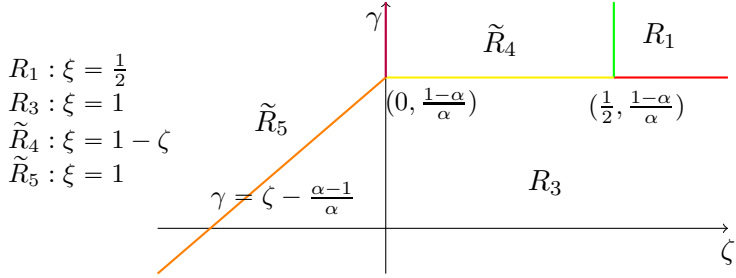

\subsection{Statement of the results}

We only state the results in regions~$\tilde R_4$ and $\tilde R_5$,
since the regions $R_1,R_2$ and $R_3$ are treated in Section~\ref{sec:results}, see Theorems~\ref{C101}, \ref{C102} and~\ref{C103} (respectively). 

\begin{theorem}[\bf Region~$\tilde R_4$]\label{thm:R4tilde}
Assume that \eqref{def:betah} holds  with $\hat \gb>0$, $\hat h <0$ and 
\[
\gamma > \frac{(2\ga-1) \zeta -(\ga-1)}{\ga} \vee \frac{1-\ga}{\ga} ,\qquad \zeta\in \big(0,\tfrac{1}{2}\big).
\]
Then $(S_n)_{0\leq n\leq N}$ has range size of order $N^{\xi}$ under $\bP_{N,\gb_N,h_N}^{\go}$ with  $\xi = 1-\zeta  \in (\frac12,1)$, and we have the following convergence
\begin{equation}\label{Etilde123}
\lim_{N\to+\infty} \frac{1}{N^{\xi-\zeta}} \log Z_{N,\beta_{N},h_{N}}^{\omega}
=  \frac{1}{2}  \hat h^2 = \sup_{u,v\geq 0} \Big\{ | \hat h| (v-u) - \tfrac{1}{2} ( u\wedge v +u+v)^2\Big\} 
\qquad \text{$ \bbP$-a.s.}
\end{equation}
Additionally,
let us consider, for $\gep>0$, the two events
\begin{equation*}
\label{eq:eventB+B-}
\cB_N^{+,\gep} :=\Big \{  \sup_{t\in [0,1]} \big| N^{-\xi} S_{\lfloor tN\rfloor} + \hat h\, t \big| \leq \gep \Big\} \,,\qquad
\cB_N^{-,\gep} :=\Big \{  \sup_{t\in [0,1]} \big| N^{-\xi} S_{\lfloor tN\rfloor} -\hat h\, t \big| \leq \gep \Big\} \,,
\end{equation*}
which corresponds to $(S_n)_{0\leq n \leq N}$ travelling with roughly constant speed to either $-\hat h N^{\xi}$ or $\hat h N^{\xi}$.
Then for any $\gep>0$, we have
\begin{equation*}
\lim_{N\to+\infty} \Big(  \bP_{N,\gb_N,h_N}^{\go} \big( \cB_N^{+,\gep} \big)  +  \bP_{N,\gb_N,h_N}^{\go} \big( \cB_N^{-,\gep} \big)  \Big) =1 
\qquad\text{$\bbP$-a.s.},
\end{equation*}
and $\hat \bbP$-a.s.
\begin{equation}
\label{eq:ballistic1}
\lim_{\gep\downarrow 0} \lim_{N\to+\infty} \bP_{N,\gb_N,h_N}^{\hat \go} \big( \cB_N^{+,\gep} \big) 
 =
 \begin{cases}
\ \ind_{\{X_{|\hat h|}^{(1)} > X_{|\hat h|}^{(2)}\}}  & \text{ if } \gamma <\frac{1-\zeta}{\ga} \,,\\[7pt]
  \frac{\exp\big( \hat \gb X_{|\hat h|}^{(1)} \big) }{ \exp\big( \hat \gb X_{|\hat h|}^{(1)} \big) + \exp\big( \hat \gb X_{|\hat h|}^{(1)} \big)  }  & \text{ if } \gamma =\frac{1-\zeta}{\ga} \,, \\[7pt]
\   \frac12 & \text{ if } \gamma >\frac{1-\zeta}{\ga} \,.
  \end{cases}
\end{equation}
\end{theorem}

Before we state the result in region~$\tilde R_5$ (which is somehow degenerate),
let us state a result in the case $\zeta=0$, that is at the boundary of regions~$\tilde R_4$ and $\tilde R_5$.
  
\begin{theorem}[\bf Boundary $\tilde R_4$---$\tilde R_5$]
\label{R4tilde}
Assume that \eqref{def:betah} holds with $\hat \gb >0$, $\hat h <0$ and with
$\zeta=0$, $\gamma>  -\frac{\ga-1}{\ga}$.
Then we have the following convergence
\begin{equation}
\label{Etilde124}
\lim_{N\to+\infty}
\frac{1}{N} \log Z_{N,\gb_N,h_N}^{\go}
=   \log \big( \sinh |\hat h| \big)   =  \sup_{ u,v\geq  0 } \left\{ |\hat h| (u+v) - \kappa( u\wedge v +u+v)  \right\} 
\qquad \text{$\bbP$-a.s.},
\end{equation}
with $\kappa(\cdot)$ defined in~\eqref{def:k}.
Additionally,
for $\gep>0$, let us consider the two events
\begin{equation*}
\cB_N^{+,\gep} :=\Big \{  \sup_{t\in [0,1]} \big| N^{-1} S_{\lfloor tN\rfloor} -\tanh(|\hat h|) t \big| \leq \gep \Big\} \,,\quad
\cB_N^{-,\gep} :=\Big \{  \sup_{t\in [0,1]} \big| N^{-1} S_{\lfloor tN\rfloor} + \tanh(|\hat h|) t \big| \leq \gep \Big\} \,,
\end{equation*}
which corresponds to $(S_n)_{0\leq n \leq N}$ travelling with roughly constant speed to either $\tanh(|\hat h|) N$ or $-\tanh(|\hat h|)N$.
Then for any $\gep>0$,  we have
\begin{equation*}
\lim_{N\to+\infty} \Big(  \bP_{N,\gb_N,h_N}^{\go} \big( \cB_N^{+,\gep} \big)  +  \bP_{N,\gb_N,h_N}^{\go} \big( \cB_N^{-,\gep} \big)  \Big) =1
\qquad\text{$\bbP$-a.s.},
\end{equation*}
and $\hat \bbP$-a.s.
\begin{equation}
\label{eq:ballistic2}
\lim_{\gep\downarrow 0} \lim_{N\to+\infty} \bP_{N,\gb_N,h_N}^{\hat \go} \big( \cB_N^{+,\gep} \big) 
 =
 \begin{cases}
\   \ind_{\{X_{\tanh|\hat h|}^{(1)} > X_{\tanh|\hat h|}^{(2)}\}}  & \text{ if } \gamma <\frac{1}{\ga} \,,\\[7pt]
\frac{\exp\big( \hat \gb X_{\tanh|\hat h|}^{(1)} \big) }{ \exp\big( \hat \gb X_{\tanh|\hat h|}^{(1)} \big) + \exp\big( \hat \gb X_{\tanh|\hat h|}^{(2)} \big)  }  & \text{ if } \gamma =\frac{1}{\ga} \,, \\[7pt]
\   \frac12 & \text{ if } \gamma >\frac{1}{\ga} \,.
  \end{cases}
\end{equation}
\end{theorem}

\noindent
To conclude, we state the result in region $\tilde R_5$.

\begin{theorem}[\bf Region~$\tilde R_5$]
\label{R5tilde}
Assume that \eqref{def:betah} holds with $\hat\gb >0$, $\hat h <0$ and
$\zeta<0$, $\gamma> \zeta -\frac{\ga-1}{\ga}$.
Then
\[
\lim_{N\to+\infty} N^{\zeta-1} \log Z_{N,\gb_N,h_N}^{\go} =|\hat h| \qquad \text{$ \bbP$-a.s.}
\]
Additionally,
for $\gep>0$, let us consider the two events
\[
\cB_N^{+,\gep} :=\Big \{  \sup_{t\in [0,1]} \big| N^{-1} S_{\lfloor tN\rfloor} -t \big| \leq \gep \Big\} \,,\quad
\cB_N^{-,\gep} :=\Big \{  \sup_{t\in [0,1]} \big| N^{-1} S_{\lfloor tN\rfloor} +t \big| \leq \gep \Big\} \,,
\]
which corresponds to $(S_n)_{0\leq n \leq N}$ travelling with roughly constant speed to either $N$ or~$-N$.
Then, for any $\gep>0$,   we have
\begin{equation*}
\lim_{N\to+\infty} \Big(  \bP_{N,\gb_N,h_N}^{\go} \big( \cB_N^{+,\gep} \big)  +  \bP_{N,\gb_N,h_N}^{\go} \big( \cB_N^{+,\gep} \big)  \Big) =1 
\qquad\text{$\bbP$-a.s.},
\end{equation*}
and $\hat \bbP$-a.s.
\begin{equation}
\label{eq:ballistic3}
\lim\limits_{\epsilon\downarrow0}\lim_{N\to+\infty} \bP_{N,\gb_N,h_N}^{\hat\go} ( \cB_N^{+,\gep} ) 
 =
 \begin{cases}
\   \ind_{\{X_{1}^{(1)} > X_{1}^{(2)}\}}  & \text{ if } \gamma <\frac{1}{\ga} \,,\\[7pt]
  \frac{\exp\big( \hat \gb X_{1}^{(1)} \big) }{ \exp\big( \hat \gb X_{1}^{(1)} \big) + \exp\big( \hat \gb X_{1}^{(2)} \big)  }  & \text{ if } \gamma =\frac{1}{\ga} \,, \\[7pt]
\  \frac12 & \text{ if } \gamma >\frac{1}{\ga} \,.
  \end{cases}
\end{equation}

\noindent
If $\ga\in(0, 1)$ or if $\ga\in(1,2]$ and $\gamma>\zeta$, then
we can upgrade the result:
we have
\begin{equation}\label{Rn=n}
\lim_{N\to+\infty}  \Big( \bP_{N,\gb_N,h_N}^{\go} \big( S_N =N \big) +\bP_{N,\gb_N,h_N}^{\go} \big( S_N = -N \big)  \Big)= 1 
\qquad
\text{$\bbP$-a.s.}
\end{equation}
and \eqref{eq:ballistic3} holds with $\{S_N=N\}$ in place of $B_N^{+,\gep}$.
\end{theorem}

\subsection{Further comments on the results  in the case $\hat h<0$}

{\bf Comment 5.}
Notice that in Theorems~\ref{thm:R4tilde}, \ref{R4tilde} and~\ref{R5tilde},
the disorder term disappears in the limiting variational problems
and the displacement of  $S_N$ under $\bP_{N,\gb_N,h_N}^{\hat{\go}}$
is given by a law of large number,  possibly with a random direction. 
Analogously to Comment~4 above, disorder should also appear in the fluctuations of the log-partition function and in the second order term for the displacement of $S_N$.
 For simplicity, let us comment further the case
of the boundary $\tilde R_4$---$\tilde R_5$ (that is Theorem~\ref{R4tilde}), \textit{i.e.}\ consider the case where  $h<0$ is fixed ($\zeta=0$) and 
$\gb_N = \hat \gb N^{-\gamma}$ with $\gamma > \frac{1-\ga}{\ga}$.
In that case, the polymer has a (non-random) velocity $v_h:=\tanh |h|$ either to the positive side or to the negative side:
assume for simplicity that $X_{v_h}^{(1)} > X_{v_h}^{(2)}$, so that $\frac1N S_N$ converges to $v_h$ (and not $-v_h$). 
Randomness should then have the effect of stretching further (or back) the polymer: let us present some heuristic explanation on what one should expect.
If we assume that $M_N^- = -u N^{\rho}$ and  $M_N^+ = v_h N + v N^{\rho}$, for some $\rho\in (\frac12,1)$ and $v\in \bbR,u\geq 0$, then, compared to the case $S_N=v_h  N$:

\begin{enumerate}
\item[(i)] there is an additional range reward $|\hat h| (u+v) N^{\rho}$;

\item[(ii)] there is an additional entropic cost 
 $\kappa'(v_h) (2u+v) N^{\rho} +\frac12 \kappa''(v_h) (2u+v)^2 N^{2\rho-1}$;

\item[(iii)] there is an energy gain of approximately $\hat \gb  (\check X_u^{(1)} +\check X_{v}^{(2)})N^{\frac{\rho}{\ga}-\gamma}$, with $\check X^{(1)},\check X^{(2)}$ the scaling limits of the fields 
$(\go_x)_{x\leq 0}, (\go_{v_h N+ x})_{x\in \bbZ}$ on a scale $N^{\rho}$.
\end{enumerate}

\noindent
In particular, there are some cancellations between
the range reward and the additional entropic cost:
one gets $-\kappa'(v_h) u N^{\rho} - \frac12 \kappa''(v_h) (2u+v)^2 N^{2\rho-1}$.
One should therefore take $u=0$:
if $u>0$ the term $-\kappa'(v_h) u N^{\rho}$ cannot be compensated by $\hat \gb \tilde X_u^{(1)} N^{\frac{\rho}{\ga}-\gamma}$, because $\gamma > \frac{1-\alpha}{\alpha}$.
Therefore, for the entropy-energy balance, one has to compare 
$\frac12 \kappa''(v_h) v^2 N^{2\rho-1}$ with $\hat \gb\check X_{v}^{(2)}N^{\frac{\rho}{\ga}-\gamma}$.

All together, this suggests that in the case $\gamma< \frac{1}{2\ga}$
(which is compatible with $\gamma > \frac{1-\ga}{\ga}$ only for $\ga \in (1,2]$) it is possible to take $\rho := \frac{(1-\gamma)\alpha}{2\alpha-1} \in (\frac12,1)$, so that $N^{2\rho-1}\asymp N^{\frac{\rho}{\ga}-\gamma}$.
Then, under $\bP_{N,\gb_N,h_N}^{\hat{\go}}$,
one should have 
$M_N^- =o(N^\rho)$ and $M_N^+ = v_h N- (1+o(1))\cV N^\rho$,
where~$\cV$ is the maximizer of the variational problem 
\[
\sup_{v\in \bbR} \left\{ \hat \gb  \check X_{v}^{(2)}  - \tfrac12 \kappa''(v_h) v^2 \right\} \,.
\]
On the other hand, when $\gamma >\frac{1}{2\ga}$,
then there is no further stretching of the polymer by the disorder:
under $\bP_{N,\gb_N,h_N}^{\hat{\go}}$ we should have $S_N = v_h N + O(\sqrt{N})$, as it is the case when $\gb_N\equiv 0$.
This goes beyond the scope of this article and we leave this as an open problem.

\smallskip
{\bf Comment 6.}
There is some room for improvement in Theorem \ref{R5tilde},
when $\ga\in(1,2]$ and $\zeta\geq\gamma>\zeta-\frac{\ga-1}{\ga}$. Indeed, we then have that $\frac1N S_N$ goes to $\pm 1$, but as in Comment~5 above, disorder should appear in the fluctuations
of the log-partition function and in the second order term for the displacement of $S_N$.
Let us assume that we are on the event where $\frac1N S_N$ goes to $+ 1$ (instead of $-1$)
and let us present a heuristic explanation on what one should expect.
If we assume that $M_N^-=-uN^\rho$ and $M_N^+=N-vN^\rho$, for some $\rho\in(0,1)$ and $u, v\geq0$ with $2u\leq v$ (because of the constraint $M_N^+-2M_N^-\leq N$),
then compared to the case $S_N=N$:

\begin{enumerate}
\item[(i)] there is a diminution of the range reward by $|\hat h| (v-u) N^{\rho-\zeta}$;

\item[(ii)] there is a reduction of the entropic cost by roughly $N^{\rho} \log N$,
coming from the combinatorial term $\binom{N}{N-cN^{\rho}}$ --- this is negligible compared to the range term since $\zeta<0$;

\item[(iii)] there is an energy gain of approximately $\hat{\beta} (\check{X}_{u}^{(2)}-\check{X}_{v}^{(1)})N^{\frac{\rho}{\ga}-\gamma}$,
where $\check{X}_{v}^{(1)},\check{X}_v^{(2)}$ are the scaling limits of
$(\go_x)_{x\leq 0}, (\go_{N-x})_{x\geq 0}$ on a scale $N^{\rho}$.

\end{enumerate}

\noindent
This suggests that for $\ga\in(1,2]$ and $\zeta>\gamma>\zeta-\frac{\ga-1}{\ga}$,
one should take $\rho=\frac{\ga}{\ga-1}(\zeta-\gamma) \in(0,1)$,
so that $N^{\rho-\zeta}\asymp N^{\frac{\rho}{\ga}-\gamma}$.
Note that one recovers  $\rho=0$ if $\gamma=\zeta$ (to be compared with~\eqref{Rn=n}) and $\rho=1$ if $\gamma=\zeta-\frac{\ga-1}{\ga}$ (\textit{i.e.}\ on the boundary of Regions $\tilde R_5-R_3$, see Figure~\ref{F4}).
Additionally, under $\bP_{N,\gb_N,h_N}^{\hat{\go}}$,
one should have $M_N^-=-(1+o(1))\cU N^\rho$ and $M_N^+= N-(1+o(1))\cV N^\rho$, where
$(\cU,\cV)$ are the maximizers of the variational problem
\begin{equation*}
\sup_{0\leq 2u\leq v}\left\{\hat{\beta}\left(\check{X}_{u}^{(2)}-\check{X}_v^{(1)}\right)-|\hat{h}|(v-u)\right\}.
\end{equation*} 
As for Comment~5, we leave this as an open problem.

\subsection{Organisation of the proof and useful notation}

\noindent
Let us give an overview of how the rest of the paper is organized:

\smallskip
\textbullet\ In Section~\ref{sec:proof1}, we start with the proof of Proposition \ref{prop:ballistic}, then we prove Theorems~\ref{C101} to \ref{C105} (in that order), \textit{i.e.}\ we prove the phase diagram of Regions~$R_1$ to~$R_6$; note that regions~$R_4$ to~$R_6$ are specific to the case $\hat h >0$.
The results in Regions $R_2$, $R_4$ and $R_5$ involve competitions between ``energy", ``range" or ``entropy"
(but all have the same scheme of proof), while Regions $R_1$, $R_3$ and $R_6$ are extreme cases where only one factor is significant and hence are much simpler.
Let us stress here that the statements on range size of trajectories $(S_n)_{0\leq n\leq N}$ under $\bP_{N,\gb_N,h_N}^{\go}$ are direct consequences of the convergence of $(M_N^-,M_N^+)$
under $\bP_{N,\gb_N,h_N}^{\go}$, so we do not write their proof explicitly.

\smallskip
\textbullet\ In Section~\ref{sec:proof2}, we prove the remaining Theorems~\ref{thm:R4tilde} to \ref{R5tilde}, \textit{i.e.}\ we complete the phase diagram in the case $\hat h<0$.
Here, the main contribution to the partition function comes from the range term and finding the limit of the rescaled log-partition function is not difficult. The harder part consists in showing that disorder plays a role in deciding whether the random walk moves to the positive or to the negative side: this is done by a careful decomposition of the partition function.

\smallskip
\textbullet\ In Appendix~\ref{sec:app}, we regroup several technical estimates: large deviations for the range of the random walk in Section~\ref{sec:appLD},
deviation for sums of $\go_x$ (\textit{i.e.}\ the proof of Lemma~\ref{lem1} below) in Section~\ref{app:lemma},
the proof of Proposition~\ref{uniquemaximizer} in Section~\ref{sec:unique} and some technical estimate on c\`adl\`ag path in Section~\ref{sec:cadlag}.

\paragraph*{Some further notation and a useful lemma.}
In the rest of the paper, to lighten notation, we will drop the dependence on $\gb_N$ and~$h_N$:
we write  $\bP_N^{\go}$ instead of $\bP_{N,\gb_N,h_N}^{\go}$ and  $Z_N^{\go}$ instead of $Z_{N,\gb_N,h_N}^{\go}$.
We also use the convenient notation
$Z_{N}^{\go}(E)$ for the partition function restricted to trajectories $(S_n)_{n\geq 0}$ in~$E$:
more precisely, 
\begin{equation}
\label{def:PartRestrict}
Z_N^{\go} (E) := \bE\Big[   \exp\Big(  \sum_{x\in \bbZ^d} (\gb_N \go_x - h_N) \ind_{\{x\in \mathcal{R}_N\}} \Big)  \ind_E \Big]\, .
\end{equation}
This way, we have that $\bP_{N}^{\go}(E) = Z_{N}^{\go}(E) /Z_N^{\go}$.

For any $j\ge 0$ let us recall the notation $\Sigma_j^{+}:= \sum_{x=0}^{j} \go_x$ and $\Sigma_{j}^{-} := \sum_{x=-j}^{-1} \go_x$, introduced in \eqref{eq:DefSigma} 
(with the convention that $\Sigma_{0}^{-}=0$).
We then let
\begin{equation}
\label{def:Omegasup}
\Sigma_\ell^*:= \sup_{0\leq j \leq \ell } |\Sigma_j^-| +  \sup_{0\leq j \leq \ell } |\Sigma_j^+| \, .
\end{equation}

Recall that we have set $M_N^+ :=\max_{0 \leq n\leq N} S_n \geq 0$ and $M_N^{-} :=  \min_{0\leq n\leq N} S_n \leq 0$
the right-most and left-most points of the random walk after $N$ steps;
we also denote
\[
M_N^* := \max_{0 \leq n\leq N} |S_n| = \max(M_N^+, -M_N^{-}) \,.
\]
With these notation, notice that we have $\sum_{x\in \cR_N} \go_x  = \Sigma_{M_N^{+}}^{+} + \Sigma_{-M_N^{-}}^{-}$. We now state the following (standard) lemma, that we prove in Appendix~\ref{app:lemma} for completeness.

\begin{lemma}\label{lem1}
Let $\Sigma_\ell^*$ defined as in \eqref{def:Omegasup}. 
Then, under Assumption~\ref{assump1} ($\alpha\in (0,1)\cup(1,2]$), there exists a constant $c \in (1,+\infty)$ such that
 for any $\T>0$ and any $\ell$ we have
\begin{equation}
\bbP \big(\Sigma_\ell^* >\T \big) \le c\,  \ell \, \T^{-\alpha} \, .
\end{equation}
Also, $\bbP$-a.s.\ 
there is a constant $C=C(\omega)$ such that $\Sigma_\ell^* \leq C \ell^{1/\alpha} (\log_2 \ell)^{2/\alpha}$ for all $\ell\geq 1$.
\end{lemma}

\noindent
Finally,
while we keep the distinction between $\bbP$ and $\hat \bbP$,
we will only write $\omega$ (and not $\hat \omega$) in order to lighten the notation,

\section{Proof of the main results}
\label{sec:proof1}

\subsection{Ballisticity of trajectories: proof of Proposition~\ref{prop:ballistic}}

Let $\gep>0$ be fixed.
For $\delta>0$, let us define (recall we assume $\xi>\frac12$)
\[
\cA_N^{\delta} := \bigg\{ \Big| \frac{1}{N^{\xi}} (M_N^-, M_N^+ )   -(-\cU, \cV)\Big| \leq \gd   \bigg\} \,.
\]
Then, by assumption, we have that for any $\gd>0$,
$\lim_{N\to+\infty} \frac{Z_N^{\go} ( \cA_N^{\delta} )  }{Z_N^{\go}} =1$
in $\hat \bbP$-probability (resp.\ $\hat\bbP$-a.s.),
so the proof will be complete if we show that
\begin{equation}
\label{eq:ballistictoshow}
\lim_{N\to+\infty} \frac{Z_N^{\go}\big( \cA_N^{\delta} , \cB_N^{\gep} (\cU,\cV)^c \big) }{Z_N^{\go} \big( \cA_N^{\delta} \big) } =0 \,, \qquad \text{ $\hat\bbP$-a.s.}
\end{equation}
where we refer to \eqref{eq:ball} for the definition of $\cB_N^{\gep} (\cU,\cV)$.

To do so, we decompose the partition function as follows (here and in the rest of the paper, we often omit integer parts for simplicity):
\begin{multline*}
Z_N^{\go}\big( \cA_N^{\delta} , \cB_N^{\gep} (\cU,\cV)^c \big)
  =\sum_{|(x,y)-N^{\xi}(\cU,\cV)| \leq \delta N^\xi}
Z_N^{\go}\big( M_N^{-}=-x, M_N^+ =y , \cB_N^{\gep} (\cU,\cV)^c \big)  \\
 = \sum_{|(x,y)-N^{\xi}(\cU,\cV)| \leq \delta N^\xi} e^{\gb_N (\Sigma_x^{-} +\Sigma_y^+) - h_N (x+y+1)}
\bP\big( M_N^{-}=-x, M_N^+ =y , \cB_N^{\gep} (\cU,\cV)^c \big) 
\end{multline*}
Now, for $\gd>0$ small enough,
we have thanks to~\eqref{ldp:localballistic}
(recall $\cU,\cV\geq 0$ with $\cU\neq \cV$) that
\begin{equation}
\label{eq:locallargedev}
\bP\big( M_N^{-}=-x, M_N^+ =y , \cB_N^{\gep} (\cU,\cV)^c \big) 
 \leq e^{- c\, N^{2\xi-1}}
 \bP\big(M_N^{-}=-x, M_N^+ =y\big)
\end{equation}
uniformly for  $x,y$ such that  $|(x,y)-N^{\xi}(\cU,\cV)| \leq \delta N^\xi$,
for some constant $c = c_{\gd,\gep} (\cU,\cV) >0$.

Using~\eqref{eq:locallargedev},
we obtain that
\begin{align*}
Z_N^{\go}\big( \cA_N^{\delta} &, \cB_N^{\gep} (\cU,\cV)^c \big)\\
 & \leq e^{- c\, N^{2\xi-1}}  \sum_{|(x,y)-N^{\xi}(\cU,\cV)| \leq \delta N^\xi} e^{\gb_N (\Sigma_x^{-} +\Sigma_y^+) - h_N (x+y+1)}
\bP\big( M_N^{-}=-x, M_N^+ =y \big) \\
&  \leq e^{- c\, N^{2\xi-1}}  \sum_{ |(x,y)-N^{\xi}(\cU,\cV)| \leq \delta N^\xi} Z_{N}^{\go} \big( M_N^{-}=-x, M_N^+ =y  \big)  \leq e^{- c\, N^{2\xi-1}}  Z_N^{\go}\big( \cA_N^{\delta} \big) \,, 
\end{align*}
which shows~\eqref{eq:ballistictoshow}.
\qed

\subsection{Region R1: Proof of Theorem \ref{C101}}\label{sec:reg1}
 Recall that in Region $R_1$ we have
\[
\begin{cases} \gamma>\frac{1}{2\ga} \quad \text{and} \quad \zeta>\frac12, &\quad \text{if}\quad \alpha \in  [\frac{1}{2},1)\cup(1,2], \\
\gamma>\frac{1-\alpha}{\ga} \quad \text{and} \quad \zeta>\frac12, &\quad \text{if}\quad \alpha \in (0,\frac{1}{2}).
\end{cases}
\]
Let us note that we always have $\gamma > \frac{1}{2\alpha}$, since $\frac{1-\ga}{\ga} >\frac{1}{2\ga}$ when $\ga<1/2$.

\subsubsection*{Convergence of the partition function}
Fix $A$ (large) and split  the partition function in the following way
\begin{equation}\label{E201}
Z_{N}^{\omega}=Z_{N}^{\omega}\big( M_N^* \leq A\sqrt{N}\big)+Z_{N}^{\omega}\big(M_N^*>A\sqrt{N}\big).
\end{equation}

\paragraph*{Upper bound.}
 Recalling the definition~\eqref{def:PartRestrict} of the restricted partition function, one easily sees that
\begin{equation}\label{E202}
\begin{split}
Z_{N}^{\omega}\big( M_N^* \leq A\sqrt{N}\big)
&\leq\exp\left(\hat{\beta}N^{-\gamma} \Sigma_{A\sqrt{N}}^* + 2A |\hat h| N^{\frac12 -\zeta} \right)\mathbf{P}\big(M_N^*\leq A\sqrt{N}\big)  \\
& \leq \exp\left( \hat{\beta}N^{-\gamma} \Sigma_{A\sqrt{N}}^* + 2 A |\hat h| N^{\frac12 -\zeta} \right)
\end{split}
\end{equation}
Notice that $N^{\frac12 -\zeta}$ goes to $0$ as $N\to\infty$, since $\zeta>\frac12$.
Also,  by Lemma \ref{lem1}, since $\gamma > \frac{1}{2\alpha}$ we get that 
$N^{-\gamma} \Sigma_{A\sqrt{N}}^*$ goes to $0$ almost surely.
We therefore get that 
$\limsup_{N\to\infty} Z_{N}^{\omega}( M_N^* \leq A\sqrt{N}) \leq 1$ $\bbP$-a.s.

\smallskip 
It remains to show that  the second term in \eqref{E201} is also small.
 We split the 
partition function as 
\begin{align*}
Z_{N}^{\omega}\Big(M_N^*>A\sqrt{N}\Big) & \le\sum\limits_{k=1}^{ \log_2(\frac{1}{A}\sqrt N ) }Z_{N}^{\omega}\Big(M_N^*\in(2^{k-1}A\sqrt{N},2^{k}A\sqrt{N}]\Big) \notag\\
&\leq \sum\limits_{k=1}^{ \log_2(\frac{1}{A}\sqrt N ) } \exp\Big(  \hat \gb N^{-\gamma} \Sigma^*_{2^k A\sqrt{N}}  + 2^{k+1} A |\hat h| N^{\frac12-\zeta}\Big) \bP\big( M_N^* \geq 2^{k-1} A \sqrt{N}  \big) \,.
\end{align*}
Then, it is standard to get that $\bP(M_N^* >x) \leq 2  \exp( -\frac{x^2}{2N})$
 for any $x>0$ and $N\in \bbN$ (thanks to L\'evy's inequality and a standard Chernov bound), so that
\begin{equation}\label{E208}
\bbP\big(M_N^*>2^{k-1}A\sqrt{N}\big)\leq 2 \exp\left(-2^{2k-3}A^{2}\right) \, .
\end{equation}
Hence, choosing $N$ large enough so that $|\hat h| N^{\frac12-\zeta} \leq 2^{-5} A$ and in particular $ 2^{k+1} A |\hat h| N^{\frac12-\zeta} \leq 2^{2k-4} A^2$, we get $\bbP$-a.s.
\begin{equation}
\label{sumkboundZ}
Z_{N}^{\omega}\Big(M_N^*>A\sqrt{N}\Big) \leq \sum\limits_{k=1}^{ \log_2(\frac{1}{A}\sqrt N ) } 2 \exp\Big(  C(\hat \gb,A,\go) 2^{k/\alpha} N^{\frac{1}{2\ga}-\gamma} (\log_2 N)^{2/\alpha}  -2^{2k-4} A^2 \Big)\, ,
\end{equation}
where we also used Lemma~\ref{lem1}
to get an almost sure bound on $\Sigma^*_{2^k A\sqrt{N}}$.
Now, note that uniformly for $k$ in the sum,
\[
\frac{2^{k/\alpha} N^{\frac{1}{2\ga}-\gamma} (\log_2 N)^{2/\alpha} }{2^{2k} }
\leq  (\log_2 N)^{2/\alpha} \times
\begin{cases}
N^{\frac{1}{2\ga}-\gamma} & \quad \text{ if } \alpha \in [\frac12,2]\,, \\
A^{- \frac{1-2\alpha}{2}} N^{\frac{1}{\ga}-1-\gamma}  & \quad \text{ if } \alpha \in (0,\frac12) \,.
\end{cases}
\]
This (uniform) upper bound goes to $0$ as $N\to\infty$,
because $\gamma>\frac{1}{2\alpha}$ if $\ga\geq \frac12$
and $\gamma > \frac{1-\alpha}{\alpha}$ if $\ga\in (0,\frac12)$.

All together, we get that $\bbP$-a.s., for $N$ large enough,
\[
Z_{N}^{\omega}\Big(M_N^*>A\sqrt{N}\Big)
\leq \sum\limits_{k=1}^{ \log_2(\frac{1}{A}\sqrt N ) } 2 \exp\Big(   -2^{2k-5} A^2 \Big) \leq C \exp( - A^2/8) \,.
\]
We have therefore proven that
for any $A>0$, 
$\limsup_{N\to\infty} Z_{N}^{\go} \leq 1 + Ce^{-A^2/8}$.
Since $A$ is arbitrary,
this gives that $\limsup_{N\to\infty} Z_{N}^{\go} \leq 1$.

\paragraph*{Lower bound.}
For the lower bound, we use that 
\begin{equation}\label{E203}
Z_N^{\go} \geq Z_{N}^{\omega}(M_N^*\leq A\sqrt{N}) \geq \exp\left(-\hat{\beta}N^{-\gamma} \Sigma^*_{A\sqrt{N}}-2A|\hat{h}| N^{\frac{1}{2}-\zeta}\right)\mathbf{P}\left(M_N^*\leq A\sqrt{N}\right).
\end{equation}
As above, we get that $\hat{\beta}N^{-\gamma} \Sigma^*_{A\sqrt{N}}+2A|\hat{h}| N^{\frac{1}{2}-\zeta}$ goes to $0$ almost surely.
Using that $\mathbf{P}(M_N^*\leq A\sqrt{N}) \geq 1-2 e^{- A^{2}/2}$,
we therefore get that 
$\liminf_{N\to\infty} Z_{N}^{\go} \geq 1 -2e^{-A^2/2}$ $\bbP$-a.s.
Since $A$ is arbitrary,
this gives $\liminf_{N\to\infty} Z_{N}^{\go} \geq 1$ $\bbP$-a.s.,
which concludes the proof.
\qed

\subsection*{Convergence in total variation distance}

We show that for any $\epsilon\in(0,\frac{1}{8})$,
\begin{equation}\label{TVdistance}
\limsup\limits_{N\to\infty}\sup\limits_{\cB}\big|\bP_N^{\go}(\cB)-\bP(\cB)\big|<5\epsilon \qquad \text{$\bbP$-a.s.},
\end{equation}
where $\cB$ ranges over all $\bP$-measurable sets. This implies the convergence from $\bP_N^\omega$ to $\bP$ in total variation distance, since $\epsilon$ is arbitrary.

Let $\cC_{N,\gep}:=\{\go: |Z_N^\go-1|<\epsilon\}$: we have proven above that $\bbP$-a.s. $\lim_{N\to\infty}\ind_{\cC_{N,\gep}}=1$. Note that $\bP_N^{\go}(\cB)=Z_N^\go(\cB)/Z_N^\go$. Hence, on the event $\cC_{N,\epsilon}$, we have 
\begin{equation*}
\frac{1}{1+\epsilon}\big(Z_N^{\go}(\cB)-\bP(\cB)\big)-\frac{\epsilon}{1+\epsilon}<\bP_N^\go(\cB)-\bP(\cB)<\frac{1}{1-\epsilon}\big(Z_N^{\go}(\cB)-\bP(\cB)\big)+\frac{\gep}{1-\gep},
\end{equation*}
where we also used that $\bP(\cB)\leq1$.
Therefore, to prove \eqref{TVdistance}, we need to show that
$\limsup_{N\to\infty}\sup_{\cB}|Z_N^\go(\cB)-\bP(\cB)|<\epsilon, \bbP$-a.s..
For $A>0$, we have that
\begin{multline*}
|Z_N^{\go}(\cB)-\bP(\cB)|\leq\Big|Z_N^\omega\big(\cB\cap\{M_N^*\leq A\sqrt{N}\}\big)-\bP\big(\cB\cap\{M_N^*\leq A\sqrt{N}\}\big)\Big|\\
+Z_N^\go\big(M_N^*> A\sqrt{N}\big)+\bP\big(M_N^*> A\sqrt{N}\big).
\end{multline*}
As seen above, we have  $\limsup_{N\to\infty} Z_N^\go(M_N^*> A\sqrt{N}) \leq Ce^{-A^2/8}$ almost surely, and  also $\bP(M_N^*> A\sqrt{N})\leq 2 e^{-A^2/2}$:
these two terms can be made arbitrarily small by taking~$A$ large.
Hence it is enough to show that for any $A>0$
\begin{equation*}
\limsup_{N\to\infty} \sup\limits_{\cB}\Big|Z_N^\go\big(\cB\cap\{M_N^*\leq A\sqrt{N}\}\big)-\bP\big(\cB\cap\{M_N^*\leq A\sqrt{N}\}\big)\Big| =0 \qquad \text{$\bbP$-a.s.}.
\end{equation*}
Analogously to \eqref{E202} and \eqref{E203} we have,
for any measurable $\cB$
\begin{equation*}
\begin{split}
& Z_N^\go\big(\cB\cap\{M_N^*\leq A\sqrt{N}\}\big) \leq\exp\left(\hat{\beta}N^{-\gamma} \Sigma_{A\sqrt{N}}^*+ 2A|\hat{h}| N^{\frac{1}{2}-\zeta}\right)\mathbf{P}\big(\cB\cap\{M_N^*\leq A\sqrt{N}\}\big),\\
& Z_N^\go\big(\cB\cap\{M_N^*\leq A\sqrt{N}\}\big) \geq \exp\left(-\hat{\beta}N^{-\gamma} \Sigma^*_{A\sqrt{N}}-2A|\hat{h}| N^{\frac{1}{2}-\zeta}\right)\mathbf{P}\big(\cB\cap\{M_N^*\leq A\sqrt{N}\}\big) \, ,
\end{split}
\end{equation*}
so $\big|Z_N^\go (\cB\cap\{M_N^*\leq A\sqrt{N}\})-\bP(\cB\cap\{M_N^*\leq A\sqrt{N}\})\big|$
is bounded by
\[
\left|\exp\left(\hat{\beta}N^{-\gamma} \Sigma_{A\sqrt{N}}^*+ 2A|\hat{h}| N^{\frac{1}{2}-\zeta}\right)-1\right| + \left|\exp\left(-\hat{\beta}N^{-\gamma} \Sigma^*_{A\sqrt{N}}-2A|\hat{h}| N^{\frac{1}{2}-\zeta}\right)-1\right| \,,
\]
where we also bounded $\bP(\cB\cap\{M_N^*\leq A\sqrt{N}\})$
by $1$.
Since $\hat{\beta}N^{-\gamma} \Sigma_{A\sqrt{N}}^*+ 2A|\hat{h}| N^{\frac{1}{2}-\zeta}$ goes to $0$ almost surely (see~Lemma~\ref{lem1}),
this concludes the proof.
\qed

\subsection{Region R2: proof of Theorem \ref{C102}}
Recall that in Region $R_2$ we have 
\[
2\xi-1 = \frac{\xi}{\alpha} -\gamma > \xi-\zeta \quad  \text{with} \ \ \alpha \in (\tfrac12,1)\cup(1,2] \, ,
\]
and that this region does not exist when $\ga<1/2$.
We prove that the range size is of order~$N^{\xi}$ with $\xi = \tfrac{\ga}{2\ga-1}(1-\gamma)$.

\subsubsection*{Convergence of the rescaled log-partition function}

We fix some $A$ large and we split the partition function as
\begin{equation}\label{E301}
Z_{N}^{\omega}=Z_{N}^{\omega}\left(M_N^*\leq AN^{\xi}\right)+Z_{N}^{\omega}\left(M_N^*> AN^{\xi}\right).
\end{equation}

The proof of the convergence is divided into three steps: 
(1) we show that after taking logarithm and dividing by $N^{2\xi-1}$, the first term converges to some random variable $\mathcal{W}_{2}^A$ as $N \to \infty$;
(2) we show that the second term is small compared to the first one;
(3) we let $A\to\infty$ and we observe that $\mathcal{W}_{2}^A$ converges to~$\cW_{2}$.

\paragraph*{Step 1.}
We prove the following lemma.

\begin{lemma}
\label{lem:R2step1}
In Region~$R_2$, we have that $\hat\bbP$-a.s., for any $A \in \bbN$, 
\[
\lim_{N\to\infty} \frac{1}{N^{2\xi-1}}  \log Z_{N}^{\omega}\left(M_N^*\leq AN^{\xi}\right)  =  \mathcal{W}_{2}^A:=\sup\limits_{ u,v \in[0,A] }\left\{\hat{\beta}(X_{v}^{(1)} +X_{u}^{(2)}) - I(u,v)  \right\} \, ,
\]
with $(X_{v}^{(1)},X_{u}^{(2)})_{u,v\geq 0}$ from Notation~\ref{defL} and $I(u,v) :=\frac12 (u\wedge v +u+v)^2$.
\end{lemma}
\begin{proof}
Let us fix $\delta>0$ and write 
\begin{equation}\label{E302}
Z_N^{\go,\leq } := Z_{N}^{\omega}\big(M_N^*\leq AN^{\xi}\big) =\sum\limits_{k_{1}=0}^{\lfloor A/\gd \rfloor} \sum\limits_{k_{2}=0}^{\lfloor A/\gd\rfloor}Z_{N}^{\omega}(k_{1},k_{2},\delta),
\end{equation}
where we define
\begin{equation}\label{E303}
Z_{N}^{\omega}(k_{1},k_{2},\delta):=Z_{N}^{\omega}\Big(M_N^-\in( -(k_{1}+1)\delta N^{\xi},-k_{1}\delta N^{\xi}],M_N^+\in [k_{2}\delta N^{\xi},(k_{2}+1)\delta N^{\xi})\Big)
\end{equation}
(recall $M_N^- :=\min_{0\leq n\leq N} S_n$ and $M_N^+ :=\max_{0\leq n\leq N} S_n$).
Since there are at most~$(A/\gd)^2$ terms in the sum, 
we easily get that 
\begin{equation}
\max_{0\leq k_1, k_2 \leq \frac{A}{\gd}}  \log  Z_{N}^{\omega}(k_{1},k_{2},\delta)  \leq  \log Z_{N}^{\go,\leq }  \leq   2\log (\tfrac{A}{\gd})  + \max_{0\leq k_1, k_2 \leq \frac{A}{\gd}}   \log  Z_{N}^{\omega}(k_{1},k_{2},\delta) \, .
\end{equation}

\smallskip
\noindent
{\it Upper bound.}
An upper bound on $\log Z_{N}^{\omega}(k_{1},k_{2},\delta)$
is given by
\begin{equation*}
 \gb_N \Big( \Sigma_{ \lfloor k_{1}\delta N^{\xi}\rfloor}^- + \Sigma_{ \lfloor k_{2}\delta N^{\xi}\rfloor}^+\Big)  + \gb_N R_{N}^{\gd}(k_{1} \gd,k_{2}\delta) + |\hat{h}|(k_1+k_2+2)\delta N^{\xi-\zeta} + p_{N}^{\gd}(k_{1}\gd,k_{2}\gd),
\end{equation*}
where for $u,v\geq 0$ we defined
\begin{equation}\label{E306}
R_{N}^{\gd}(u,v) :=\max\limits_{ uN^{\xi}\leq j \leq (u+\delta) N^{\xi}-1}  \Big| \Sigma_j^{-} - \Sigma_{\lfloor u N^{\xi} \rfloor }^{-}  \Big|+\max\limits_{v N^{\xi}\leq j \leq (v+\delta) N^{\xi}-1} \Big|\Sigma_j^{+} - \Sigma_{\lfloor v N^{\xi} \rfloor}^{+}   \Big| \, ,
\end{equation}
and
\begin{equation}\label{E307}
p_{N}^{\gd}(u,v):= \log \mathbf{P}\Big(M_N^- \in \big( -(u+\delta) N^{\xi}, -u N^{\xi} \big],M_N^+\in \big[ v N^{\xi},(v+\gd) N^{\xi} \big)\Big).
\end{equation}


\noindent
Let us write $u=k_{1}\delta, v=k_{2}\delta$ and set $ U_{\gd}=U_\delta(A)=\{0,\delta,2\gd,\ldots,A\}$:
using that $2\xi-1 =\xi/\ga-\gamma$,
we get that
\begin{equation}
\label{E308}
\begin{split}
\max_{0\leq k_1, k_2 \leq \frac{A}{\gd}}  \frac{ \log Z_{N}^{\omega} (k_1,k_2,\gd) }{N^{2\xi-1}} \leq  & \max_{u,v \in U_{\gd}}  \Big\{   \hat \gb  N^{-\frac{\xi}{\ga}}\big(  \Sigma_{\lfloor u N^{\xi}\rfloor}^{-} +  \Sigma_{\lfloor v N^{\xi}\rfloor}^{+} \big) + \hat\gb N^{-\frac{\xi}{\ga}} R_{N}^{\gd}(u,v)\\
&\quad  +|\hat{h}|(u+v+2\delta)N^{(\xi-\zeta)-(2\xi-1)}  + N^{-(2\xi-1)} p_N^{\gd}(u,v) \Big\} \,.
 \end{split}
\end{equation}
 It is easy to see that the third term in the maximum goes to $0$ uniformly in $u, v$, since $u+v+2\delta<3A$ and since $\xi-\zeta<2\xi-1$ in Region $R_2$. Note that  thanks to the coupling introduced in Section~\ref{sec:setting} we have that $ ( N^{-\xi/\ga}\Sigma_{\lfloor u N^{\xi} \rfloor}^{-})_{u\in[0,A+\gd]}$ and $( N^{-\xi/\ga}\Sigma_{\lfloor v N^{\xi} \rfloor}^{+})_{v\in [0,A+\gd]}$
converge $\hat\bbP$-a.s.\ in the Skorokhod topology to two independent L\'evy processes
 $(X_u^{(2)})_{u\in [0,A+\gd]}$ and $(X_v^{(1)})_{v\in [0,A+\gd]}$
 (of Notation~\ref{defL}).

Note also that thanks to Lemma~\ref{lem:superdiffusive} (see~\eqref{AE01})
we have 
\[
\lim_{N\to\infty}  N^{-(2\xi-1)} p_N(u,v,\gd)  = - I(u,v) \, , \quad  \text{with }\  I(u,v):= \frac12 (u\wedge v + u+ v)^2\, , \ \ u,v\geq 0\,.
\]

Since $U_\delta$ is a finite set, by (1.10) in \cite{Sato99}, the limiting L\'evy processes $X^{(1)}_v$ and $X^{(2)}_u$ are $\hat\bbP$-a.s.\ continuous at every point in $U_\delta$.
Hence, thanks to Lemma \ref{cadlag},
$\hat\bbP$-a.s., for any $\gep>0$ there is 
 a random integer $N_0 = N_0(\gep,\gd,\omega)$, such that for all $N\geq N_0$,
\begin{align*}
& N^{-\frac{\xi}{\ga}}R_N^\delta(u,v)\leq 2\gep+\sup\limits_{u\leq u'\leq u+\delta +\gep}|X^{(2)}_{u'}-X^{(2)}_u|+\sup\limits_{v\leq v'\leq v+\delta+\gep}|X^{(1)}_{v'}-X^{(1)}_v| \,,\\
&
 \big|N^{-\frac{\xi}{\ga}}  \Sigma_{\lfloor u N^{\xi}\rfloor}^{-}  -  X_u^{(2)}\big|  \leq \gep 
 \qquad \text{and} \qquad   \big|N^{-\frac{\xi}{\ga}}  \Sigma_{\lfloor v N^{\xi}\rfloor}^{+}  -  X_v^{(1)}\big|  \leq \gep \,,
\end{align*}
uniformly for all $u, v\in U_\delta$.
Now letting $N\to\infty$ and then $\gep\to0$,
we readily have that the $\limsup$ as $N\to+\infty$ of the right-hand side of \eqref{E308} is $\hat \bbP$-a.s.\ smaller than
\begin{equation}\label{E312}
\what \cW_{2}^{A,\gd} :=\max_{u,v \in U_{\gd}}  \Big\{\hat{\beta}(X_{u}^{(2)}+X_{v}^{(1)})+\hat \gb \sup_{0\leq t\leq\delta}|X_{u+t}^{(2)}-X_{u}^{(2)}|+\hat \gb \sup_{0\leq t\leq\delta}|X_{v+t}^{(1)}-X_{v}^{(1)}| - I(u,v)\Big\}.
\end{equation}

\smallskip
\noindent
{\it Lower bound.}
On the other hand, a lower bound on $\log Z_{N}^{\omega}(k_{1},k_{2},\delta)$ is given by
\begin{equation*}
 \gb_N \Big( \Sigma_{ \lfloor k_{1}\delta N^{\xi}\rfloor}^- + \Sigma_{ \lfloor k_{2}\delta N^{\xi}\rfloor}^+\Big) - \gb_N R_{N}^{\gd}(k_{1}\gd,k_{2}\gd)
- |\hat h|(k_{2}+k_{1} +2)\delta N^{\xi}  + p_{N}^{\gd}(k_{1}\gd,k_{2}\gd). 
\end{equation*}
Thus, setting $u=k_{1}\delta, v=k_{2}\delta$ and  $U_{\gd}=\{0,\delta,\ldots,A\}$ as above, we obtain
\begin{multline*}\label{E315}
\max_{0\leq k_1, k_2 \leq \frac{A}{\gd}}  \frac{ \log Z_{N}^{\omega} (k_1,k_2,\gd) }{N^{2\xi-1}}  \geq   \max_{u,v \in U_{\gd}}  \Big\{  \hat \gb  N^{-\frac{\xi}{\ga}}  \big(  \Sigma_{\lfloor u N^{\xi}\rfloor}^{-} +  \Sigma_{\lfloor v N^{\xi}\rfloor}^{+} \big) - \hat\gb N^{-\frac{\xi}{\ga}} R_{N}^{\gd}(u,v)\\
 - |\hat h| (u+v +2\gd) N^{(\xi-\zeta)-(2\xi-1) }   + N^{-(2\xi-1)} p_N^{\gd}(u,v) \Big\} \, .
\end{multline*}
We get as above that the  $\liminf$ as $N\to+\infty$ of the  right-hand side
is $\hat \bbP$-a.s.\ larger than
\begin{equation}\label{E316}
\wcheck \cW_{2}^{A,\gd} := \max\limits_{u,v\in U_\gd} \Big\{\hat{\beta}(X_{u}^{(2)}+X_{v}^{(1)})-\hat \gb \sup_{0\leq t\leq\delta}|X_{u+t}^{(2)}-X_{u}^{(2)}|-\hat \gb \sup_{0\leq t\leq\delta}|X_{v+t}^{(1)}-X_{v}^{(1)}| - I(u,v)\Big\}.
\end{equation}

\smallskip
\noindent
{\it Conclusion.}
 Summarizing, 
we have $\hat \bbP$-a.s.\ the upper bound~\eqref{E312}
and the lower bound~\eqref{E316} for $\limsup_N N^{-(2\xi-1)} \log Z_N^{\go,\leq}$ and $\liminf_N N^{-(2\xi-1)} \log Z_N^{\go,\leq}$ respectively.
Notice that, since trajectories of L\'evy processes are a.s.\ c\`ad-l\`ag  (continuous in the case of Brownian motion),
we have that
\[
\lim_{\gd\downarrow 0} \wcheck\cW_{2}^{A,\gd} = \lim_{\gd\downarrow 0 } \what\cW_{2}^{A,\gd} = \sup_{u,v\in [0,A]} \Big\{\hat{\beta}(X_{u}^{(2)}+X_{v}^{(1)})- I(u,v)\Big\} \, ,
\]
which is exactly $\cW_{2}^{A}$.
\end{proof}

\paragraph*{Step 2.}
Next, we prove the following result.
\begin{lemma}
\label{lem:R2step2}
In Region $R_2$, there is some $A_0>0$ and some constant $C=C_{\hat \gb}$ such that, for all $A\ge A_0$, for all $N\geq 1$
\[
\bbP\Big(  \frac{1}{N^{2\xi-1}} \log Z_{N}^{\omega}\left(M_N^*>AN^{\xi}\right)  \geq  - 1  \Big) \leq  C A^{1-2\ga} \, .
\]
\end{lemma}
\noindent
Since $\ga>1/2$ in Region $R_2$,
Lemma \ref{lem:R2step2} implies that for any $\gep>0$ we can choose $A>0$ such that
$\frac{1}{N^{2\xi-1}} \log Z_{N}^{\omega}\left(M_N^*>AN^{\xi}\right) <  - 1$ with $\bbP$-probability larger than $1-\gep$.
Therefore, thanks to Lemma~\ref{lem:R2step1} and because $\cW_{2}^{A}\geq 0$ (by taking $u=0=v$), one can choose $A$ such that
the second term in \eqref{E301} is negligible compared to the first one in $\hat \bbP$-probability.

\begin{proof}[Proof of Lemma~\ref{lem:R2step2}]
Let us write
\begin{equation*}
 Z_{N}^{\omega,>}  :=Z_{N}^{\omega} \left(M_N^*>AN^{\xi}\right)= \sum\limits_{k=1}^{\infty} Z_{N}^{\omega}\left(M_N^*\in(2^{k-1}AN^{\xi},2^{k}AN^{\xi}]\right) \,,
\end{equation*}
so that 
\[
 Z_{N}^{\omega,>}  \leq  \sum\limits_{k=1}^{\infty} \exp\left( \hat \gb N^{-\gamma} \Sigma_{2^k A N^\xi}^* + |\hat{h}| 2^{k+1} A N^{\xi-\zeta}\right) \bP \big( M_N^* \geq 2^{k-1} A N^{\xi} \big) \, .
\]
Using that $\mathbf{P}\left(M_N^*\geq2^{k-1}AN^{\xi}\right)\leq 2 \exp\left(-2^{2k-3}A^{2}N^{2\xi-1}\right)$,
we get by subadditivity
\begin{align*}
 \bbP  \left( Z_{N}^{\omega,>} \geq e^{-N^{2\xi-1}}\right) &\leq \sum\limits_{k=1}^{\infty} \bbP \Big( 2 \exp\left( \hat \gb N^{-\gamma} \Sigma_{2^k A N^\xi}^* -2^{2k-4}A^{2}N^{2\xi-1} \right) \geq  \frac{1}{2^{k+1}} e^{- N^{2\xi-1}} \Big)\,,
\end{align*}
where we have used the fact that $2^{k+1}|\hat{h}| A N^{\xi-\zeta}\leq 2^{2k-4}A^2 N^{2\xi-1}$ for large enough $A$, uniformly in $k,N\geq 1$ (using that we have $\xi-\zeta<2\xi-1$ in Region $R_2$).
Therefore, provided that~$A$ is sufficiently large, recalling that $\gamma +2\xi-1 = \xi/\alpha$, we end up with
\[
\bbP  \left( Z_{N}^{\omega,>} \geq e^{-N^{2\xi-1}}\right)  \leq  \sum\limits_{k=1}^{\infty} \bbP \left( \hat\gb \Sigma_{2^k A N^\xi}^* \geq  2^{2k-5} A^{2} N^{\xi/\alpha} \right)
\leq \sum\limits_{k=1}^{\infty} c \, \hat \gb^{\ga} 2^{(1-2\ga)k} A^{1-2\ga} \,,
\]
where we have used Lemma \ref{lem1} for the last inequality.
Since $\alpha>1/2$, this concludes the proof of Lemma~\ref{lem:R2step2}.
\end{proof}

\begin{remark}
\label{rem:almostsure}
If we want to upgrade our convergence in $\hat \bbP$-probability to a $\hat \bbP$-a.s.\ convergence, we would need to upgrade Lemma~\ref{lem:R2step2} to the following:
\[
\lim_{A\to\infty} \hat \bbP\Big(  \limsup_{N\to\infty} \frac{1}{N^{2\xi-1}} \log Z_{N}^{\omega}\left(M_N^*>AN^{\xi}\right)  \geq  - 1  \Big)  =0\,.
\]
With the same proof as above, we would need to bound

\[
\lim_{N_0\to\infty} \bbP  \left( \sup_{N\geq N_0} ( Z_{N}^{\omega,>}   e^{N^{2\xi-1}}) \geq 1 \right)  \leq   \lim_{N_0\to\infty} \sum\limits_{k=1}^{\infty} \bbP \left(   \sup_{N\geq N_0} (N^{-\xi/\alpha} \Sigma_{2^k A N^\xi}^*) \geq  c_{\hat \gb} 2^{2k} A^{2} \right) \,.
\] 
The proof would be complete if one could exchange the limit and the sum since we have $$\limsup_{N\to\infty} N^{-\xi/\alpha} \Sigma_{2^k A N^\xi}^* \leq 3 X_{2^kA}^* \quad \text{with} \quad X_t^* = \sup_{0\leq s\leq t} |X_s^{(1)}| + \sup_{0\leq s\leq t} |X_s^{(2)}|)$$ the continuous process analogous to $\Sigma^*$, see e.g.\ Lemma~~2.23 in~\cite[Ch.~VI]{JS03}.
But one needs to apply dominated convergence, that is control the tail of $\sup_{N\geq N_0} (N^{-\xi/\alpha} \Sigma_{2^k A N^\xi}^*)$; for that one would need a better control of the convergence in the coupling. 
Indeed, a sufficient condition to obtain the $\hat \bbP$-a.s.\ convergence is the following: there is some $\gep<2-\frac1\alpha$ such that $\hat \bbP$-a.s.\ there exists a constant $C(\omega)$ such that
\[
 \sup_{N\geq 1} N^{-\xi/\alpha}  \Sigma^*_{t N^{\xi}} \leq C(\omega) t^{\gep} X_t^*\,,\quad\text{uniformly in } t\geq 1. 
\]
The important part in this condition is that the constant $C(\omega)$ is uniform on all scales $t \in [2^kA, 2^{k+1}A]$.
For instance, this condition is verified if the coupling is \emph{exact}, that is if the $\go_i$'s are $\alpha$-stable, in which case we can set $\hat \go_i^{(N)} = N^{\xi/\alpha} (X_{(i+1)N^{-\xi}}^{(1)} -X_{iN^{-\xi}}^{(1)})$ for $i\geq 0$ (and analogously for $i< 0$).
\end{remark}

\paragraph*{Step 3.}
Let us note that, by monotonicity in $A$, we have that $\cW_{2} = \lim_{A\uparrow\infty} \cW_{2}^A$
is well-defined (possibly infinite) and non-negative.
We prove the following lemma:
\begin{lemma}
\label{lem:R2step3}
If $\ga\in (\frac12,2]$, we have that $ \mathcal{W}_{2}:=\sup_{u,v\geq 0} \{\hat{\beta}(X_{v}^{(1)} + X_{u}^{(2)})-I(u,v)\}$
is $\hat \bbP$-a.s.\ positive and finite.
\end{lemma}
Combined with Lemmas~\ref{lem:R2step1}-\ref{lem:R2step2},
this readily proves that $N^{-(2\xi-1)} \log Z_N^{\go}$ converges almost surely to $\cW_{2}$ as $N\to\infty$.

\begin{proof}
To show that $\cW_{2}>0$ almost surely, notice that taking $u=0$
we have
\[
\cW _{2} \geq \sup_{v\geq 0} \big\{\hat{\beta} X_{v}^{(1)} - \tfrac12 v^2 \big\} \, .
\]
Then, almost surely, we can find some sequence $v_n \downarrow 0$ such that $X_{v_n}^{(1)} \geq v_n^{1/\ga} $ for all $n$ (see e.g.~\cite[Th.~2.1]{Levyproc}):
we  get that $\cW_{2} \geq \sup_{n\geq 0} \{ \hat \gb v_n^{1/\ga} - \frac12 v_n^2\} >0$ since $\ga > 1/2$.

To show that $\cW_{2}<+\infty$ a.s., notice that $I(u,v) = \frac12( u\wedge v + u+v  )^2 \geq \frac12 (u^2 + v^2)$:
we therefore get that 
\[
 \mathcal{W}_{2} \leq \sup_{u\geq 0} \big\{\hat{\beta} X_{u}^{(2)} - \tfrac12 u^2 \big\}
+\sup_{v\geq 0} \big\{\hat{\beta} X_{v}^{(1)} - \tfrac12 v^2 \big\} \, .
\]
Let us consider the second term and show that it is a.s.\ finite (the first term is identical). 
For any $\epsilon >0$, a.s.\ $X_v^{(1)} \le v^{(1+\epsilon)/\ga}$ for $v$ large enough, see e.g.~\cite[Sec.~3]{pruitt1981}. 
Hence $\hat{\beta} X_{v}^{(1)} - \tfrac12 v^2 \le \hat{\beta} v^{(1+\epsilon)/\ga} - \tfrac12 v^2 \le 0$ for all $v$ large enough, provided that $(1+\epsilon)/\ga<2$,
which concludes the proof. 
\end{proof}

\subsubsection*{Convergence of $(M_N^-,M_N^+)$}

Let us define, for $\gep,\gep' \in (0,1)$ 
\[
\cU_2^{\gep,\gep'} = \Big\{ (u,v)\in (\bbR_+)^2  \colon  \sup_{(s,t) \in  B_{\gep}(u,v)}  \big\{\hat{\beta}(X_t^{(1)}  + X_s^{(2)})-I(s,t)\big\} \geq \cW_{2} -\gep'\Big\},
\]
where $ B_{\gep}(u,v)$ is the closed ball of center $(u,v)$ and of radius $\gep>0$. 
Let us observe that $\cU_2^{\gep,\gep'}$ is a.s.\ bounded: we know that a.s.\ the supremum outside a compact set $[-A(\go),0]\times [0,A(\go)]$ is smaller than $-1\leq \cW_{2}-\gep'$, see Lemma \ref{lem:R2step3}. 
Moreover, by Lemma~\ref{lem:R2step3} we can choose~$\epsilon'$ such that $\cW_2-\epsilon'>0$.

We now prove that for any $\gep,\, \gep' \in (0,1)$, $\lim_{N\to\infty} \bP_{N}^{\go} ( \frac{1}{N^{\xi}} (-M_N^-,M_N^+) \in \cU_2^{\gep,\gep'} ) =1 $ {  in $\hat \bbP$-probability}.
To simplify the notation, we denote the event $\{\frac{1}{N^\xi}(-M_N^-,M_N^+)\notin\cU_2^{\gep,\gep'}\}$ by $\cA_{N,2}^{\gep,\gep'}$. We have
\begin{equation*}
\log\bP_N^\omega(\cA_{N,2}^{\gep,\gep'})=\log Z_N^\omega(\cA_{N,2}^{\gep,\gep'})-\log Z_N^\omega \, .
\end{equation*}
From what we showed above, we have that $N^{-(2\xi-1)}\log Z_N^\omega$ converges  in $\hat\bbP$-probability to~$\cW_{2}$, so the proof will be complete if we show that $N^{-(2\xi-1)}\log Z_N^\omega(\cA_{N,2}^{\gep,\gep'}) < \cW_{2}$   with $\hat \bbP$-probability close to $1$.
Thanks to Lemma \ref{lem:R2step2}
 we only need to estimate $Z_N^\omega(M_N^*\leq AN^\xi;\cA_{N,2}^{\gep,\gep'})$. For any $\delta>0$, we perform a similar decomposition as in \eqref{E302} to get
\begin{equation*}\label{joint1}
Z_N^{\omega,\leq}(\cA_{N,2}^{\gep,\gep'}):=Z_{N}^{\omega}\big(M_N^*\leq AN^{\xi};\cA_{N,2}^{\epsilon,\epsilon'}\big) =\sum\limits_{k_{1}=0}^{\lfloor A/\gd \rfloor}\sum\limits_{k_{2}=0}^{\lfloor A/\gd \rfloor} Z_{N}^{\omega}(k_{1},k_{2},\delta;\cA_{N,2}^{\gep,\gep'}),
\end{equation*}
where we defined 
$Z_{N}^{\omega}(k_{1},k_{2},\delta;\cA_{N,2}^{\gep,\gep'})$
as 
\begin{equation*}\label{joint2}
Z_{N}^{\omega}\left(M_N^-\in( -(k_{1}+1)\delta N^{\xi},-k_{1}\delta N^{\xi}],M_N^+\in [k_{2}\delta N^{\xi},(k_{2}+1)\delta N^{\xi});\cA_{N,2}^{\epsilon,\epsilon'}\right).
\end{equation*}
By definition of $\cA_{N,2}^{\gep,\gep'}$, we get that
\begin{equation*}
\label{ZAR2}
Z_N^{\omega,\leq}(\cA_{N,2}^{\gep,\gep'}) 
\leq\left(\frac{A}{\delta}\right)^2\max\limits_{(k_1,k_2)\in U_{\delta,2}^{\gep,\gep'}} Z_N^\omega(k_1,k_2,\delta),
\end{equation*}
where
$U_{\delta,2}^{\gep,\gep'}:=\big\{ (k_1,k_2): k_1\delta,k_2\delta\in U_\delta, [k_1\delta,(k_1+1)\delta)\times[k_2\delta, (k_2+1)\delta) \not\subset  \cU_2^{\gep,\gep'} \big\}$, 
with $U_\delta=\{0,\delta,2\delta,\ldots,A\}$; by convention the maximum is $0$ if $U_{\delta,2}^{\gep,\gep'}$ is empty.

Now, by the same argument as in \textbf{Step 1}, we have that
\begin{equation*}
\label{limsupR2}
\limsup\limits_{\delta\to0}\limsup\limits_{N\to\infty}\frac{1}{N^{2\xi-1}}\log Z_N^{\omega,\leq}(\cA_{N,2}^{\gep,\gep'})\leq\sup\limits_{(u,v)\notin \cU_2^{\gep,\gep'}}\left\{\hat{\beta}(X_v^{(1)} + X_u^{(2)})-I(u,v)\right\}\leq\cW_{2}-\epsilon'\, ,
\end{equation*}
by definition of  $\cU_2^{\gep,\gep'}$.
This concludes the proof that $\lim_{N\to\infty}\bP_{N}^{\hat\go} (\cA_{N,2}^{\gep,\gep'}) = 0$ $\hat\bbP$-a.s.

 Let us now observe that
thanks to Proposition~\ref{uniquemaximizer}, the maximiser $(\cU^{(2)}, \cV^{(2)})$ of $\cW_2$ is $\hat\bbP$-a.s.\ unique: 
hence,   $\bigcap_{\gep'>0}\cU_2^{\gep,\gep'}\subset B_{4\epsilon}(\cU^{(2)}, \cV^{(2)})$.
Therefore, for any $\gep>0$
there is a.s.\ some $\gep'>0$ such that $\cU_2^{\gep,\gep'}$ is included in $B_{8\epsilon}(\cU^{(2)}, \cV^{(2)})$, which concludes the proof.
\qed

\subsection{Region R3: proof of Theorem \ref{C103}}
 Recall that in 
Region~$R_3$ we have
\[
\gamma<\zeta-\tfrac{\ga-1}{\ga}\quad \text{and} \quad \gamma< \tfrac{1-\ga}{\ga} \, ,
\qquad \text{with } \ga \in (0,1)\cup(1,2] \,.
\] 
We prove that the range size is of order $N$.

\subsubsection*{Convergence of the rescaled log-partition function}

First of all, notice that we can reduce to the case $h_N \equiv 0$.
Indeed, we have the bounds
\[
 Z_{N,\gb_N,h_N=0}^{\omega} \times e^{-|h_N| N}  \leq Z_{N,\gb_N,h_N}^{\omega} \leq Z_{N,\gb_N,h_N=0}^{\omega} \times e^{|h_N| N} \, .
\]
Since $h_N = \hat h N^{-\zeta}$ with $\zeta> \gamma + \frac{\ga-1}{\ga}$, we have that $\lim_{N\to\infty }N^{-(\frac{1}{\ga}-\gamma)}|h_N| N = 0$.
In the following, we therefore focus on the convergence of $N^{-(\frac{1}{\ga}-\gamma)}\log Z_{N,\gb_N,h_N=0}^{\omega}$.
We write for simplicity $Z_{N}^{\omega}$ for $Z_{N,\gb_N,h_N=0}^{\omega}$.

For any $\delta>0$, we can write
\begin{equation*}
Z_{N}^{\omega}=  \sum_{k_1=0}^{ \lfloor 1/\gd \rfloor} \sum_{k_2=0}^{\lfloor 1/\gd \rfloor}
Z_{N}^{\omega}(k_{1},k_{2},\delta),
\end{equation*}
with $Z_N^\go(k_1,k_2,\delta)$ as in \eqref{E303} with $\xi=1$.
Since there are at most~$N$ steps for the random walk, we can have $M_N^- \leq -k_1 \gd N$ and $M_N^+ \geq k_2 \gd N$ only if $\gd ( k_1\wedge k_2 + k_1+k_2 ) \leq 1 $.
Hence, writing $u= k_{1}\delta$, $v= k_{2}\delta$, and $U_{\gd} = \{0,\gd, 2\gd,\ldots, 1\}$, we have
\begin{equation*}
\max_{u, v \in U_{\gd} \atop  u\wedge v +u+v\leq 1 }\log Z_{N}^{\omega}(\tfrac{u}{\gd}, \tfrac{v}{\gd}, \gd)\leq\log Z_{N}^{\omega}\leq-2\log\delta+\max_{u, v \in U_{\gd} \atop  u\wedge v +u+v\leq 1 } \log Z_{N}^{\omega}(\tfrac{u}{\gd}, \tfrac{v}{\gd}, \gd).
\end{equation*}

\smallskip
\noindent
{\it Upper bound.}
We have
\begin{equation*}\label{E404}
\max_{u, v \in U_{\gd} \atop  u\wedge v +u+v\leq 1 }  \!\!  N^{\gamma-\frac{1}{\alpha}} \log Z_{N}^{\omega}(\tfrac{u}{\gd}, \tfrac{v}{\gd}, \gd)  \leq   \max_{u, v \in U_{\gd} \atop  u\wedge v +u+v\leq 1 }    \hat{\beta}\left( N^{-\frac{1}{\ga}} \Sigma_{\lfloor uN \rfloor}^{-} 
  +  N^{-\frac{1}{\ga}}  \Sigma_{\lfloor vN \rfloor}^{+} + N^{-\frac{1}{\ga}} R_{N}^{\gd}(u,v)\right)
  \, ,
\end{equation*}
where $R_{N}^{\gd}(u,v)$ is as in \eqref{E306} with $\xi=1$.
As in the previous section, we get that $\hat \bbP$-a.s.\ the $\limsup$ of the right-hand side is bounded above by
\begin{equation*}
\what\cW_{3}^{\gd} :=\max_{u, v \in U_{\gd} \atop  u\wedge v +u+v\leq 1 } \Big\{\hat{\beta}(X_{u}^{(2)}+X_{v}^{(1)})+\hat \gb \sup_{0\leq t\leq\delta}|X_{u+t}^{(2)}-X_{u}^{(2)}|+\hat \gb \sup_{0\leq t\leq\delta}|X_{v+t}^{(1)}-X_{v}^{(1)})| \Big\}.
\end{equation*}

\smallskip
\noindent
{\it Lower bound.}
We have
\begin{multline*}
\max_{u, v \in U_{\gd} \atop  u\wedge v +u+v\leq 1 }  N^{\gamma-\frac{1}{\ga}} \log Z_{N}^{\omega}(\tfrac{u}{\gd}, \tfrac{v}{\gd}, \gd)  \\
\geq   \max_{u, v \in U_{\gd} \atop  u\wedge v +u+v\leq 1 }  \Big\{ \hat{\beta}\left(  N^{-\frac{1}{\ga}}\Sigma_{\lfloor u N \rfloor}^{-} 
  + N^{-\frac{1}{\ga}}  \Sigma_{\lfloor v N \rfloor}^{+} \right) 
   - N^{-\frac{1}{\ga}} R_{N}^{\gd}(u,v)  - N^{1+\gamma-\frac{1}{\ga}} \log 2\Big\}
  \, ,
\end{multline*}
where we used that any non-empty event of $(S_n)_{0\leq n \leq N}$
has probability at least $2^{-N}$.
Now, since $\gamma<\zeta+\frac{1}{\alpha}-1$ and $\gamma<\frac{1}{\alpha}-1$, the last
two terms in the maximum go to $0$: 
we therefore get that $\hat \bbP$-a.s.\ the $\liminf$  of the right-hand side is bounded below by
\begin{equation*}
\wcheck\cW_{3}^{\gd} :=\max_{u, v \in U_{\gd} \atop  u\wedge v +u+v\leq 1 } \Big\{\hat{\beta}(X_{u}^{(2)}+X_{v}^{(1)})- \hat \gb \sup_{0\leq t\leq\delta}|X_{u+t}^{(2)}-X_{u}^{(2)}| - \hat \gb \sup_{0\leq t\leq\delta}|X_{v+t}^{(1)}-X_{v}^{(1)}| \Big\}.
\end{equation*}

\smallskip
\noindent
{\it Conclusion.}
We can conclude 
in the same manner as in the proof of Lemma~\ref{lem:R2step1}:
letting $N\to\infty$ and then $\gd\downarrow 0$, we get that
$N^{\gamma-\frac1\ga} \log Z_N^{\go}$ converges $\hat \bbP$-a.s.\ to
\[
\lim_{\gd\downarrow 0} \what\cW_{3}^{\gd}= \lim_{\gd\downarrow 0 } \wcheck\cW_{3}^{\gd}  = \sup_{u,v\in [0,1] \atop u\wedge v +u+v\leq 1} \Big\{\hat{\beta}(X_{u}^{(2)}+X_{v}^{(1)})\Big\} \, ,
\]
where the limit holds thanks to the a.s.\ c\`adl\`ag property of trajectories of the L\'evy process (or continuity in the Brownian motion case).
This is exactly the variational problem~$\cW_{3}$ defined in Theorem~\ref{C103}.
Together with the (trivial) fact that 
$\cW_{3} \in (0,+\infty)$ a.s., this concludes the proof of~\eqref{E120}.
\qed

\subsubsection*{Convergence of $(M_N^-,M_N^+)$}

The proof follows the same strategy as in Region $R_2$, so we only give a sketch. Let us define the counterpart of $\cU_2^{\epsilon,\epsilon'}$ in Region $R_3$ by
\[
\cU_{3}^{\gep,\gep'} = \Big\{ (u,v)\in  (\bbR_+)^2  \, :\, u\wedge v+u+v\leq1, \sup_{s,t\geq 0  , (s,t) \in  B_{\gep}(u,v)\atop s\wedge t+s+t\leq 1}  \{\hat{\beta}(X_t^{(1)} + X_s^{(2)})\} \geq \cW_{3} -\gep'\Big\}.
\]
Then we denote the event $\{\frac{1}{N}(-M_N^-,M_N^+)\notin\cU_{3}^{\gep,\gep'}\}$ by $\cA_{N,3}^{\gep,\gep'}$. By the same procedure as in Region $R_2$ {  (here we can use $\hat\bbP$-a.s.\ convergences, since we do not have to restrict trajectories)}, we can first show that $\hat \bbP$-a.s.
\begin{equation*}
\limsup_{N\to\infty} \frac{1}{N^{\frac{1}{\alpha}-\gamma}}\log Z_{N}^{\go}(\cA_{N,3}^{\gep,\gep'})  < \cW_{3} \, \quad \text{ and so } \quad 
\limsup_{N\to\infty} \frac{1}{N^{\frac{1}{\alpha}-\gamma}}\log\bP_{N}^{\go}(\cA_{N,3}^{\gep,\gep'}) < 0 \,.
\end{equation*}
We then deduce, as done in Region $R_2$ that $\hat\bbP$-a.s.  
$
\lim_{N\to\infty}\bP_{N}^{\omega} ( \frac{1}{N}(-M_N^-,M_N^+)\in \cU_{3}^{\gep,\gep'})=1.
$
By uniqueness of the maximizer $(\cU^{(3)},\cV^{(3)})$ of $\cW_3$ (Propposition~\ref{uniquemaximizer}), we get that if $\epsilon'$ is small enough, then $\cU_{3}^{\gep,\gep'}$ is contained in $B_{8\epsilon}(\cU^{(3)},\cV^{(3)})$,
which completes the proof.
\qed

\subsection{Region R4: proof of Theorem \ref{C106}}

We prove that in Region $R_4$ the range size is of order $N^\xi$ with $\xi=\frac{\ga}{\ga - 1}(\zeta -\gamma) \in (0,1)$.
We take $\hat h>0$, and recall that in this region  we have 
\[
\big(\tfrac{(2\ga-1)\zeta-(\ga - 1)}{\ga} \big)\vee \big(\zeta-\tfrac{\ga-1}{\ga}\big)<\gamma<\big(\tfrac{(2\ga+1)\zeta-(\ga - 1)}{3\ga}\big) \wedge \zeta\,, \quad \text{with }\ga\in (1,2] \,,
\] 
and that $\xi-\zeta =\xi/\ga -\gamma> |2\xi-1|$.
Recall also that region~$R_4$ does not exist if $\ga<1$.

\subsubsection*{Convergence of the rescaled log-partition function} 
For any $A>0$, we first write
\begin{equation}\label{E350}
Z_{N}^{\omega}=Z_{N}^{\omega}\left(M_{N}^{*}\leq AN^{\xi}\right)+Z_{N}^{\omega}\left(M_{N}^{*}> AN^{\xi}\right).
\end{equation}
The strategy is similar to that in Region $R_2$ and we use analogous notation. We proceed in three steps: (1)~after taking logarithm and dividing by $N^{\xi-\zeta}$, we show that the first term converges to some limit $\cW_{4}^A$ when $N\to\infty$; (2)~we show that the second term in~\eqref{E350} is small compared to the first one; (3)~we show that $\cW_{4}^A\rightarrow\cW_{4}$ as $A\to\infty$, with $\cW_{4}\in  (0,+\infty)$ almost surely.

\paragraph*{Step 1.} 
We prove the following lemma.

\begin{lemma}\label{R4LM1} 
In Region $R_4$,  we have that $\hat\bbP$-a.s., for any $A\in \bbN$,
\begin{equation*}
\lim_{N\to\infty} \frac{1}{N^{\xi-\zeta} }\log Z_{N}^{\omega}\left(M_{N}^{*}\leq AN^{\xi}\right) =\cW_{4}^A:=\sup\limits_{u,v \in [0,A]}\left\{\hat{\beta}(X_{v}^{(1)} +X_{u}^{(2)})-\hat{h}(u+v)\right\}\,,
\end{equation*}
with $(X_{v}^{(1)}, X_u^{(2)})_{u,v\geq 0}$ from  Notation~\ref{defL}.
\end{lemma}

\begin{proof}
For fixed $\delta>0$, we write (recall the notation~\eqref{E303})
\begin{equation*}
\label{splitR4}
Z_{N}^{\omega,\leq }:=Z_{N}^{\omega}\left(M_{N}^{*}\leq AN^{\xi}\right)=\sum\limits_{k_{1}=0}^{\lfloor A/\delta\rfloor}\sum\limits_{k_{2}=0}^{\lfloor A/\delta\rfloor}Z_{N}^{\omega}(k_{1},k_{2},\delta).
\end{equation*}
Since the number of summands above is finite, we can write
\begin{equation*}
\max_{0\leq k_1, k_2 \leq \frac{A}{\gd}}  \log  Z_{N}^{\omega}(k_{1},k_{2},\delta)  \leq  \log Z_{N}^{\go,\leq }  \leq   2\log (\tfrac{A}{\gd})  + \max_{0\leq k_1, k_2 \leq \frac{A}{\gd}}   \log  Z_{N}^{\omega}(k_{1},k_{2},\delta) \, .
\end{equation*}

\smallskip
\noindent
{\it Upper bound.} We write $u=k_{1}\delta, v=k_{2}\delta$ and set $U_{\gd}=\{0,\delta,2\gd,\ldots,A\}$. 
Recalling that $\xi -\zeta =\xi/\ga-\gamma$, we get that
\begin{multline*}
\label{E353}
\max_{0\leq k_1, k_2 \leq \frac{A}{\gd}} 
 \frac{ 1 }{N^{\xi-\zeta}} \log Z_{N}^{\omega} (k_1,k_2,\gd)
\\ \leq  
\max_{u,v \in U_{\gd}}  \Big\{   \hat\gb  N^{-\frac{\xi}{\ga}} \big(  \Sigma_{\lfloor u N^{\xi}\rfloor}^{-} 
+  \Sigma_{\lfloor v N^{\xi}\rfloor}^{+} \big)  + \hat\gb N^{-\frac{\xi}{\ga}} R_{N}^{\gd}(u,v) 
- \hat{h}(v+u) + N^{\xi-\zeta} p_N^{\gd}(u,v) \Big\} \,.
\end{multline*}
Then, notice that $\lim_{N\to\infty}N^{\xi-\zeta} p_{N}^{(\delta)}(u,v)=0$, thanks to Lemmas~\ref{lem:superdiffusive} and \ref{lem:subdiffusive}, since we have $ \xi-\zeta >|2\xi-1|$.
Therefore, similarly to the previous sections, $\hat\bbP$-a.s.\
the $\limsup$ of right-hand side is bounded above by
\begin{equation*}\label{E354}
\what\cW_{4}^{A,\gd}:=\max\limits_{u,v\in U_\gd}\Big\{\hat{\beta}\Big(X_{u}^{(2)}+X_{v}^{(1)}+\sup\limits_{t\in[0,\delta]}|X_{u+t}^{(2)}-X_{u}^{(2)}|+\sup\limits_{t\in[0,\delta]}|X_{v+t}^{(1)}-X_{v}^{(1)}|\Big)-\hat{h}(v+u)\Big\}.
\end{equation*}

\smallskip
\noindent
{\it Lower bound.} On the other hand, we bound $\log Z_{N}^{\omega}(k_{1},k_{2},\delta)$ from below by
\begin{equation*}
\gb_N \Big( \Sigma_{ \lfloor k_{1}\delta N^{\xi}\rfloor}^- + \Sigma_{ \lfloor k_{2}\delta N^{\xi}\rfloor}^+\Big) - \gb_N R_{N}^{\gd}(k_{1}\gd,k_{2}\gd)
- h_N(k_{2}+k_{1} +2)\delta N^{\xi}  - p_{N}^{\gd}(k_{1}\gd,k_{2}\gd). 
\end{equation*}
Thus, setting $u=k_{1}\delta, v=k_{2}\delta$ and  $U_{\gd}=\{0,\delta,\ldots,A\}$ as above, we obtain
\begin{multline*}\label{E355}
\max_{0\leq k_1, k_2 \leq \frac{A}{\gd}}  \frac{ 1 }{N^{\xi-\zeta}}  \log Z_{N}^{\omega} (k_1,k_2,\gd) \\
 \geq   
\max_{u,v \in U_{\gd}} \Big\{  \hat \gb   N^{-\frac{\xi}{\ga}}  \big(  \Sigma_{\lfloor u N^{\xi}\rfloor}^{-} +  \Sigma_{\lfloor u N^{\xi}\rfloor}^{+} \big) - \hat\gb N^{-\frac{\xi}{\ga}} R_{N}^{\gd}(u,v)
  - \hat h (u+v +2\gd)   - N^{\xi-\zeta} p_N^{\gd}(u,v) \Big\} \, .
\end{multline*}
Hence, similarly to what is done above, $\hat \bbP$-a.s.\ the $\liminf$ of the right-hand side is bounded below by
\begin{equation*}\label{E356}
\wcheck\cW_{4}^{A,\gd}:=\max\limits_{u,v\in U_\gd}\Big\{\hat{\beta}\Big(X_{u}^{(2)}+X_{v}^{(1)}-\sup\limits_{t\in[0,\delta]}|X_{u+t}^{(2)}-X_{u}^{(2)}|-\sup\limits_{t\in[0,\delta]}|X_{v+t}^{(1)}-X_{v}^{(1)}|\Big)-\hat{h}(v+u+2\delta)\Big\}.
\end{equation*}

\smallskip
\noindent
{\it Conclusion.} 
The terms $\what\cW_{4}^{A,\gd}$, $\wcheck\cW_{4}^{A,\gd}$
are almost sure upper and lower bound for the $\limsup$ and $\liminf$  of $N^{\zeta- \xi}\log Z_{N}^{\omega,\leq}$.
By the a.s.\ c\`ad-l\`ag property of trajectories of L\'evy processes (or continuity in the Brownian motion case), we have
\begin{equation*}
\lim\limits_{\delta\downarrow0}\what\cW_{4}^{A,\gd}=\lim\limits_{\delta\downarrow0}\wcheck\cW_{4}^{A,\gd}=\sup\limits_{u,v\in[0,A]}\left\{\hat{\beta}\left(X_{u}^{(2)}+X_{v}^{(1)}\right)-\hat{h}(v+u)\right\} ,
\end{equation*}
which is exactly $\cW_{4}^{A}$.
The  convergence in Lemma \ref{R4LM1} is  therefore achieved by letting $N\to\infty$ and then $\delta\to0$.
\end{proof}

\paragraph*{Step 2.}
Next, we prove the following lemma.
\begin{lemma}\label{R4LM2}
In region $R_4$, there is some $A_{0}>0$ and some constant $C=C_{\hat{\beta},\hat{h}}$, such that for all $A\geq A_{0}$ and any $N\geq 1$
\begin{equation*}
\mathbbm{P}\left(\frac{1}{N^{\xi-\zeta}}\log Z_{N}^{\omega}\big(M_{N}^{*}>AN^{\xi}\big)\geq -1\right)\leq CA^{1-\ga}\,.
\end{equation*}
\end{lemma}

{ 
\noindent
Since $\ga>1$ in region $R_4$, this proves that for any $\gep>0$ we can choose $A$ large enough so that $\frac{1}{N^{\xi-\zeta}}\log Z_{N}^{\omega}\big(M_{N}^{*}>AN^{\xi}\big) <-1$ with $\bbP$ probability larger than $1-\gep$.
Hence, thanks to Lemma~\ref{R4LM1} and the fact that $\cW_{4}^{A}\geq 0$, the second term in \eqref{E350} is negligible compared to the first one in $\hat\bbP$-probability.
Note that here again, we are not able to upgrade the convergence to a $\hat \bbP$-a.s.\ convergence, see Remark~\ref{rem:almostsure} which also applies here.
}

\begin{proof}
Let us write $Z_{N}^{\omega,>}:=Z_{N}^{\omega}\left(M_N^*>AN^{\xi}\right)$ so
\begin{align*}
Z_{N}^{\omega,>}=\sum\limits_{k=1}^{\infty} Z_{N}^{\omega}\left(M_N^*\in(2^{k-1}AN^{\xi},2^{k}AN^{\xi}]\right) 
 & \leq \sum_{k=1}^{\infty} \exp\Big(\hat \gb N^{-\gamma} \Sigma^*_{2^k A N^{\xi}} - \hat h 2^{k-1} A N^{\xi-\zeta}  \Big)\, .
\end{align*}
By subadditivity, we therefore get that
\begin{align*}
\bbP  \left( Z_{N}^{\omega,>} \geq e^{-N^{\xi-\zeta}}\right)
&\leq \sum\limits_{k=1}^{\infty} \bbP \left(   e^{\hat \gb N^{-\gamma} \Sigma^*_{2^k A N^{\xi}}  - \hat h   2^{k-1} A N^{\xi-\zeta} }  \geq  \frac{1}{2^{k+1}} e^{- N^{\xi-\zeta}} \right) \\
& \leq \sum\limits_{k=1}^{\infty} \bbP \left( \hat \gb N^{-\gamma} \Sigma^*_{2^k A N^{\xi}}  \geq  \hat h  2^{k-2} A N^{ \frac{\xi}{\alpha} -\gamma}   \right) \, ,
\end{align*}
where the last inequality holds provided that $A$ has been fixed large enough (we also used that $\xi-\zeta = \frac{\xi}{\ga} -\gamma$).
Then Lemma \ref{lem1} gives that each probability in the sum 
is bounded by a constant times $2^{k(1-\ga)} A^{1-\ga}$.
Since $\ga>1$, summing this  over $k$  gives the conclusion of the proof of Lemma~\ref{R4LM2}.
\end{proof}

\paragraph*{Step 3.} By monotone convergence, $\cW_{4}^A$ converges a.s.\ to $\cW_{4}$: we only need to show that~$\cW_{4}$ is positive and finite.  Combining this with Lemmas \ref{R4LM1} and \ref{R4LM2}, this completes the proof of~\eqref{E123}.

\begin{lemma}\label{R4LM3}
If $\ga\in (1,2]$, we have that $\cW_{4}:= \sup_ {u,v\geq0}\{\hat{\beta}(X_{v}^{(1)} + X_{u}^{(2)})-\hat{h}(u+v)\}$ is $\hat \bbP$-a.s.\ positive and finite.
\end{lemma}
\begin{proof} The proof is analogous to the proof of Lemma \ref{lem:R2step3}. 
To show that $\cW_{4}>0$, we use that  $\cW_{4}\geq\sup_{v\geq 0}\{\hat{\beta}X_{v}^{(1)}-\hat{h}v\}$. By \cite[Th~2.1]{Levyproc}, there is a.s.\ a sequence $v_{n}\downarrow0$, such that $X_{v_{n}}^{(1)}\geq v_{n}^{1/\ga}$ for all $n$. Hence, for large enough~$n$, $\cW_{4} \geq\hat{\beta}v_{n}^{1/\ga}-\hat{h}v_{n}>0$, since $\ga>1$.

To show that $\cW_{4}<\infty$, we use that $\cW_{4}\leq\sup_{u\geq0}\{\hat{\beta}X_{u}^{(2)}-\hat{h}u\}+\sup_{v\geq 0}\{\hat{\beta}X_{v}^{(1)}-\hat{h}v\}$. By~\cite{pruitt1981}, we have that for any $\epsilon >0$, a.s. $X_v^{(1)} \le v^{(1+\epsilon)/\ga}$ for $v$ large enough. Therefore, if  $(1+\epsilon)/\ga <1$ (recall $\ga>1$),
we get that $\hat{\beta}X_{v}^{(1)}-\hat{h}v\le \hat{\beta}v^{(1+\epsilon)/a}-\hat{h}v\le 0$ for all $v$ sufficiently large. Similarly we also have that $\hat{\beta}X_{u}^{(2)}-\hat{h}u\le 0$ for all $u$ large enough.
This concludes the proof. 
\end{proof}

\subsubsection*{Convergence of $(M_N^-,M_N^+)$}

As in previous sections, we define
\begin{align*}
\cU_{4}^{\gep,\gep'} &= \Big\{ (u,v)\in (\bbR_+)^2  \colon  \sup_{(s,t) \in  B_{\gep}(u,v)}  \big\{ \hat{\beta}(X_t^{(1)} + X_s^{(2)})-\hat{h}(s+t) \big\} \geq \cW_{4} -\gep'\Big\}
\end{align*}
and the event $\cA_{N,4}^{\epsilon,\epsilon'} = \{\frac{1}{N^\xi}(-M_N^-,M_N^+)\notin\cU_{4}^{\epsilon,\epsilon'}\}.$
Then, in an identical manner as in Regions~$R_2$, 
we have that {  with $\hat \bbP$-probability close to $1$,}
\begin{equation*}
\frac{1}{N^{\xi-\zeta}}\log Z_{N}^{\go}(\cA_{N,4}^{\gep,\gep'})  < \cW_{4} \, \quad \text{ and so } \quad 
\frac{1}{N^{\xi-\zeta}}\log\bP_{N}^{\go}(\cA_{N,4}^{\gep,\gep'}) < 0 \,,
\end{equation*}
from which one deduces that 
\begin{equation}
\lim\limits_{N\to\infty}\bP_{N}^{\omega} \Big( \frac{1}{N^{\xi}}(-M_N^-,M_N^+)\in\cU_{4}^{\gep,\gep'} \Big)=1\,, \qquad \text{  in $\hat\bbP$-probability.}
\end{equation}
Moreover, if $\epsilon'$ is small enough, then $\cU_{4}^{\gep,\gep'}$ is contained in $B_{8\epsilon}(\cU^{(4)},\cV^{(4)})$,
which completes the proof.
\qed

\subsection{Region R5: proof of Theorem \ref{C104}}\label{sec:R5}

In Region $R_5$, we prove that the range size is of order $N^{\xi}$ with $\xi=\frac{1+\zeta}{3} \in (0,\frac12)$. Note that in this region we take $\hat h>0$ and that we have 
\[
1-2\xi=\xi-\zeta > \frac{\xi}{\ga}-\gamma\, \quad \text{and} \quad 
-1<\zeta<\frac{1}{2}\, .
\]

\subsubsection*{Convergence of the rescaled log-partition function}
We fix some constant $A =A(\hat h)$ (large) and we split the partition function as
\begin{equation*}\label{E501}
Z_{N}^{\omega}=Z_{N}^{\omega}\left(M_N^*\leq AN^{\xi}\right)+Z_{N}^{\omega}\left(M_N^*> AN^{\xi}\right).
\end{equation*}
The strategy of proof is similar to that in Region $R_2$, but with only two steps: (1)~we show that after taking logarithm and dividing by $N^{1-2\xi}$, the first term converges to some constant independent of $A$ (if $A$ is large enough); (2)~we show that for $A$ large the second term is negligible compared to the first one.

\paragraph*{Step 1.} We prove the following lemma.
\begin{lemma}\label{R5LM1}
In Region $R_5$, we have that for any $A>0$
\begin{equation*}
\lim_{N\to+\infty} \frac{1}{N^{1-2\xi}}\log Z_{N}^{\omega}\left(M_{N}^{*}\leq AN^{\xi}\right)  = \sup\limits_{ u,v \in [0,A]}\left\{-\hat{h}(u+v)-\bar{I}(u,v)\right\}
\quad \text{$\bbP$-a.s.},
\end{equation*}
where $\bar{I}(u,v):=\frac{\pi^{2}}{2}(u+v)^{-2}$ for $u,v\geq0$. By a simple calculation, the supremum is $-\frac{3}{2}(\hat{h}\pi)^{\frac{2}{3}}$ for any $A> \pi^{\frac{2}{3}}\hat{h}^{-\frac{1}{3}}$ and it is achieved at $u+v=\pi^{\frac{2}{3}}\hat{h}^{-\frac{1}{3}}$.
\end{lemma}

\begin{proof}
For any fixed $A$, we have the following upper and lower bounds
\begin{equation}\label{eq:ProofR5Zn}
  \log \widehat Z_{N}^A -\hat \gb N^{-\gamma}  \Sigma_{AN^{\xi}}^* \leq\log Z_{N}^{\go} \left(M_{N}^{*}\leq AN^{\xi}\right) \leq   \log \widehat Z_{N}^A + \hat \gb N^{-\gamma} \Sigma_{AN^{\xi}}^*\, ,
\end{equation}
where  $\widehat Z_N^A := \bE\big[ \exp( - h_N |\cR_N| ) \ind_{\{M_N^* \leq A N^{\xi}\}} \big]$.

Since in Region $R_5$ we have $ 1-2\xi  >  \frac{\xi}{\ga} -\gamma$,
we get that $N^{-(1-2\xi)}   \hat \gb N^{-\gamma} \Sigma_{AN^{\xi}}^*$ goes to~$0$ $\bbP$-a.s.\ (see e.g.~Lemma~\ref{lem1}).
Therefore, we only need to prove that
\[
\lim_{N\to\infty }\frac{1}{N^{1-2\xi}} \log \widehat Z_{N}^A = \sup\limits_{ u,v \in [0,A] }\big\{-\hat{h}(u+v)-\bar{I}(u,v)\big\} 
\]
(there is no disorder anymore).
But this convergence is quite standard, since $\bar{I}(u,v)$ is the rate function for the LDP for $(N^{-\xi} M_N^-$, $N^{-\xi}M_N^+)$: more precisely, by Lemma~\ref{lem:subdiffusive}
\begin{equation}
- \bar{I}(u,v) = \lim_{N\to\infty} \frac{1}{N^{1-2\xi}} \log \bP\big( M_N^- \geq -u N^{\xi} ; M_N^+ \leq v N^{\xi} \big)\, .
\end{equation}
This is enough to conclude thanks to Varadhan's lemma.
\end{proof}

\paragraph*{Step 2.} Next, we prove the following lemma.
\begin{lemma}\label{R5LM2}
In Region $R_5$, 
we have for any fixed $A$
\[
\limsup_{N\to\infty} \frac{1}{N^{1-2\xi}}\log Z_{N}^{\omega}\left(M_{N}^{*}>AN^{\xi}\right) \leq  -\frac12 A \hat h \qquad \text{$\bbP$-a.s.}
\]
\end{lemma}

\noindent
Combining this result with Lemma~\ref{R5LM1} readily yields
the convergence~\eqref{E121}, provided that $A> \pi^{\frac{2}{3}}\hat{h}^{-\frac{1}{3}}$ and  $\frac12 A \hat h > \frac{3}{2}(\hat{h}\pi)^{\frac{2}{3}}$.

\begin{proof}
We consider four cases, which correspond to four different conditions on $\gamma, \zeta$  (see Figures~\ref{F1}-\ref{F2}-\ref{F3}):
(i) $\ga\in(1,2]$ and $\zeta\in(-1,1/2)$;
(ii) $\alpha \in (0,1)$ and $\zeta\in(-1,0]$;
(iii) $\alpha \in (\frac12,1)$ and $\zeta \in (0,\frac12)$;
(iv) $\alpha \in (0,\frac12)$ and $\zeta \in (0,\frac12)$.
We deal with the first two ones at the same time and we treat the third and fourth one afterwards since the strategy of the proof is slightly different.

%

\smallskip
\textit{Cases (i)-(ii).} Let us write
\begin{multline*}
Z_{N}^{\omega} \big(M_{N}^{*} > AN^{\xi}\big)
 = \sum\limits_{k=1}^{\log_{2}(\frac1A N^{1-\xi}) } Z_{N}^{\omega}\left(M_{N}^{*}\in(2^{k-1}AN^{\xi},2^{k}AN^{\xi}]\right)  \\ 
 \leq \sum\limits_{k=1}^{\log_{2}(\frac1A N^{1-\xi}) } \exp\left( C(A,\hat \gb,\go) \, 2^{k/\alpha}  N^{\frac{\xi}{\alpha} -\gamma} (\log _2 N)^{2/\alpha}-\hat{h}2^{k-1}AN^{\xi-\zeta} \right) 
\end{multline*}
where we have used Lemma~\ref{lem1} to bound $\Sigma_{2^{k} AN^{\xi}}^{*}$ by $C(\go) 2^{k/\alpha} A^{1/\alpha} N^{\xi/\alpha} (\log _2 N)^{2/\alpha}$, $\bbP$-a.s..
Now, uniformly for $k$ in the sum we have
\[
\frac{2^{k/\alpha}  N^{\frac{\xi}{\alpha} -\gamma} (\log _2 N)^{2/\alpha}}{ 2^{k} N^{\xi-\zeta} } \leq 
(\log _2 N)^{2/\alpha} 
\begin{cases}
 N^{\frac{\xi}{\alpha} -\gamma - (\xi-\zeta)}  & \quad \text{ if } \alpha\in (1,2]\,,\\
A^{-\frac{1-\alpha}{\alpha} } N^{\frac{1-\alpha}{\alpha} +\zeta-\gamma}  & \quad \text{ if } \alpha\in (0,1)\,.
\end{cases}
\]
Notice that in cases (i)-(ii) the upper bound always goes to $0$ as $N\to\infty$,
because $\frac{\xi}{\alpha} -\gamma > \xi-\zeta$ in the case $\alpha\in (1,2]$ and $\gamma > \zeta - \frac{\alpha-1}{\alpha}$
in the case $\alpha\in (0,1)$.
Therefore, $\bbP$-a.s., for $N$ large enough, we have
\[
Z_{N}^{\omega} \big(M_{N}^{*} > AN^{\xi}\big)
\leq \sum\limits_{k=1}^{\log_{2}(\frac1A N^{1-\xi}) } \exp\left( -\hat{h}2^{k-2} AN^{\xi-\zeta} \right) 
\leq C \exp\Big( - \frac12 \hat{h} AN^{\xi-\zeta} \Big) \,.
\]
Since $\xi-\zeta = 1-2\xi$, this concludes the proof.

\smallskip
\textit{Cases (iii)-(iv)}. In that case, we have $\xi\in(0,1/2)$ and $\zeta\in(0,1/2)$. Hence, we can write
\begin{equation}\label{1-zeta}
Z_N^\omega(M_N^*>AN^\xi)=Z_N^\omega(M_N^*\in(AN^\xi,N^{1-\zeta}])+Z_N^\omega(M_N^*\in(N^{1-\zeta},N]).
\end{equation}
For the first term, we have similarly as above that
$Z_N^\omega(M_N^*\in(AN^\xi,N^{1-\zeta}])$
is bounded by 
\begin{equation*}
\sum\limits_{k=1}^{\log_2(\frac{1}{A}N^{1-\xi-\zeta})}\exp\left(C(A,\hat{\beta},\omega)2^{k/\ga}N^{\frac{\xi}{\ga}-\gamma}(\log_2 N)^{2/\ga}-\hat{h}2^{k-1}AN^{\xi-\zeta}\right)\,.
\end{equation*}
Now, uniformly for $k$ in the sum we have
\begin{equation*}
\frac{2^{k/\alpha}  N^{\frac{\xi}{\alpha} -\gamma} (\log _2 N)^{2/\alpha}}{ 2^{k} N^{\xi-\zeta} }\leq(\log_2N)^{2/\ga}A^{-\frac{1-\ga}{\ga}}N^{-\gamma+\frac{2\ga-1}{\ga}\zeta+\frac{1-\ga}{\ga}}\,,
\end{equation*}
with the upper bound vanishing, because in cases (iii)-(iv) we have 
$\gamma>\frac{2\ga-1}{\ga}\zeta+\frac{1-\ga}{\ga}$ for $\zeta\in(0,1/2)$ (note that if $\alpha<\frac12$, $\frac{2\ga-1}{\ga}<0$). Therefore, $\bbP$-a.s., for $N$ large enough, we get
\begin{equation}\label{1-zeta_a}
Z_{N}^{\omega} \big(M_{N}^{*} \in(AN^{\xi},N^{1-\zeta}]\big)
\leq \sum\limits_{k=1}^{\log_{2}(\frac1A N^{1-\xi-\zeta}) } \exp\left( -\hat{h}2^{k-2} AN^{\xi-\zeta} \right) 
\leq C \exp\Big( - \frac12 \hat{h} AN^{\xi-\zeta} \Big)\,.
\end{equation}
For the second term on the right-hand side of \eqref{1-zeta}, we have
\begin{equation*}
\begin{split}
Z_N^\omega(M_N^*\in(N^{1-\zeta},N])&=\sum\limits_{k=1}^{\log_2(N^\zeta)}Z_N^\omega(M_N^*\in(2^{-k}N,2^{-k+1}N])\\
&\leq\sum\limits_{k=1}^{\log_2(N^\zeta)}2\exp\left(C(\hat{\beta},\omega)2^{-k/\ga}N^{\frac{1}{\ga}-\gamma}(\log_2N)^{2/\ga}-2^{-2k-1}N\right),
\end{split}
\end{equation*}
where we have used Lemma~\ref{lem1} to bound $\sum_{2^-kN}^*$ by $C(\omega)2^{-k/\ga}N^{\frac{1}{\ga}-\gamma}(\log_2 N)^{2/\ga},\bbP$-a.s. and also the fact that $\bP(M_N^*>x)\leq2\exp(-\frac{x^2}{2N})$.
Now, uniformly for $k$ in the sum, we have
\begin{equation*}
\frac{ 2^{-k/\ga}N^{\frac{1}{\ga}-\gamma-1}(\log_2N)^{2/\ga}}{2^{-2k} N}\leq(\log_2N)^{2/\ga}
\begin{cases}
N^{-\gamma+\frac{2\ga-1}{\ga}\zeta+\frac{1-\ga}{\ga}}  &\text{if}\,\, \alpha \in(\frac12,1),\\
N^{\frac{1-\alpha}{\alpha}-\gamma}  &\text{if}\,\, \alpha \in(0,\frac12) \,,
\end{cases}
\end{equation*}
which goes to $0$ in cases (iii)-(iv), because
$\gamma>\frac{2\ga-1}{\ga}\zeta+\frac{1-\ga}{\ga}$ if $\alpha\in(\frac12,1)$
and $\gamma>\frac{1-\ga}{\ga}$ if $\alpha\in(0,\frac12)$. Therefore, $\bbP$-a.s., for $N$ large enough, we have
\begin{equation}\label{1-zata_b}
Z_N^\omega(M_N^*\in(N^{1-\zeta},N])\leq C\log_2 N\exp\left(-\frac{1}{4}N^{1-2\zeta}\right)\leq\exp\left(-\hat hAN^{\xi-\zeta}\right),
\end{equation}
where in the last inequality we have used the fact that $1-2\zeta-(\xi-\zeta)=1-\xi-\zeta>0$, since $\xi<\frac12$ and $\zeta<\frac12$.

Combining \eqref{1-zeta_a}-\eqref{1-zata_b} with \eqref{1-zeta} and since $\xi-\zeta=1-2\xi$, this concludes the proof.
\end{proof}

\subsubsection*{Convergence of $M_N^+-M_N^-$}
Let us define $c_{\hat h}:=\pi^{\frac{2}{3}}\hat{h}^{-\frac{1}{3}}$. Let $\epsilon>0$ and define 
the event 
\[
\cA_{N,5}^{\gep} = \Big\{ \Big|\frac{1}{N^{\xi}} (M_N^+ -M_N^-)  - c_{\hat h} \Big|  > \gep \Big\}.
\]
As in the previous sections, since 
$\log \bP_{N}^{\go} (\cA_{N,5}^{\gep}  ) = 
\log Z_{N}^{\go} ( \cA_{N,5}^{\gep}  )- \log Z_{N}^{\go}$,
using the convergence~\eqref{E121}
we simply need to show that
there is some $\gd_{\gep} >0$ such that 
\begin{equation*}
\lim_{N\to\infty} \bbP\Big( \frac{1}{N^{1-2\xi}} \log Z_{N}^{\go} ( \cA_{N,5}^{\gep} )
 <  -\frac{3}{2} (\hat h \pi)^{2/3} - \gd_{\gep} \Big)   = 1
\end{equation*}
But this is simply due to the fact that analogously to Lemma~\ref{R5LM1}, we have $\bbP$-a.s.
\[
\lim_{N\to+\infty}
 \frac{1}{N^{1-2\xi}} \log Z_{N}^{\go} (\cA_{N,5}^{\gep}) 
 =
 \sup_{ u,v\geq 0 , |u+v -c_{\hat h}| >\gep } \left\{-\hat{h}(u+v)-\bar{I}(u,v)\right\}  < -\frac{3}{2} (\hat h \pi)^{2/3}  \,,
\]
where the inequality is strict since the supremum 
in Lemma~\ref{R5LM1} is attained for ${u+v = c_{\hat h}}$.
\qed

\subsection{Region R6: proof of Theorem \ref{C105}} 
Recall that in Region $R_6$,
we have 
\[
\zeta<(-1)\wedge\gamma \quad \text{ if } \alpha \in (1,2]
\qquad \text{ and } \qquad \zeta <(-1)\wedge \Big(\gamma + \frac{\ga-1}{\ga} \Big) \quad \text{ if } \alpha\in (0,1) \,.
\]
Let us note that in all cases we have $\gamma >\gz$.
We split $Z_{N}^{\omega}$ in two parts
\begin{equation}\label{E601}
Z_{N}^{\omega}=Z_{N}^{\omega}\left( |\cR_N| =2\right)+Z_{N}^{\omega}\left( |\cR_N| \geq 3\right).
\end{equation}
It is clear that 
\begin{align*}
Z_{N}^{\omega}\left( |\cR_N| =2\right) = e^{- 2 \hat h N^{-\zeta} } \Big( e^{\hat \gb N^{-\gamma} (\go_0+\go_1)} 2^{-N} + e^{\hat \gb N^{-\gamma} (\go_0+\go_{-1})}  2^{-N}\Big)\, ,
\end{align*}
so that $N^{\zeta} \log Z_{N}^{\omega}\left( |\cR_N| =2\right)$
converges $\bbP$-a.s.\ to $-2\hat h$, since $\zeta<\gamma$ and $\zeta<-1$.

\smallskip
We now prove that $\limsup_{N\to\infty} N^{\zeta} \log Z_{N}^{\omega}\left( |\cR_N| \geq 3\right)$ is strictly smaller than $- 2 \hat h $ a.s.: this will imply that the second term in \eqref{E601} is negligible compared to the first one and 
as a consequence prove that $\bP_{N}^{\go} (|\cR_N|=2)$ converges to $1$ $\bbP$-a.s.

We write
\[
Z_{N}^{\omega}\left( |\cR_N| \geq 3\right)
= \sum_{k=3}^{N} Z_{N}^{\omega}\left( |\cR_N| = k\right)
\leq  \sum_{k=3}^{N} \exp\Big( -k \hat h  N^{-\zeta} + C(\go) \, \hat \gb  N^{-\gamma}  k^{1/\alpha}   (\log _2 N)^{2/\alpha} \Big) \,,
\]
where we have used Lemma~\ref{lem1}
to bound $\Sigma_{k}^{*}$ for the last inequality.
Now, uniformly for $k$ in the sum we have
\[
\frac{k^{1/\alpha}  N^{ -\gamma} (\log _2 N)^{2/\alpha}}{ k N^{-\zeta} } \leq 
(\log _2 N)^{2/\alpha} 
\begin{cases}
 N^{\zeta -\gamma}  & \quad \text{ if } \alpha\in (1,2]\,,\\
 N^{\frac{1-\alpha}{\alpha} +\zeta-\gamma}  & \quad \text{ if } \alpha\in (0,1)\,.
\end{cases}
\]
The upper bound always goes to $0$ as $N\to\infty$,
because $\gamma > \zeta$ in the case $\alpha\in (1,2]$ and $\gamma > \zeta - \frac{\alpha-1}{\alpha}$
in the case $\alpha\in (0,1)$.
Therefore, $\bbP$-a.s., for $N$ large enough, we have
\[
Z_{N}^{\omega} \big(M_{N}^{*} > AN^{\xi}\big)
\leq \sum\limits_{k=3}^{N} \exp\left( - \frac{5}{6} k\hat{h} N^{-\zeta} \right) 
\leq C \exp\left( - \frac{5}{2} \hat{h} N^{-\zeta} \right) 
\,.
\]
We therefore get $\limsup_{N\to\infty} N^{\zeta} \log Z_{N}^{\omega}\left( |\cR_N| \geq 3\right) \leq  - \frac{5}{2} \hat{h} < -2 \hat h$ a.s.,
which concludes the proof.
\qed

\section{Proof of the remaining results: the case $\hat h<0$}
\label{sec:proof2}

In all this section, we consider only the case $\hat h<0$.

\subsection{Region~$\tilde R_4$: proof of Theorem~\ref{thm:R4tilde}}
\label{sec:tildeR4}

In this region, we prove that the range size is of order $N^{\xi}$ with $\xi=1-\zeta \in (\frac12,1)$. 
Recall that in Region~$\tilde R_4$ we have 
\[
2\xi-1=\xi-\zeta > \frac{\xi}{\ga}-\gamma\, \quad \text{and} \quad 
0<\zeta<\frac{1}{2}\, .
\]
The proofs are almost identical to what is done in regions~$R_5$-$R_6$,
so we give much less detail.
We fix some constant $A = A_{\hat h} := 32 |\hat h|$ and we split the partition function as
\begin{equation*}\label{E401tilde}
Z_{N}^{\omega}=Z_{N}^{\omega}\left(M_N^*\leq AN^{\xi}\right)+Z_{N}^{\omega}\left(M_N^*> AN^{\xi}\right).
\end{equation*}

\paragraph*{Step 1.} We have the following lemma, analogous to Lemma \ref{R5LM1}.
\begin{lemma}\label{tildeR4LM1}
In Region $\tilde R_4$, for any $A>0$ we have the following convergence
\begin{equation*}
\lim_{N\to+\infty}\frac{1}{N^{2\xi-1}}\log Z_{N}^{\omega}\left(M_{N}^{*}\leq AN^{\xi}\right) = \sup\limits_{ u,v\in [0,A] }\left\{|\hat{h}|(u+v)-I(u,v)\right\}
\qquad \text{$\bbP$-a.s.}
\end{equation*}
By a simple calculation, the supremum is $\frac{1}{2}{\hat{h}^2}$ for any $A\geq |\hat h|$ and it is attained at $(u,v) = (0,|\hat h|)$ or $(u,v) = (|\hat h|,0)$.
\end{lemma}
\begin{proof}
Since in Region $\tilde R_4$ we have $ 2\xi -1  >  \frac{\xi}{\ga} -\gamma$, for any fixed $A$
we get that $N^{-(2\xi-1)}  \times \hat \gb N^{-\gamma} \Sigma_{AN^{\xi}}^*$ almost surely goes to $0$. Therefore
we only need to prove that
\begin{equation}
\label{conv:variational}
\lim_{N\to\infty }\frac{1}{N^{2\xi-1}} \log Z_{N}^{0} \left(M_{N}^{*}\leq AN^{\xi}\right) = \sup\limits_{ u,v\in [0,A] }\left\{|\hat{h}|(u+v)-{I}(u,v)\right\} \,,
\end{equation} 
where $Z_{N}^0$ denotes the partition function with $\go\equiv 0$
(or equivalently $\gb_N \equiv 0$).

But~\eqref{conv:variational} follows from Varadhan's lemma, since $I(u,v)$ is the rate function for the LDP for $(N^{-\xi} M_N^-$, $N^{-\xi}M_N^+)$, by Lemma~\ref{lem:superdiffusive}.
\end{proof}

\paragraph*{Step 2} To conclude the proof of the convergence~\eqref{Etilde123}, it remains to show the  following.
\begin{lemma}\label{tildeR4LM2}
In Region $\tilde R_4$, 
we have for any $A\geq 32 |\hat h|$
\[
\limsup_{N\to\infty} \frac{1}{N^{2\xi-1}}\log Z_{N}^{\omega}\left(M_{N}^{*}>AN^{\xi}\right) \leq  0 \qquad \text{$\bbP$-a.s.}
\]
\end{lemma}
\noindent
Together with Lemma~\ref{tildeR4LM1}, this readily yields
the convergence~\eqref{Etilde123}.

\begin{proof}
We write
\begin{multline*}
Z_{N}^{\omega} \big(M_{N}^{*} > AN^{\xi}\big)
 = \sum\limits_{k=1}^{\log_{2}(\frac1A N^{1-\xi}) } Z_{N}^{\omega}\left(M_{N}^{*}\in(2^{k-1}AN^{\xi},2^{k}AN^{\xi}]\right)  \\ 
 \leq \sum\limits_{k=1}^{\log_{2}(\frac1A N^{1-\xi}) } 2 \exp\left( C(A,\hat \gb,\go) \, 2^{k/\alpha}  N^{\frac{\xi}{\alpha} -\gamma} (\log _2 N)^{2/\alpha} +2^{k+1}  |\hat h|  A N^{\xi-\zeta} 
 - 2^{2k-3} A^2 N^{2\xi-1}\right) \,, 
\end{multline*}
where we have used Lemma~\ref{lem1} to get an almost sure bound on $\Sigma_{2^{k} AN^{\xi}}^{*}$,
and also the fact that $\bP(M_{N}^{*} >x) \leq 2 \exp(- \frac{x^2}{2N})$.
Now, as in the proof of Lemma~\ref{R5LM2},
the ``disorder'' term is seen to be negligible compared to the ``range'' term (uniformly for $k$ in the sum):
we get that $\bbP$-a.s., for $N$ large enough,
\[
Z_{N}^{\omega} \big(M_{N}^{*} > AN^{\xi}\big)
 \leq \sum\limits_{k=1}^{\log_{2}(\frac1A N^{1-\xi}) } 2 \exp\left( 2^{k+2}  |\hat h|  A N^{\xi-\zeta} 
 - 2^{2k-3} A^2 N^{2\xi-1}\right) \leq  2 \log_{2}(\tfrac1A N^{1-\xi}) \,,
\]
where for the last inequality we have used that $\xi-\zeta =2\xi-1$ and that $A\geq 2^5 |\hat h|$. 
This concludes the proof.
\end{proof}

\paragraph*{Convergence of trajectories.}

First of all, let us go one step further in the proof of Lemma~\ref{tildeR4LM1}. Indeed, in~\eqref{conv:variational} the supremum in the variational problem
is attained at $(u,v) = (0,|\hat h|)$ or $(u,v) = (|\hat h|,0)$,
so we can deduce that the main contribution 
to~$Z_{N}^{0}$ (hence to $Z_{N}^{\go}$ in view of the proof of Lemma~\ref{tildeR4LM1}) comes from trajectories with  $N^{-\xi} (M_{N}^-,M_N^+)$ either close to $(0,|\hat h|)$ or to $(-|\hat h|,0)$.
One can actually show that the main contribution comes from
trajectories moving at roughly constant speed to these endpoints
(similarly to Proposition~\ref{prop:ballistic}):
using~\eqref{ldp:ballistic}, one easily gets that analogously to~\eqref{conv:variational}, for any $\gep>0$,
\begin{equation}
\limsup_{N\to\infty }\frac{1}{N^{2\xi-1}} \log Z_{N}^{0} \left(M_{N}^{*}\leq AN^{\xi}, (\cB_N^{+,\gep} \cup \cB_N^{-,\gep} )^c\right) < \sup\limits_{ u,v\in [0,A] }\left\{|\hat{h}|(u+v)-{I}(u,v)\right\} \,,
\end{equation}
where we recall the definition of the events $\cB_N^{\pm,\gep}$ (recall $\hat h<0$):
\[
\cB_N^{+,\gep} :=\Big \{  \sup_{t\in [0,1]} \big| N^{-\xi} S_{\lfloor tN\rfloor} +\hat h\, t \big| \leq \gep \Big\} \,,\qquad
\cB_N^{-,\gep} :=\Big \{  \sup_{t\in [0,1]} \big| N^{-\xi} S_{\lfloor tN\rfloor} -\hat h\, t \big| \leq \gep \Big\} \,.
\]
All together, in view of the fact that  $N^{-(2\xi-1)}  \times \hat \gb N^{-\gamma} \Sigma_{AN^{\xi}}^*$ goes to $0$ a.s.,
we get that 
\[
\limsup_{N\to\infty }\frac{1}{N^{2\xi-1}} \log Z_{N}^{\go} \left( (\cB_N^{+,\gep} \cup \cB_N^{-,\gep} )^c\right) < 
\lim_{N\to\infty }\frac{1}{N^{2\xi-1}} \log Z_{N}^{\go} \,,
\]
from which one deduces that
\begin{equation*}
\lim_{N\to\infty}\bP_N^{\go} (\cB_N^{+,\gep} \cup \cB_N^{-,\gep} ) = 1 \qquad \text{$\bbP$-a.s.}
\end{equation*}
Given $\hat{h}$, the events $\cB_N^{\pm,\gep}$ are disjoint for $\gep$ small enough: this implies in particular that 
\begin{equation}
\label{eq:twoprobabto1}
\lim_{N\to\infty}\bP_N^{\go} ( \cB_N^{+,\gep} ) + \bP_N^{\go}( \cB_N^{-,\gep} ) = \lim_{N\to+\infty} \frac{Z_{N}^{\go} ( \cB_N^{+,\gep} ) +  Z_N^{\go}(\cB_N^{-,\gep} )}{Z_{N}^{\go}} =1 \,.
\end{equation}

Now, denoting again $Z_{N}^0$ the partition function with $\go\equiv 0$
(or equivalently $\gb_N \equiv 0$)
and $\bP_N^0$ the corresponding measure, we have
\[
e^{\gb_N \Sigma_{ (|\hat h|-\epsilon) N^{\xi} }^{+} -\gb_N  R_N^{2\gep}(0,|\hat h|-\epsilon) }  \bP_{N}^0 ( \cB_N^{+,\gep}  ) \leq 
\frac{Z_{N}^{\go} ( \cB_N^{+,\gep}  )}{ Z_{N}^0 } \leq e^{ \gb_N \Sigma_{ (|\hat h|-\gep) N^{\xi}  }^{+} +\gb_N  R_N^{2\gep}(0,|\hat h|-\gep) } \bP_{N}^0 ( \cB_N^{+,\gep}  ) \,,
\]
where $R_N^{\gep}(u,v)$ is defined in~\eqref{E306}.
A similar inequality holds with $\cB_N^{-,\gep} $ in place of $\cB_N^{+,\gep} $, simply by replacing $\Sigma_{ (|\hat h|-\gep) N^{\xi}  }^{+}$ by $\Sigma_{ (|\hat h|-\gep) N^{\xi}  }^{-}$ and $R_N^{2\gep}(0,|\hat h|-\gep)$ by $R_N^{2\gep}(|\hat h|-\gep,0)$.

Therefore, we have 
\begin{multline}
\frac{Z_{N}^{\go} ( \cB_N^{+,\gep} )}{Z_{N}^{\go} ( \cB_N^{+,\gep} ) +  Z_N^{\go}(\cB_N^{-,\gep} )}   \\
  \leq \frac{ e^{ \gb_N \Sigma_{ (|\hat h|-\gep) N^{\xi}  }^{+} +\gb_N  R_N^{2\gep}(0,|\hat h|-\gep) }  \bP_{N}^0 (\cB_N^{+,\gep} ) }{ e^{ \gb_N \Sigma_{ (|\hat h|-\gep) N^{\xi}  }^{+} -\gb_N  R_N^{2\gep}(0,|\hat h|-\gep) } \bP_{N}^0  (\cB_N^{+,\gep} )  +  e^{ \gb_N \Sigma_{ (|\hat h|-\gep) N^{\xi}  }^{-} -\gb_N  R_N^{2\gep}(|\hat{h}|-\gep,0) } \bP_{N}^0  (\cB_N^{-,\gep} ) } \,,
\label{eq:splitprobab}
\end{multline}
and similarly
\begin{multline}
\frac{Z_{N}^{\go} ( \cB_N^{+,\gep} )}{Z_{N}^{\go} ( \cB_N^{+,\gep} ) +  Z_N^{\go}(\cB_N^{-,\gep} )}  \\
  \geq \frac{ e^{ \gb_N \Sigma_{ (|\hat h|-\gep) N^{\xi}  }^{+} -\gb_N  R_N^{2\gep}(0,|\hat h|-\gep) }  \bP_{N}^0 (\cB_N^{+,\gep} ) }{ e^{ \gb_N \Sigma_{ (|\hat{h}|-\gep)N^{\xi}  }^{+} +\gb_N  R_N^{2\gep}(0,|\hat h|-\gep) } \bP_{N}^0  (\cB_N^{+,\gep} )  +  e^{ \gb_N \Sigma_{ (|\hat h|-\gep) N^{\xi}  }^{-} +\gb_N  R_N^{2\gep}(|\hat h|-\gep,0) } \bP_{N}^0  (\cB_N^{-,\gep} ) } \,.
\label{eq:splitprobab2}
\end{multline}
Let us make a few observations. First of all, 
notice also that by symmetry we get that  when $\go\equiv 0$, for any $|\hat h|>0$
\begin{equation}
\label{eq:probabto1/2}
\lim_{N\to\infty}\bP_{N}^{0} \big(\cB_N^{+,\gep}  \big)
 = \lim_{N\to+\infty} \bP_{N}^0(\cB_N^{-,\gep} ) = \frac12 \,.
\end{equation}
Recall that $\lim_{N\to+\infty} N^{-\xi/\ga} \Sigma_{ (|\hat h|-\gep) N^{\xi} }^{+} =X_{|\hat h|-\gep}^{(1)} $  and $\lim_{N\to+\infty} N^{-\xi/\ga} \Sigma_{ (|\hat h|-\gep) N^{\xi} }^{-} = X_{|\hat h|-\gep}^{(2)}$  $\hat \bbP$-a.s.,
and that for fixed $\hat{h}$, $\hat\bbP$-a.s.\ the two processes $X^{(1)}_t$ and $X^{(2)}_t$ are both continuous at $t=|\hat{h}|$, so $\lim_{\gep\downarrow0}X^{(1)}_{|\hat{h}|-\gep}=X^{(1)}_{|\hat{h}|}$ and $\lim_{\gep\downarrow0}X^{(2)}_{|\hat{h}|-\gep}=X^{(2)}_{|\hat{h}|}$.
Note also that, for $\hat\bbP$-almost every realization of $ \go$, for any $\delta$ one can choose $\gep>0$ small enough so that
$N^{-\xi/\ga} R_N^{2\gep}(0,|\hat h|-\gep) \leq \delta$ for all $N$ large enough, 
and similarly for $N^{-\xi/\ga} R_N^{2\gep}(|\hat h|-\gep,0)$.

Let us now consider three cases.

\smallskip
\noindent
(i) If $\gamma < \xi/\ga$.
On the event $X_{|\hat h|}^{(1)} <X_{|\hat h|}^{(2)}$, one can choose $\gep>0$ small enough so that 
\[
\liminf_{N\to+\infty}
N^{-\xi/\ga} \Big( \Sigma_{ (|\hat h|-\gep) N^{\xi}  }^{-} +  R_N^{2\gep} (|\hat h|-\gep,0)   -\big(  \Sigma_{ (|\hat h|-\gep) N^{\xi}  }^{+}  +  R_N^{2\gep} (0, |\hat h|-\gep) \big) \Big) >0 \,.
\]
Since $\gb_N =\hat \gb N^{\frac{\xi}{\ga} -\gamma} N^{-\frac{\xi}{\ga}}$ 
with $N^{\frac{\xi}{\ga}-\gamma} \to +\infty$,
from~\eqref{eq:splitprobab} we deduce that a.e.\ on the event $X_{|\hat h|}^{(1)} <X_{|\hat h|}^{(2)}$, for $\gep>0$ small enough 
\[
\limsup_{N\to+\infty} \bP_{N}^{\go} ( \cB_N^{+,\gep} )
=
\limsup_{N\to+\infty} \frac{Z_{N}^{\go} ( \cB_N^{+,\gep} )}{Z_{N}^{\go} ( \cB_N^{+,\gep} ) +  Z_N^{\go}(\cB_N^{-,\gep} )}  =0\,,
\]
recalling also~\eqref{eq:probabto1/2}.
By an identical reasoning, 
we get that a.e.\ on the event $X_{|\hat h|}^{(1)} >X_{|\hat h|}^{(2)}$, for $\gep>0$ small enough $\limsup_{N\to+\infty} \bP_{N}^{\go} ( \cB_N^{-,\gep} ) =0$. Hence, because the event $X_{|\hat h|}^{(1)} =X_{|\hat h|}^{(2)}$ has probability $0$ and recalling~\eqref{eq:twoprobabto1}, we can conclude that 
\[
\lim_{\gep\downarrow 0} \lim_{N\to+\infty} \bP_{N}^{\go} ( \cB_N^{+,\gep} ) = \ind_{\{ X_{|\hat h|}^{(1)} >X_{|\hat h|}^{(2)}\}}   \qquad \text{$\hat \bbP$-a.s.} \,,
\]
where the limit in $N$ is well-defined provided that $\gep$ is small enough.
A similar statement holds for $\bP_{N}^{\go} ( \cB_N^{-,\gep} )$,
exchanging the role of $X^{(1)}$ and $X^{(2)}$.

\smallskip
\noindent
(ii)  If $\gamma = \xi/\ga$, then similarly as above, from~\eqref{eq:splitprobab}  and \eqref{eq:splitprobab2} we deduce that
for any $\delta>0$,
$\hat \bbP$-a.s. we can choose $\gep>0$ small enough so that
\[
\limsup_{N\to+\infty} \bP_{N}^{\go} ( \cB_N^{+,\gep} )
=
\limsup_{N\to+\infty} \frac{Z_{N}^{\go} ( \cB_N^{+,\gep} )}{Z_{N}^{\go} ( \cB_N^{+,\gep} ) +  Z_N^{\go}(\cB_N^{-,\gep} )}  \leq \frac{e^{\hat \gb X_{|\hat h|-\gep}^{(1)} +\delta}}{ e^{\hat \gb X_{|\hat h|-\gep}^{(1)} -\delta} + e^{\hat \gb X_{|\hat h|-\gep}^{(2)} -\delta}  }\,,
\]
\[
\liminf_{N\to+\infty} \bP_{N}^{\go} ( \cB_N^{+,\gep} )
=
\liminf_{N\to+\infty} \frac{Z_{N}^{\go} ( \cB_N^{+,\gep} )}{Z_{N}^{\go} ( \cB_N^{+,\gep} ) +  Z_N^{\go}(\cB_N^{-,\gep} )}  \geq \frac{e^{\hat \gb X_{|\hat h|-\gep}^{(1)} -\delta}}{ e^{\hat \gb X_{|\hat h|-\gep}^{(1)} +\delta} + e^{\hat \gb X_{|\hat h|-\gep}^{(2)} +\delta}  }\,,
\]
recalling again~\eqref{eq:twoprobabto1} and~\eqref{eq:probabto1/2}.
Taking $\delta$ arbitrarily small, we get that
\[
\lim_{\gep\downarrow0}\limsup_{N\to+\infty} \bP_{N}^{\go} ( \cB_N^{+,\gep} )
=\lim_{\gep\downarrow0}\liminf_{N\to+\infty} \bP_{N}^{\go} ( \cB_N^{+,\gep} )
 = \frac{e^{\hat \gb X_{|\hat h|}^{(1)}}}{ e^{\hat \gb X_{|\hat h|}^{(1)} } + e^{\hat \gb X_{|\hat h|}^{(2)}}  } \qquad \text{$\hat\bbP$-a.s.}
\]
The statement is analogous for $\bP_{N}^{\go} ( \cB_N^{-,\gep} )$,
exchanging the role of $X^{(1)}$ and $X^{(2)}$.

\smallskip
\noindent
(iii) If $\gamma > \xi/\ga$, since $\gb_N =\hat \gb N^{\frac{\xi}{\ga} -\gamma} N^{-\frac{\xi}{\ga}}$ and  $N^{\frac{\xi}{\ga} -\gamma}\to 0$, we get from~\eqref{eq:splitprobab}  and \eqref{eq:splitprobab2} that for any $\gep>0$
\[
\lim_{N\to+\infty} \bP_{N}^{\go} ( \cB_N^{+,\gep} )
=
\lim_{N\to+\infty} \frac{Z_{N}^{\go} ( \cB_N^{+,\gep} )}{Z_{N}^{\go} ( \cB_N^{+,\gep} ) +  Z_N^{\go}(\cB_N^{-,\gep} )}  = \frac12 \qquad\text{$\hat\bbP$-a.s.},
\]
recalling again~\eqref{eq:probabto1/2}.
We also get 
$\lim_{N\to+\infty} \bP_{N}^{\go} ( \cB_N^{-,\gep} ) = \frac12$,
$\hat\bbP$-a.s.

This concludes the proof of~\eqref{eq:ballistic1}.
\qed

\subsection{Boundary region~$\tilde R_4 $---$\tilde R_5$: proof of Theorem~\ref{R4tilde}}
\label{sec:tildeR4R5}

The proof is similar to that for region $\tilde R_4$: one only needs the analogous to Lemma~\ref{tildeR4LM1}.
The rate function $I(u,v)$ for $(N^{-\xi} M_N^{-}, N^{-\xi} M_N^+)$ is replaced by the rate function $\kappa( u\wedge v +u+v)$ for $(N^{-1} M_N^{-},N^{-1} M_N^{+})$, see Lemma~\ref{lem:superdiffusive2}.
We end up with:
\[
\lim_{N\to+\infty} \frac{1}{N} \log Z_{N}^{\go} = \sup_{u,v\in [0,1]} \big\{ |\hat h| (u+v) - \kappa (u\wedge v +u+v) \big\} \quad \text{$\bbP$-a.s.}
\]
Then, using that $\kappa(t) = \frac12 (1+t) \log (1+t) + \frac12 (1-t) \log(1-t)$ if $0\leq t\leq 1$ and $\kappa(t) = +\infty$ if $t>1$,
a straightforward calculation finds that 
the supremum is attained at 
$(u,v) = (0,  \tanh |\hat h|)$ or
$(u,v) = ( \tanh |\hat h| ,0)$
and equals $ \log( \sinh |\hat h|)$.

Then, by following the same ideas as above, 
one can show that the events 
\[
\cB_N^{+,\gep} :=\Big \{  \sup_{t\in [0,1]} \big| N^{-1} S_{\lfloor tN\rfloor} -\tanh(|\hat h|) t \big| \leq \gep \Big\} \,,\quad
\cB_N^{-,\gep} :=\Big \{  \sup_{t\in [0,1]} \big| N^{-1} S_{\lfloor tN\rfloor} + \tanh(|\hat h|) t \big| \leq \gep \Big\} \,,
\]
verify
\[
\lim_{N\to\infty}\bP_N^{\go} (\cB_N^{+,\gep} \cup \cB_N^{-,\gep} ) = 1 \qquad \text{$\bbP$-a.s.}
\]
From this, one can proceed as above (see in particular~\eqref{eq:splitprobab} and~\eqref{eq:splitprobab2})
to get~\eqref{eq:ballistic2}.
Details are left to the reader.

\subsection{Region~$\tilde R_5$: proof of Theorem~\ref{R5tilde}}

First of all, notice that
\begin{equation}
\label{eq:convtildeR5}
Z_N^{\go} \geq Z_N^{\go}(|S_N|=N) \geq e^{-\hat \beta N^{-\gamma}\Sigma_N^*-\hat h N^{1-\zeta} -N \log 2} \quad \text{ and }\quad 
Z_N^{\go} \leq e^{-\hat h N^{1-\zeta} + \hat \beta N^{-\gamma}\Sigma_N^*} \,.
\end{equation}
Hence, the $\bbP$-a.s.\ convergence $\lim_{N\to\infty} N^{\zeta-1} \log Z_N^{\go} = |\hat h|$ is immediate, since $\zeta<0$ and $\frac1\ga - \gamma < 1-\zeta$.

The rest of the proof of Theorem~\ref{R5tilde}
is similar to what is done in Sections~\ref{sec:tildeR4}-\ref{sec:tildeR4R5} above.
In particular, for any $\gep>0$ one has that 
\[
\limsup_{N\to\infty} N^{\zeta-1}\log Z_N^{\go} \big(  |S_N| \leq (1-\epsilon)N \big) \leq -  (1-\tfrac{1}{2}\gep )\hat h 
\qquad \text{$\bbP$-a.s.},
\]
so that $\lim_{N\to\infty}\bP_N^{\go}(  |S_N| \geq (1-\epsilon)N )=1$ $\bbP$-a.s.
Then, by following the same ideas as above, 
one can show that the events 
\[
\cB_N^{+,\gep} :=\Big \{  \sup_{t\in [0,1]} \big| N^{-1} S_{\lfloor tN\rfloor} - t \big| \leq \gep \Big\} \,,\quad
\cB_N^{-,\gep} :=\Big \{  \sup_{t\in [0,1]} \big| N^{-1} S_{\lfloor tN\rfloor} + t \big| \leq \gep \Big\} \,,
\]
verify
\[
\lim_{N\to\infty}\bP_N^{\go} (\cB_N^{+,\gep} \cup \cB_N^{-,\gep} ) = 1 \qquad \text{$\bbP$-a.s.}
\]
From this, one can proceed as above
to get~\eqref{eq:ballistic3}.
Details are left to the reader.

\paragraph*{Improvement in the case $\alpha\in(0,1)$ or $\alpha\in(1,2]$ and $\gamma>\zeta$.}

We now prove~\eqref{Rn=n}.
As mentioned above, we have 
$\lim_{N\to\infty} \bP_N^\omega\left(|S_N|\geq (1-\gep) N\right) = 1$
almost surely.
Now, we can split the event $|S_N| > (1-\gep) N$ according to whether $M_N^+ >\frac12 N$ or 
$M_N^{-} < -\frac12 N$. Hence, we only have to prove that 
$\bP_{N}^{\go}( \frac12 N < M_N^+ \leq  N-1 )/ \bP_N^{\go}( |S_N| =N  )$ a.s.\ goes to $0$, and similarly for $M_N^-$.

To this end, we show the following:
$\bbP$-a.s., for $N$ large enough we have
\begin{equation}
\label{tildeR6eq}
  \frac{Z_N^{\go}( \frac12 N < M_N^+ \leq  N-1 )}{Z_N^{\go}( |S_N| =N  )}  \leq 
  \frac{Z_N^{\go}( \frac12 N < M_N^+ \leq  N-1 )}{Z_N^{\go}( S_N =N  )}
\leq C \exp\Big( \tfrac{1}{8} \hat h N^{-\zeta}\Big) \,.
\end{equation}
This will conclude the proof of~\eqref{Rn=n} since
$\hat h N^{-\zeta} \to -\infty$ (recall $\zeta<0$).

We have $Z_N^{\go}( S_N=N) = 2^{-N} e^{\gb_N \Sigma_N^+ - h_N N} $.
Hence,  using that in the case $M_N^+=N-k$
then $M_N^- \geq - \frac12 k$ so $|\cR_N| \leq N-\frac12 k$,
we get that for $ 1\leq  k <\frac12 N $, after simplifications of the numerator and denominator,
\[
\frac{Z_N^{\go}( M_N^+= N-k )}{Z_N^{\go}( S_N=N)}   
\leq \exp\Big( \gb_N \sum_{i=N-k+1}^N \go_i  + \gb_N \Sigma_{\frac12 k}^* + \tfrac12 h_N k  \Big) 2^N \bP(M_N^+ = N-k )\, .
\]
Denoting $\tilde \Sigma_{k}^* := \Sigma_{\frac12 k}^*+\sup_{1\leq j \leq k} |\sum_{i=N-j+1}^N \go_i |$,
we then get that for $\ell\in \{1, \ldots, \log_2 N-1\}$
\begin{align*}
\frac{Z_N^{\go} \big( M_N^+\in [N-2^{\ell},N-2^{\ell-1}) \big)}{Z_N^{\go}( S_N=N)}   
\leq \exp\Big( \hat \gb N^{-\gamma} \tilde \Sigma_{2^{\ell}}^* + \hat h N^{-\zeta} 2^{\ell-2} \Big)
 2^{N+1} \bP( S_N \geq N - 2^{\ell} ),
\end{align*}
where we also used that $\bP(M_N^+ \geq N-2^{\ell}) = 2 \bP( S_N \geq N- 2^{\ell}) -1$ by the reflection principle. 
Now, since $N-S_N$ has a $\mathrm{Binom}(N,\frac12)$ distribution, we have 
\[
\bP( S_N \geq N- 2^{\ell}) 
= \sum_{i=0}^{2^\ell} 2^{-N} \binom{N}{i}\leq  2^{-N}  2^{\ell} \binom{N}{2^\ell} \,,
\]
for $\ell$ such that $2^\ell \leq \frac12N$.
Note that $2^{\ell} \binom{N}{2^{\ell}} \leq N^{2^{\ell}}$,
which is smaller than $\exp(  2^{\ell-3}  |\hat h| N^{-\zeta} )$ for $N$ large enough (uniformly for the range of $\ell$ considered).
We therefore end up with
\begin{equation*}
\label{tildeR5j}
\frac{Z_N^{\go} \big( M_N^+\in [N-2^{\ell},N-2^{\ell-1}) \big)}{Z_N^{\go}( S_N=N)} 
\leq \exp\Big( C(\hat \gb,\go) N^{-\gamma} (\log_2 N)^{2/\alpha} 2^{\ell/\alpha} +   \hat h N^{-\zeta} 2^{\ell-3}\Big) \, ,
\end{equation*}
where we have also bounded $\tilde \Sigma_{2^\ell}^*$ by a constant $c=c(\go)$ times  $\ell^{2/\alpha} 2^{\ell/\alpha}$,
analogously to Lemma~\ref{lem1}.
Now, uniformly for $\ell \in \{1,\ldots, \log_2N-1\}$,
we have
\[
\frac{N^{-\gamma} (\log_2 N)^{2/\alpha} 2^{\ell/\alpha}}{ N^{-\zeta} 2^{\ell}} 
\leq  (\log_2 N)^{2/\alpha}
\begin{cases}
N^{\zeta-\gamma} & \quad \text{ if } \alpha \in (1,2]\,,\\
N^{\frac{1-\alpha}{\alpha}+\zeta-\gamma} & \quad \text{ if } \alpha \in (0,1) \,.
\end{cases}
\]
This upper bound goes to $0$ as $N\to\infty$
since $\gamma>\zeta$ if $\alpha\in (1,2]$
and $\gamma >\zeta+\frac{1-\alpha}{\alpha}$ if $\alpha\in (0,1)$.
Hence, $\bbP$-a.s., for $N$ large enough we have
\begin{align*}
\frac{Z_N^{\go} \big(\frac12N < M_N^+\leq N-1 \big)}{Z_N^{\go}( S_N=N)} & = \sum_{\ell=1}^{\log_2 N -1} \frac{Z_N^{\go} \big( M_N^+\in [N-2^{\ell},N-2^{\ell-1}) \big)}{Z_N^{\go}( S_N=N)}  \\
& \leq \sum_{\ell=1}^{\log_2 N -1}  \exp\Big(  \hat h N^{-\zeta} 2^{\ell-4}\Big)  \leq C \exp\Big( \frac18 \hat h N^{-\zeta} \Big) \,,
\end{align*}
which gives~\eqref{tildeR6eq}.

For the proof of the last statement (\textit{i.e.}\ the analogous of~\eqref{eq:ballistic3}), notice that
\[
Z_{N}^{\go} (S_N =N) = 2^{-N} e^{\gb_N \Sigma_N^{+} -h_N N } \,,\qquad
Z_{N}^{\go} (S_N =-N) = 2^{-N} e^{\gb_N (\go_0+\Sigma_N^{-})  -h_N N } \,,
\]
so that
\begin{equation}
\frac{Z_{N}^{\go} (S_N =N)}{Z_{N}^{\go} (S_N =N) + Z_{N}^{\go} (S_N =-N)  } = \frac{ e^{\gb_N \Sigma_N^{+} }}{ e^{\gb_N \Sigma_N^{+}} + e^{\gb_N (\go_0+\Sigma_N^{-})} } \,.
\end{equation}
Then, we proceed as in the previous sections to get
\eqref{eq:ballistic3} with $\{S_N=N\}$ in place of~$B_N^{+,\gep}$.
Details are left to the reader.


\appendix
\section{Technical estimates}
\label{sec:app}

\subsection{Estimates on deviation probabilities}
\label{sec:appLD}

We present here some large deviation estimates
for the simple random walk that are needed throughout the paper.
Recall $M_N^- :=  \min_{0\leq n\leq N} S_n $ and 
$M_N^+ :=\max_{0\leq n\leq N} S_n $.

\subsubsection*{Stretching}

Our first lemma
deals with the super-diffusive case:
we estimate the probability that $M_N^+ \geq v N^{\xi}$
and $M_N^- \leq  -u N^{\xi}$ when $\xi\in (\frac12, 1)$, for $u,v\geq 0$.
The one-sided large deviation results are classical,  using e.g.\ explicit calculations for the simple random walk (see~\cite[Ch.~III.7]{Feller1}):  we get that if $\xi\in (\frac12,1)$
\[
\lim_{N\to\infty} - \frac{1}{N^{2\xi-1}} \log \bP\big( M_N^+ \geq v N^{\xi} \big)   = \lim_{N\to\infty} - \frac{1}{N^{2\xi-1}} \log \bP\big( S_N \geq v N^{\xi} \big)  = \frac12 v^2 \,.
\]
The case where both
the minimum and maximum are required to have large deviations is an easy extension of the result:
it follows from the reflection principle that
\[
\bP ( S_N \geq 2a+b ) \vee \bP ( S_N \geq a+2b ) \leq  \bP( M_N^- \leq -a ; M_N^+ \geq b ) \leq \bP ( S_N \geq 2a+b )
  + \bP ( S_N \geq a+2b )\,,
\]
so that $\log \bP( M_N^- \leq -u N^{\xi} ; M_N^+ \geq v N^{\xi} ) \sim \log \bP ( S_N \geq  (u\wedge v +u+v) N^{\xi} )$ as $N\to+\infty$.

\begin{lemma}
\label{lem:superdiffusive}
If $\xi \in(\frac12,1)$ then for any $u,v\geq 0 $ we have that
\begin{equation}
\label{ldp:1}
\lim_{N\to\infty} - \frac{1}{N^{2\xi-1}} \log \bP\Big( M_N^- \leq - u N^{\xi} ; M_N^+ \geq v N^{\xi} \Big) = I(u,v) := \frac{ 1 }{2} ( u\wedge v + u+v)^2\, .
\end{equation}
\end{lemma}

\noindent
As an easy consequence of this lemma, we  get that
for any $\gd>0$, for any $u,v\geq 0 $,
\begin{equation}\label{AE01}
\lim_{N\to\infty} - \frac{1}{N^{2\xi-1}} \log \bP\Big( M_N^- \in \big(-(u+\gd), -u\big] N^{\xi} ; M_N^+ \in \big[ v, v+\gd \big) N^{\xi} \Big) =  I(u,v) .
\end{equation}
Using again the reflection principle, it is also not difficult to show that a local version of~\eqref{AE01}
holds: omitting the integer parts for simplicity, we have
\begin{equation}
\label{localldp}
\limsup_{N\to\infty} - \frac{1}{N^{2\xi-1}} \log 
\inftwo{x\in [u,u+\delta)}{y\in [v,v+\delta)}  \bP\big( M_N^- = - xN^\xi ; M_N^+ = yN^\xi  \big)    = I( u,v )  \,.
\end{equation}

We now state a result that shows that the large deviation
is essentially realized by the event that the random walk moves ballistically to one end (whichever is the closest) and then ballistically to the other one.
For $u,v\geq 0$ with $u\neq v$, 
recall the definition~\eqref{eq:ballisticfunction}
of the function $b_{u,v}$ that goes
with constant speed from $0$ to the closest point between $-u$ and $v$ and then to the other one.
Recall also the notation~\eqref{eq:ball}:
\[
\cB_N^{\gep}(u,v) := \bigg\{ \sup_{t\in [0,1]} \Big| \frac{1}{N^{\xi}} S_{\lfloor tN \rfloor}  -b_{u,v}(t) \Big| \leq \gep \bigg\} \,.
\]
We then have the following result, which is a direct consequence of
Lemma~\ref{lem:superdiffusive}:
let $u,v\geq 0$ with $u\neq v$, then for any $\gep>0$,
\begin{equation}
\label{ldp:ballistic}
\liminf_{N\to\infty} - \frac{1}{N^{2\xi-1}} \log \bP\Big( M_N^- \leq - u N^{\xi} ; M_N^+ \geq v N^{\xi} ; \cB_N^{\gep} (u,v)^c \Big)
 \geq  I(u,v) +c_{\gep}(u,v)\, ,
\end{equation}
for some constant $c_{\gep} (u,v)>0$.

As a consequence,
for $\delta>0$ small enough such that 
$(u-\delta)^+>0$ or $(v-\delta)^+>0$
(where $x^+ := \max\{x,0\}$)
and $[u-\delta,u+\delta] \cap [v-\delta,v+\delta] =\emptyset$,
we also have 
\begin{multline*}
\liminf_{N\to\infty} - \frac{1}{N^{2\xi-1}} \log 
\suptwo{x\in [(u-\delta)^+,u+\delta]}{y\in [(v-\delta)^+,v+\delta]} 
\bP\big( M_N^- = -xN^\xi ; M_N^+ = yN^\xi ; \cB_N^{\gep} (u,v)^c \big)  \\
 \geq I\big( (u-\delta)^+,(v-\delta)^+\big) + c_{\gep} (u-\delta,v-\delta) \,.
\end{multline*}
Together with~\eqref{localldp}, we end up with 
\begin{equation}
\label{ldp:localballistic}
\begin{split}
\liminf_{N\to\infty} - & \frac{1}{N^{2\xi-1}}  \suptwo{x\in [(u-\delta)^+,u+\delta]}{y\in [(v-\delta)^+,v+\delta]}  \log  \frac{\bP\big( M_N^- = -xN^\xi ; M_N^+ = yN^\xi , \cB_N^{\gep} (u,v)^c \big)}{\bP\big( M_N^- = -xN^\xi ; M_N^+ = yN^\xi  \big)}  \\
& \geq  I\big( (u-\delta)^+,(v-\delta)^+\big) + c_{\gep} (u-\delta,v-\delta) - I\big( u+\delta,v+\delta \big)  =: c_{\gep,\gd}(u,v)
\end{split}
\end{equation}
where $c_{\gep,\gd}(u,v)>0$, provided that $\gd$
is small enough (how small depends on $\gep,u,v$).

\smallskip

Let us also state the large deviation result in the case $\xi=1$.
As above, it derives from the fact that $\log \bP( M_N^- \leq -u N ; M_N^+ \geq v N ) \sim \log \bP ( S_N \geq  (u\wedge v +u+v) N)$
as $N\to\infty$.

\begin{lemma}
\label{lem:superdiffusive2}
For any $u,v \geq 0 $, we have that 
\[
\lim_{N\to\infty} - \frac{1}{N} \log \bP\big( M_N^- \leq -u N; M_N^+ \geq v N \big) = \kappa \big( u\wedge v + u+v \big)\, ,
\]
where $\kappa:\bbR_+\to \bbR_+$ is the LDP rate function for the simple random walk, that is $\kappa(t) := \frac12 (1+t) \log (1+t) + \frac12 (1-t) \log(1-t)$ if $0\leq t\leq 1$ and $\kappa(t) = +\infty$ if $t>1$.
\end{lemma}

\noindent
Note that analogues of the ballisticity statements~\eqref{ldp:ballistic} and~\eqref{ldp:localballistic}
hold in the case~$\xi=1$.

\subsubsection*{Folding}
Our second lemma
deals with the sub-diffusive case:
we estimate the probability that $M_N^+ \leq v N^{\xi}$
and $M_N^- \geq -u N^{\xi}$ when $\xi\in (0,\frac12)$,
for $u,v\geq 0$.
The result follows from classical random walk calculations,
leading to explicit expressions of ruin probabilities (see
Eq.~(5.8) in \cite[Ch.~XIV]{Feller1});
one may refer to~\cite[Lem.~2.1]{CP09} and its proof 
for the following statement.

\begin{lemma}
\label{lem:subdiffusive}
If $\xi\in (0,\frac12)$, then for any $u,v\geq 0$ we have that
\begin{equation}
\lim_{N\to\infty} - \frac{1}{N^{1-2\xi}} \log \bP\Big( M_N^- \geq  -u N^{\xi} ; M_N^+ \leq v N^{\xi} \Big) = \frac{\pi^2}{2 (u+v)^2} \, .
\end{equation}
\end{lemma}

\noindent
As an easy consequence of this lemma, we  get  that
for any $\gd>0$ and any $u,v\geq 0 $,
\begin{equation}\label{AE02}
\lim_{N\to\infty} - \frac{1}{N^{1-2\xi}} \log \bP\Big( M_N^- \in [-u, -u+\delta) N^{\xi} ; M_N^+ \in (v-\gd, v] N^{\xi} \Big) = \frac{\pi^{2} }{2 (u+v)^2} \, .
\end{equation}

\subsection{Proof of Lemma~\ref{lem1}}
\label{app:lemma}

We start with the first part of the statement.
First of all, notice that the bound is trivial if $\ell \T^{-\ga} > 1$: we therefore assume that $\ell \T^{-\ga} \leq 1$.
Using Etemadi's inequality  (see~\cite[Thm.~2.2.5]{Billingsley}) we get that 
\begin{align*}
\bbP \big(\Sigma_\ell^* >\T \big) \le 3 \max\limits_{k\in\{1,\dots, \ell \}}\bbP \big(|\Sigma_k^+| > \tfrac16 \T \big) + 3\max\limits_{k\in\{1,\dots, \ell \}}\bbP \big(|\Sigma_k^-| >\tfrac16 \T\big)\, .
\end{align*}
We only bound  $\bbP \big(|\Sigma_k^+| >\frac16 \T \big)$, since the same bound will hold
for $\bbP \big(|\Sigma_k^-| > \frac16 \T \big)$.
The case $\alpha =2$ is a consequence of Kolmogorov's maximal inequality and the case $\ga \in (0,2)$ ($\ga\neq 1$) follows from the so-called \emph{big-jump} (or one-jump) behavior.
Let us give an easy proof: define $\bar \go_x := \go_x \ind_{\{|\go_x| \leq \T\}}$, so that 
\begin{align*}
\bbP \big(|\Sigma_k^+ |>\tfrac16 \T \big) 
&\leq \bbP \big( \exists  \,  0\leq x \leq k  \, , |\go_x| > \T  \big) 
 + \bbP \Big(  \Big|\sum_{x=0}^{k} \bar \go_x  \Big|>\tfrac16 \T\Big)  \\
& \leq   (k+1) \bbP \big( |\go_0| > \T\big) +  \frac{36}{\T^2}  \Big( (k +1) \bbE\big[ (\bar \go_0)^2 \big] + k(k+1) \bbE[ \bar \go_0]^2 \Big)\,,
\end{align*}
where we used a union bound for the first term and Markov's inequality
(applied to $(\sum_{x=0}^{k} \bar \go_x)^2$) for the second.
Now, the first term is clearly bounded by a constant times~$k\, \T^{-\ga}$ thanks to Assumption~\ref{assump1}.
For the second term, we use again Assumption~\ref{assump1}, 
to get that if $\ga \in (0,1)\cup(1,2)$,
$\bbE[ (\bar \go_0)^2] \leq  c \T^{2-\ga}$
and $\bbE[ \bar \go_0 ] \leq  c \T^{1-\ga}$ (when $\ga \in (1,2)$ we use for this last inequality that $\bbE[\go_0]=0$).
Therefore, we end up with the bound
\[
\bbP \big(|\Sigma_\ell^+ |>\tfrac16 \T \big) 
\leq c  \ell \T^{-\ga} + c \ell^2 \T^{-2\ga}  \leq 2 c  \ell \T^{-\ga} \, ,
\] 
where we have used that $\ell \T^{-\ga} \leq 1$ for the last inequality.

For the second part of the statement, notice that
$\bbP\big(\Sigma_{2^{k}}^*> k^{2/\alpha} 2^{k/\alpha} \big) \leq c k^{-2}$:
hence, by Borel--Cantelli, $\bbP$-a.s.\ there is a constant $C'=C'(\go)$ such that $\Sigma_{2^{k}}^* \leq C' k^{2/\alpha} 2^{k/\alpha}$ for all $k\geq 0$.
Since $\Sigma_\ell^*$ is monotone in $\ell$,
we get that $\bbP$-a.s.\ there is a constant $C=C(\go)$ such that $\Sigma_{\ell}^* \leq C (\log_2 \ell)^{2/\alpha} \ell^{1/\alpha}$
for all $\ell\geq 1$.
\qed

\subsection{Uniqueness of the maximizer: proof of Proposition~\ref{uniquemaximizer}}
\label{sec:unique}

We start with the following lemma, at the core of the proof.

\begin{lemma}\label{lem:unique2}
Let $(X_v^{(1)})_{v\geq0}$ and $(X_u^{(2)})_{u\geq0}$ be two independent $\ga$-stable L\'evy processes and $f: (\bbR_+)^2\to \bbR$ be any function. Denote $\cI_{a,b}^{c,d}:=[c,d]\times[a,b]$. Then for any two disjoint rectangles $\cI_{a,b}^{c,d}$ and $\cI_{a',b'}^{c',d'}$, we have that
\begin{equation*}
\bbP\bigg(\sup\limits_{(u,v)\in\cI_{a,b}^{c,d}}\Big\{X_v^{(1)}+X_u^{(2)}+f(u,v)\Big\}=\sup\limits_{(u',v')\in\cI_{a',b'}^{c',d'}}\Big\{X_{v'}^{(1)}+X_{u'}^{(2)}+f(u',v')\Big\}\bigg)=0 \,.
\end{equation*}
\end{lemma}

\begin{proof}
Let us assume that $a'>b$ (other cases are treated similarly).
Then, the difference of the supremums can be written as
$Z -Z' - (X_{a'}^{(1)} - X_{b}^{(1)})$,
with 
\[
Z= \sup_{(u,v)\in\cI_{a,b}^{c,d}}\big\{X_v^{(1)}-X_{b}^{(1)} +X_u^{(2)}+f(u,v)\big\} \,,\quad 
Z'= \sup_{(u',v')\in\cI_{a',b'}^{c',d'}}\big\{X_{v'}^{(1)}-X_{a'}^{(1)} +X_{u'}^{(2)}+f(u',v')\big\} \,.
\]
Note that $Z,Z'$ are independent of $X_{a'}^{(1)} - X_{b}^{(1)}$: we therefore get that
\[
\bbP\big( X_{a'}^{(1)} - X_{b}^{(1)} = Z'-Z\big) =0 \,,
\]
since $X_{a'}^{(1)} - X_{b}^{(1)}$ has no atom (it is $\alpha$-stable).
\end{proof}


\begin{proof}[Proof of Proposition \ref{uniquemaximizer}]
By Lemma \ref{lem:unique2} and subadditivity, we have
\begin{equation*}
\bbP\bigg(\exists~a,b,c,d,a',b',c',d'\in\bbQ, ~\text{s.t.}~\cI_{a,b}^{c,d}\cap\cI_{a',b'}^{c',d'}=\emptyset, \sup_{(u,v)\in\cI_{a,b}^{c,d}}Y_{u,v}=\sup_{(u',v')\in\cI_{a',b'}^{c',d'}}Y_{u',v'}\bigg)=0.
\end{equation*}
Since $\bbP$-a.s. $\sup_{u,v\geq 0}Y_{u,v}>0$ and $Y_{u,v}\to -\infty$ as $u\to+\infty$ or $v\to+\infty$, then for $\bbP$-a.s. realization $(X_v^{(1)})_{v\geq 0}$, $(X_u^{(2)})_{u\geq 0}$, there exists some rational constant $A=A(\omega)$, such that the supremum is achieved on the rectangle $[0,A]^2$. Then a sequential application of dichotomy yields the uniqueness of the maximizer.

Next, we show that
\begin{equation}\label{nosymmetry}
\bbP\Big(\exists~v\geq0,~\text{such that}~\argmax_{u,v\geq 0}Y_{u,v}=\{(v, v)\}\Big)=0
\end{equation}
Note that the above probability is bounded above by
\begin{equation}
\label{split}
\bbP\Big(\sup_{ u,v\geq 0 }Y_{u,v}=\sup_{ v\geq 0}Y_{v,v}\Big)\leq \bbP\Big(\sup_{0\leq v \leq u }Y_{u,v}=\sup_{v\geq 0}Y_{v,v}\Big)+\bbP\Big(\sup_{ 0\leq u \leq v } Y_{u,v}=\sup_{v\geq 0} Y_{v,v}\Big) \,.
\end{equation}
It suffices to show that both probabilities on the right-hand side of \eqref{split} are $0$.
We only deal with the first  term, the second one is identical.


For any $n,k\geq 0$, let $B_{n,k}=[\frac{k}{2^n},\frac{k+1}{2^n})\times[\frac{k}{2^n},\frac{k+1}{2^n})$. Note that $\bigcup_{k=1}^\infty B_{n,k}$ covers the line $u=v$ and that, as $n\to+\infty$,
 $\bigcup_{k=1}^\infty B_{n,k}\downarrow\{u=v\}$  and $\{v\leq u\}\backslash\bigcup_{k=1}^\infty B_{n,k}\uparrow\{v<u\}$. Hence, by Lemma \ref{lem:unique2} and the monotone convergence theorem, we have that
\begin{equation*}
\bbP\Big(\sup_{0\leq v<u}Y_{u,v}=\sup_{0\leq v}Y_{v,v}\Big)=0.
\end{equation*}
Furthermore, $\bbP$-a.s., for any $v\geq0$, we can take $u_n\downarrow v$, such that $Y_{u_n,v}\to Y_{v,v}$ by the right continuity of L\'evy processes. Hence,
\begin{equation*}
\bbP\Big(\sup_{0\leq v<u}Y_{u,v}=\sup_{0\leq v\leq u}Y_{u,v}\Big)=1 \,.
\end{equation*}
This shows that the upper bound in \eqref{split} is equal to $0$
and concludes the proof of \eqref{nosymmetry}.
\end{proof}

\subsection{Estimates on c\`ad-l\`ag paths at points of continuity}
\label{sec:cadlag}

\begin{lemma}\label{cadlag}
Let $(\alpha_N(t))_{N\geq1}$ be a sequence of c\`ad-l\`ag paths on $[0,\infty)$ that converges to a c\`ad-l\`ag path $\alpha(t)$ for the Skorokhod distance $d_0$ (\textit{cf.} \cite{JS03}). Suppose that $\alpha$ is continuous at $u$. 
Then for any $\epsilon,\delta>0$, there exists  $N_0=N_0(u,\gep,\delta)>0$ such that for all $N\geq N_0$, 
\begin{gather}
\label{continue1}
|\alpha_N(u)-\alpha(u)|<\gep,\\ 
\label{continue2}
\sup\limits_{v\in[ u, u+\delta]}|\alpha_N(v)-\alpha_N(u)|<\gep+\sup\limits_{ v\in[u, u+\delta+\gep]}|\alpha(v)-\alpha(u)|. 
\end{gather}
\end{lemma}

\begin{proof}
We start by proving \eqref{continue1}. Fix $\gep$ and let $\eta=\eta(\gep)>0$ be some number to be chosen below.
Since $\lim_{N\to\infty}d_0(\alpha_N,\alpha)=0$, for the above $\eta>0$, there exists a sequence of non-decreasing bijections $(\lambda_N(t))_{N\geq1}:[0,T]\to [0,T]$ with $T>u$ arbitrary (but fixed) and a large enough integer $N_0$, such that for all $N\geq N_0$,
\begin{equation}\label{d0distance}
\sup\limits_{t\in[0,T]}|\lambda_N(t)-t|<\eta\qquad\text{and}\qquad 
\sup\limits_{t\in[0,T]}|\alpha_N(\lambda_N(t))-\alpha(t)|<\eta.
\end{equation}
We have that
\begin{equation*}
|\alpha_N(u)-\alpha(u)|\leq|\alpha_N(u)-\alpha(\lambda_N(u))|+|\alpha(\lambda_N(u))-\alpha(u)|.
\end{equation*}
By \eqref{d0distance} we have $|\lambda_N(u)-u|<\eta$:  if we had fixed $\eta$ small enough, the second term above is smaller than $\gep/2$ by continuity while the first term is smaller than $\gep/2$ by~\eqref{d0distance}. Therefore \eqref{continue1} is proved.

\smallskip

We now prove \eqref{continue2}. 
Using \eqref{continue1} (with $\gep/3$ instead of $\gep$)
together with the triangular inequality, we get that  
$|\alpha_N(v)-\alpha_N(u)|\leq |\alpha_N(v)-\alpha(u)| +\gep/3$
for $N$ large enough,
so we only need to estimate $|\alpha_N(v)-\alpha(u)|$.
For any fixed $\delta>0$ such that $\delta+u\le T$,
using the sequence $\lambda_N$ defined above, we have that
\begin{equation*}
\begin{split}
\sup\limits_{u\leq v\leq u+\delta} & |\alpha_N(v)-\alpha(u)|=\sup\limits_{\lambda_N^{-1}(u)\leq v'\leq\lambda_N^{-1}(u+\delta)}|\alpha_N(\lambda_N(v'))-\alpha(u)|\\
\leq&\sup\limits_{\lambda_N^{-1}(u)\leq v'\leq\lambda_N^{-1}(u+\delta)}|\alpha_N(\lambda_N(v'))-\alpha(v')|+\sup\limits_{\lambda_N^{-1}(u)\leq v'\leq\lambda_N^{-1}(u+\delta)}|\alpha(v')-\alpha(u)|.
\end{split}
\end{equation*}
The first term above is smaller than $\epsilon/3$ by \eqref{d0distance}. For the second term, always by \eqref{d0distance}, we have that
\begin{equation*}
|\lambda_N^{-1}(u)-u|<\eta,\quad|\lambda_N^{-1}(u+\delta)-u|<\eta+\delta
\end{equation*}
and hence we need to bound
\begin{equation*}
\sup\limits_{u-\eta\leq v\leq u+\eta+\delta}|\alpha(v)-\alpha(u)|\leq\sup\limits_{u-\eta\leq v\leq u}|\alpha(v)-\alpha(u)|+\sup\limits_{u\leq v\leq u+\delta +\eta}|\alpha(v)-\alpha(u)| \,.
\end{equation*}
By continuity, the first term above can be made arbitrarily small by choosing $\eta$ small, so this proves~\eqref{continue2}.
\end{proof}

\section*{Acknowledgments} Q. Berger, N. Torri and R. Wei were supported by a public grant overseen by the French National Research Agency, ANR SWiWS (ANR-17-CE40-0032-02). N. Torri was also supported by the project Labex MME-DII (ANR11-LBX-0023-01). C.-H. Huang was supported by the Ministry of Science and Technology grant MOST 110-2115-M-004-001.
The authors would like to thank referees for their comments and suggestions, which helped improve both the  presentation and the results.

\bibliographystyle{abbrv}
\bibliography{references}

\end{document}